%% file: MSSQM.tex
\documentclass{amsart}
\input{Preambolo}

\begin{document}

%
%

\title{Moduli of semistable sheaves as quiver moduli}
\author{Andrea Maiorana}

\begin{abstract}
In the 1980s Dr{\'e}zet and Le Potier realized moduli spaces of Gieseker-semistable sheaves on $\mathbb{P}^2$ as what are now called quiver moduli spaces. We discuss how this construction can be understood using t-structures and exceptional collections on derived categories, and how it can be extended to a similar result on $\mathbb{P}^1\times\mathbb{P}^1$.
\end{abstract}

\maketitle


\input{Introduction}

\input{Prel}

\input{P1}

\input{Surfaces}

\input{P2}

\input{P1xP1}	

\section*{Acknowledgements}
I would like to thank my supervisor Ugo Bruzzo for introducing me to algebraic geometry and in particular to this topic, and for his help throughout this work. I am also very grateful to Matteo Caorsi, Alexander Kuznetsov, Emanuele Macr\`{i}, Markus Reineke, Francesco Sala, Francesco Sorrentino and Jacopo Stoppa for useful discussions on these and related subjects, and to the staff of the Math Department at the University of Pennsylvania for their kind hospitality during my stay.
This research was partly supported by GNSAGA-INdAM.


\bibliographystyle{alpha}

\bibliography{Bib-CatHomAlg,Bib-ReprThe,Bib-AlgGeo,Bib-Moduli,Bib-VarPap}

\vspace{0.5cm}
\textsc{SISSA, via Bonomea 265, 34136 Trieste, Italy}\par\nopagebreak
\textsc{INFN, Sezione di Trieste, Italy}\par\nopagebreak
  \textit{E-mail address:} \texttt{and.mai.math@gmail.com}

\end{document}

%% file: Preambolo.tex
\usepackage{amsmath}%
\usepackage{amsfonts}%
\usepackage{amssymb}%
\usepackage{amsthm}
\usepackage{units}
\usepackage{enumerate} 
\usepackage{upgreek}
\usepackage{graphicx}
\usepackage{multicol}
\usepackage[usenames,dvipsnames]{xcolor}

\usepackage{array} 
\usepackage{tabulary} 
\usepackage{multirow} 

\numberwithin{equation}{section}

\usepackage{tikz}
\usepackage{tikz-cd}
\usetikzlibrary{matrix,arrows,decorations.pathmorphing}
\usepackage{hyperref} 
\usepackage{geometry}
\colorlet{ColRecap}{Black} 
\colorlet{ColDefB}{Black} 
\colorlet{ColRmkB}{Black} 
\colorlet{ColExB}{Black} 
\colorlet{ColThmB}{Black} 
\colorlet{ColApp}{Black} 

\colorlet{ColThm}{Black}
\colorlet{ColDef}{Black}
\colorlet{ColRmk}{Black}
\colorlet{ColEx}{Black}
\colorlet{ColRcp}{Black}
\colorlet{ColPrel}{Black}
\colorlet{ColSA}{Black}
\colorlet{ColAss}{Black}

\DeclareMathAlphabet{\mathpzc}{OT1}{pzc}{m}{it}

	\newcommand{\Omhol}{\Upomega} 

\DeclareMathOperator{\Id}{Id}

\DeclareMathOperator{\im}{im}

\DeclareMathOperator{\coker}{coker}
\DeclareMathOperator{\rk}{rk}
\DeclareMathOperator{\point}{pt} 

\DeclareMathOperator{\Hom}{Hom}
\DeclareMathOperator{\End}{End}

\DeclareMathOperator{\Ob}{Ob}

\DeclareMathOperator{\Ant}{Ant}
\DeclareMathOperator{\Sym}{Sym}

\DeclareMathOperator{\op}{op}
\DeclareMathOperator{\diag}{diag}

\DeclareMathOperator{\ext}{ext}
\DeclareMathOperator{\Ext}{Ext}

\DeclareMathOperator{\Hilb}{Hilb}

\DeclareMathOperator{\Pic}{Pic}

\DeclareMathOperator{\num}{num}
\DeclareMathOperator{\M}{M}

\DeclareMathOperator{\GL}{GL}

\DeclareMathOperator{\G}{G}

\DeclareMathOperator{\Rep}{\mathbf{Rep}} 
\DeclareMathOperator{\Coh}{\mathbf{Coh}} 

\DeclareMathOperator{\dimvec}{\underline{dim}}
\DeclareMathOperator{\cc}{c} 
\DeclareMathOperator{\ch}{ch} 

\DeclareMathOperator{\lleq}{(\leq)\!\!<} 
\DeclareMathOperator{\ggeq}{(\geq)\!\!>}

\DeclareMathOperator{\prqG}{\preceq_\textup{G}}
\DeclareMathOperator{\suqG}{\succeq_\textup{G}}
\DeclareMathOperator{\prG}{\prec_\textup{G}}
\DeclareMathOperator{\suG}{\succ_\textup{G}}
\DeclareMathOperator{\pprG}{(\preceq_\textup{G})\!\!\prec_\textup{G}}

\newcommand*{\triple}[2][.1ex]{%
  \mathrel{\vcenter{\offinterlineskip%
  \hbox{$#2$}\vskip#1\hbox{$#2$}\vskip#1\hbox{$#2$}}}}

\newcommand*{\triplerightarrow}{\triple{\rightarrow}} 

\theoremstyle{plain}
\newtheorem{theorem}{Theorem}[section]
\newtheorem{lemma}[theorem]{Lemma}
\newtheorem{corollary}[theorem]{Corollary}
\newtheorem{proposition}[theorem]{Proposition}

\theoremstyle{definition}
\newtheorem{definition}[theorem]{Definition}
\newtheorem{remark}[theorem]{Remark}

\newtheorem{example}[theorem]{Example}
\newtheorem{examples}[theorem]{Examples}

%% file: Introduction.tex
%
%
\section{Introduction}
\subsection{Monads and moduli of sheaves}

A recurring theme in algebraic geometry is the study of moduli spaces, varieties whose points parameterize geometric objects of some kind. The first general construction of moduli spaces of vector bundles on a projective curve was given by Mumford using GIT, and then extended by Seshadri, Gieseker, Maruyama and Simpson among others to prove the existence, as projective schemes, of moduli spaces of semistable coherent sheaves on projective varieties of any dimension. We refer to the books \cite{LePo97Lect,HuyLeh10Geo} for comprehensive guides to the subject.


By the late 1970s, some people were studying an alternative and much more explicit way to construct moduli spaces of bundles over projective spaces: they were using \emph{monads}, namely complexes $\mathcal{A}\to\mathcal{B}\to\mathcal{C}$ of vector bundles with nonzero cohomology only at the middle term. This concept was first used by Horrocks \cite{Horr64Vector}; Barth \cite{Barth77Mod} showed that every stable bundle $\mathcal{E}$ of rank 2, degree 0 and $\cc_2=k$ on the complex projective plane $\mathbb{P}^2=\mathbb{P}_\mathbb{C}(Z)$ is isomorphic to the middle cohomology of a monad in which $\mathcal{A},\mathcal{B},\mathcal{C}$ are fixed bundles and the maps between them only depend on a certain \emph{Kronecker module} $\alpha\in\Hom_\mathbb{C}(\mathbb{C}^k\otimes Z^\vee,\mathbb{C}^k)$ constructed from $\mathcal{E}$. Moreover, this construction estabilishes a bijection between such bundles $\mathcal{E}$ up to isomorphism and elements of a subvariety $\tilde{M}\subset\Hom_\mathbb{C}(\mathbb{C}^k\otimes Z^\vee,\mathbb{C}^k)$ up to the action of $\GL_k(\mathbb{C})$. This means that we have a surjective morphism $\tilde{M}\to\M^\textup{st}$ identifying the moduli space $\M^\textup{st}$ of stable bundles with the given numerical invariants as a $\GL_k(\mathbb{C})$-quotient of $\tilde{M}$. By analyzing the variety $\tilde{M}$, Barth was then able to prove rationality and irreducibility of $\M^\textup{st}$. Then Barth and Hulek extended this construction first to all moduli spaces of rank 2 bundles \cite{BaHu78Mo,Hulek79Stable}, and then to moduli of bundles with any rank and zero degree \cite{Hulek80OnThe}. These works were also fundamental to find explicit constructions of instantons, or anti self-dual Yang-Mills connections \cite{ADHM78Constr,Don84Ins}.


These techniques were improved by Beilinson \cite{Beil78Coh}, whose description of the bounded derived category of coherent sheaves on projective spaces gave a systematic way to produce monads for semistable sheaves, as explained e.g.\ in \cite[Ch.\ 2, \S4]{OSS80Ve}. In this way, Dr{\'e}zet and Le Potier generalized in \cite{DreLeP85Fibr} the works of Barth and Hulek to all Gieseker-semistable torsion-free sheaves on $\mathbb{P}^2$. They showed that, after imposing an analogue of Gieseker semistability, ``Kronecker'' complexes
\begin{equation}\label{Eq: Kron cplx intro}
V_{-1}\otimes\mathcal{O}_{\mathbb{P}^2}(-1)\longrightarrow V_0\otimes\Omhol_{\mathbb{P}^2}^1(1)\longrightarrow V_1\otimes\mathcal{O}_{\mathbb{P}^2}
\end{equation}
are forced to be monads, and taking their middle cohomology gives Gieseker-semistable sheaves. Moreover, this gives a bijective correspondence between isomorphism classes of semistable Kronecker complexes and isomorphism classes of semistable torsion-free sheaves, having fixed a class $v\in K_0(\mathbb{P}^2)$. Thus the moduli space $\M^\textup{ss}_{\mathbb{P}^2}(v)$ of such sheaves is a quotient of the semistable locus $Y^\textup{ss}\subset Y$ in the affine variety $Y$ parameterizing Kronecker complexes by the action of $G_V:=\prod_i\GL_\mathbb{C}(V_i)$.

Now we can observe that Kronecker complexes can be seen as representations of the \emph{Beilinson quiver}
\begin{center}$B_3$: \begin{tikzcd}
-1\arrow[bend left=50]{r}\arrow{r}\arrow[bend right=50]{r} &0\arrow[bend left=50]{r}\arrow{r}\arrow[bend right=50]{r} &1
\end{tikzcd}\end{center}
constrained by some relations, which we encode with an ideal $J'\subset\mathbb{C}Q$ in the path algebra of $Q$, forcing the maps in Eq.\ \eqref{Eq: Kron cplx intro} to form a complex. In fact, after fixing the dimension vector $d^v:=(\dim_\mathbb{C}V_{-1},\dim_\mathbb{C}V_0,\dim_\mathbb{C}V_1)$ of the Kronecker complexes, the above notion of Gieseker-like stability coincides with the usual concept of $\theta_v$-stability for quiver representations, for some $\theta_v\in\mathbb{Z}^{\{-1,0,1\}}$. The latter was introduced by King \cite{King94Mod}, who also showed that moduli spaces of $d^v$-dimensional $\theta_v$-semistable representations, which we denote by $\M^\textup{ss}_{B_3,J',\theta_v}(d^v)$, can be constructed via GIT: the subset $Y^\textup{ss}\subset Y$ becomes the semistable locus of a linearization $\mathcal{L}_v$ of the action of $G_V$, so we recover $\M^\textup{ss}_{\mathbb{P}^2}(v)$ as the GIT quotient $Y^\textup{ss}/\!/_{\mathcal{L}_v}G_V=\M^\textup{ss}_{B_3,J',\theta_v}(d^v)$:
\begin{equation}\label{Eq: DLP's isom MP2-MB3}
\M^\textup{ss}_{\mathbb{P}^2}(v)\simeq\M^\textup{ss}_{B_3,J',\theta_v}(d^v)\,.
\end{equation}
This also proves the existence of $\M^\textup{ss}_{\mathbb{P}^2}(v)$ as a projective scheme independently from the general theory of Gieseker and Simpson.\footnote{In fact, another linearization of the action $G_V\curvearrowright Y$ providing the interpretation of $Y^\textup{ss}$ as a GIT-semistable locus was found in \cite{LePo94AProp} without referring to quiver moduli. Remarkably, already in \cite{Hulek80OnThe} it was observed that the Kronecker modules $f\in\Hom_\mathbb{C}(\mathbb{C}^k\otimes Z^\vee,\mathbb{C}^k)$ producing rank 2, degree 0 stable bundles can be characterized as GIT-stable points.} More recently, an analogous construction was carried out by Kuleshov in \cite{Kule97OnMod}, where, for certain choices of the numerical invariants, moduli spaces of sheaves on $\mathbb{P}^1\times\mathbb{P}^1$, Gieseker-stable with respect to the anticanonical polarization, were constructed as moduli of stable representations of the quivers
\begin{center}
{\footnotesize\begin{tikzcd}
&&\bullet&\\
&\bullet\arrow[shift left]{ur}\arrow[shift right,swap]{ur}\arrow[shift left]{dr}\arrow[shift right,swap]{dr} &&\\
&&\bullet&
\end{tikzcd}} and
{\footnotesize\begin{tikzcd}
&&\bullet\arrow[shift left]{dr}\arrow[shift right,swap]{dr} &\\
&&&\bullet\\
&&\bullet\arrow[shift left]{ur}\arrow[shift right,swap]{ur}&
\end{tikzcd}}.
\end{center}
In what follows we will see that this is a special case of a general construction working for all moduli of semistable torsion-free sheaves on $\mathbb{P}^1\times\mathbb{P}^1$. Finally, we mention that the techniques of \cite{DreLeP85Fibr} have been used in \cite{NevSta07Skly,FiGiIoKu16Int} to construct moduli spaces of semistable sheaves on noncommutative projective planes.

The above-mentioned work \cite{Beil78Coh} was followed by many years of research on the structure of the bounded derived category $D^b(X)$ of a projective variety $X$. In particular, a theory of \emph{exceptional collections} of objects of $D^b(X)$ was developed in the seminar \cite{Rudak90Heli} for this purpose (we give a short introduction to this subject in \S\ref{Exceptional collections}). By using this machinery, it is natural to interpret the abelian category of Kronecker complexes \eqref{Eq: Kron cplx intro} as the heart of a bounded t-structure (Def.\ \ref{Defn: t-struc}) on $D^b(\mathbb{P}^2)$ induced by an exceptional collection.

The main goal of this paper is to understand the constructions of the moduli spaces of sheaves via linear data mentioned in the previous section from this ``categorified" point of view, and to develope a machinery to produce isomorphisms like \eqref{Eq: DLP's isom MP2-MB3} in a systematic way when we are given an exceptional sequence with good properties.


Finally, we mention that there is a different way to relate moduli of sheaves and quiver moduli by using Bridgeland stability conditions \cite{Brid07Stab}: on a surface $X$ one can define a family of so-called \emph{geometric} stability conditions (these were introduced in \cite{ArcBer13Brid}), some of which are equivalent to Gieseker stability; on the other hand, a full strong exceptional sequence on $X$ induces \emph{algebraic} stability conditions, for which semistable objects are identified to semistable quiver representations. When $X=\mathbb{P}^2$, Ohkawa \cite{Ohka10Mod} constructed stability conditions which are both geometric and algebraic, obtaining as a consequence the explicit isomorphisms between moduli of sheaves and moduli of representations of the Beilinson quiver, as in our Theorems \ref{Thm: isom sh=qui on P2 (1)} and \ref{Thm: isom sh=qui on P2 (2)}. A similar analysis should in principle be possible also for $\mathbb{P}^1\times\mathbb{P}^1$, for which algebraic stability conditions were studied in \cite{ArcMil17Proj}. The main difference in our approach is essentially that we use a weaker notion of stability structure, which includes Gieseker stability both for sheaves and Kronecker complexes. Then we can directly jump from one moduli space to the other, instead of moving through the manifold of Bridgeland stability conditions. In this way the above-mentioned isomorphisms will be obtained with easy computations as examples of a general result.

\subsection{Outline of the paper}
In section \S\ref{Preliminaries} we briefly introduce the tools used in the rest of the paper: t-structures and exceptional collections on triangulated categories, stability structures, moduli spaces of sheaves and quiver representations.

Our aim is to construct some moduli spaces (or stacks) of semistable sheaves as quiver moduli spaces by using the above tools. The central idea is the following: take a smooth projective variety $X$ with a full exceptional sequence $\mathfrak{E}$ on $D^b(X)$ whose left dual ${}^\vee\!\mathfrak{E}$ is strong, and let $\M^\textup{ss}_{X,A}(v)$ be the moduli space of coherent sheaves on $X$ in a numerical class $v\in K_{\num}(X)$ that are Gieseker-semistable with respect to an ample divisor $A\subset X$. The sequence ${}^\vee\!\mathfrak{E}$ induces a triangulated equivalence $\Psi$ between $D^b(X)$ and the bounded derived category $D^b(Q;J)$ of finite-dimensional representations of a certain quiver $Q$, usually with relations $J$. The functor $\Psi$ induces a non-standard bounded t-structure on $D^b(X)$, whose heart $\mathcal{K}$ consists of certain \emph{Kronecker complexes} of sheaves, and Gieseker stability makes sense in a generalized way for objects of $\mathcal{K}$, after being reformulated in \S\ref{Moduli of sheaves} in terms of an alternating form on $K_{\num}(X)$.

The key observation is that in some cases the hearts $\mathcal{C},\mathcal{K}$ are somehow compatible with Gieseker stability, in the following sense: imposing Gieseker semistability forces the objects of suitable classes $v$ in the standard heart $\mathcal{C}\subset D^b(X)$ to be also semistable objects of $\mathcal{K}$, and the same is true with $\mathcal{C}$ and $\mathcal{K}$ exchanged. Moreover, semistable Kronecker complexes in the class $v$ are identified through $\Psi$ with $\theta_{\textup{G},v}$-semistable $d^v$-dimensional representations of $(Q,J)$, for some dimension vector $d^v$ and some (polynomial) weight $\theta_{\textup{G},v}$ (also depending on the polarization $A$) determined by the isomorphism of Grothendieck groups induced by $\Psi$. As this identification is compatible with the notions of families of semistable sheaves and semistable quiver representations, it implies that their moduli stacks can be identified through $\Psi$, and thus in particular the coarse moduli spaces $\M^\textup{ss}_{X,A}(v)$ and $\M^\textup{ss}_{Q,J,\theta_{\textup{G},v}}(d^v)$ are isomorphic.

The simplest example of this phenomenon is discussed in \S\ref{Semshequirep P1} for sheaves on the projective line $\mathbb{P}^1$: in this case, the heart $\mathcal{K}$ can be also obtained by tilting $\mathcal{C}$ using the slope-stability condition; this description is used to give in Corollary \ref{Cor: Bir-Gro Thm} an easy proof of Birkhoff-Grothendieck theorem (the well-known classification of coherent sheaves on $\mathbb{P}^1$) via quiver representations.

When $X$ is a surface, however, this simple argument fails as the hearts $\mathcal{C},\mathcal{K}$ are no longer related by a tilt. Nevertheless, in \S\ref{Gies-quiv stab surf} we show that the above-mentioned compatibility between $\mathcal{C},\mathcal{K}$ and Gieseker stability holds under some hypotheses on the sequence $\mathfrak{E}$ (namely, when $\mathfrak{E}$ is \emph{monad-friendly}, Def.\ \ref{Defn: mod-fr exc coll}): we define a subset $\tilde{\mathpzc{R}}_{A,\mathfrak{E}}\subset K_{\num}(X)$ depending on the ample divisor $A$ and the sequence $\mathfrak{E}$, and we prove in Corollary \ref{Cor: isom mod st/sp} that:

\begin{theorem}\label{Thm: Msh=Mqr (intro)}
For all $v\in\tilde{\mathpzc{R}}_{A,\mathfrak{E}}$ we have isomorphisms $\M^\textup{ss}_{X,A}(v)\simeq\M^\textup{ss}_{Q,J,\theta_{\G,v}}(d^v)$ and $\M^\textup{st}_{X,A}(v)\simeq\M^\textup{st}_{Q,J,\theta_{\G,v}}(d^v)$.
\end{theorem}

The assumptions on $\mathfrak{E}$ are easily seen to be satisfied by some well-known exceptional sequences on the projective plane $\mathbb{P}^2$ and the smooth quadric $\mathbb{P}^1\times\mathbb{P}^1$. The application of the Theorem to them is treated in Sections \ref{Appl to P2} and \ref{Appl to P1xP1}, where the only thing left is to determine the data $d^v,\theta_{\G,v}$, and $\tilde{\mathpzc{R}}_{A,\mathfrak{E}}$ for the given exceptional sequences. In both cases, the regions $\tilde{\mathpzc{R}}_{A,\mathfrak{E}}$ that we obtain are large enough to include, up to twisting $\mathfrak{E}$ by line bundles, any class $v$ of positive rank. So, for example, if we start from the exceptional sequence $\mathfrak{E}=(\mathcal{O}(-1),\Omhol^1(1),\mathcal{O})$ on $\mathbb{P}^2$, then we deduce the isomorphism \eqref{Eq: DLP's isom MP2-MB3} as a manifestation of an equivalence between abelian categories of Gieseker-semistable sheaves with fixed reduced Hilbert polynomial and King-semistable quiver representations. On $X=\mathbb{P}^1\times\mathbb{P}^1$ we will get a similar construction of $\M_{X,A}^\textup{ss}(v)$ for any polarization $A$ and any class $v$ of positive rank, providing thus a complete generalization of the result of \cite{Kule97OnMod}.  

Standard properties of the moduli spaces of sheaves, such as smoothness, dimensions and existence of universal sheaves will be quickly deduced using the theory of quiver moduli, and some concrete examples will be discussed.

%% file: Prel.tex

%
%
\section{Preliminaries}\label{Preliminaries}

In this section we give short accounts of the notions used throughout the paper, mostly in order to fix notation and conventions. The material is almost all standard, except for some concepts and notation in \S\ref{Stability structures}, \S\ref{Compat hearts stab} and \S\ref{Moduli of sheaves}, where we reformulate Gieseker stability in a way that makes sense for both sheaves and Kronecker complexes.

\subsection{t-structures}\label{t-structures}
Let $\mathcal{D}$ be a triangulated category. In this paragraph we recall the concept of t-structure on $\mathcal{D}$. All the details can be found in \cite{BeBeDe82Fais}.

\begin{definition}\label{Defn: t-struc}
A \emph{t-structure} on $\mathcal{D}$ consists of a pair $(\mathcal{D}^{\leq0},\mathcal{D}^{\geq0})$ of strictly full subcategories of $\mathcal{D}$ such that, writing $\mathcal{D}^{\leq\ell}:=\mathcal{D}^{\leq0}[-\ell]$ and $\mathcal{D}^{\geq\ell}:=\mathcal{D}^{\geq0}[-\ell]$ for $\ell\in\mathbb{Z}$, we have:
	\begin{enumerate}
	\item $\Hom_\mathcal{D}(X,Y)=0\ \forall X\in\mathcal{D}^{\leq0},\forall Y\in\mathcal{D}^{\geq1}$;
	\item $\mathcal{D}^{\leq0}\subset\mathcal{D}^{\leq0}[-1]$ and $\mathcal{D}^{\geq0}\subset\mathcal{D}^{\geq0}[1]$;
	\item for all $E\in\mathcal{D}$ there is a distinguished triangle $X\to E\to Y\to X[1]$ for some $X\in\mathcal{D}^{\leq0}$ and $Y\in\mathcal{D}^{\geq1}$.
	\end{enumerate}
The intersection $\mathcal{A}:=\mathcal{D}^{\leq0}\cap\mathcal{D}^{\geq0}$ is called the \emph{heart} of the t-structure. Finally, the t-structure $(\mathcal{D}^{\leq0},\mathcal{D}^{\geq0})$ is said to be \emph{bounded} when for all $E\in\mathcal{D}$ there exists $\ell\in\mathbb{N}$ such that $E\in\mathcal{D}^{\leq\ell}\cap\mathcal{D}^{\geq-\ell}$.
\end{definition}

It turns out that:
	\begin{enumerate}
	\item the heart $\mathcal{A}$ is an extension-closed abelian category, and when the t-structure is bounded the inclusion $\mathcal{A}\hookrightarrow\mathcal{D}$ gives an isomorphism $K_0(\mathcal{A})\cong K_0(\mathcal{D})$ between the Grothendieck groups;
	\item a sequence $0\to A_1\to A_2\to A_3\to 0$ in $\mathcal{A}$ is exact if and only if it can be completed to a distinguished triangle in $\mathcal{D}$;
	\item the inclusions $\mathcal{D}^{\leq\ell}\hookrightarrow \mathcal{D},\mathcal{D}^{\geq\ell}\hookrightarrow \mathcal{D}$ have a right adjoint $\tau_{\leq\ell}$ and a left adjoint $\tau_{\geq\ell}$ respectively, and the functors
	\begin{equation*}
	H_\mathcal{A}^\ell:=\tau_{\geq0}\circ\tau_{\leq0}[\ell]:\mathcal{D}\longrightarrow\mathcal{A}
	\end{equation*}	
	are cohomological.
	\end{enumerate}

\begin{examples}
The t-structures that we will see in this paper will arise in three ways:
	\begin{enumerate}
	\item if $\mathcal{A}$ is an abelian category, then the bounded derived category $D^b(\mathcal{A})$ has a \emph{standard} bounded t-structure whose heart is $\mathcal{A}$;
	\item if $\Psi:\mathcal{D}_1\to\mathcal{D}_2$ is an equivalence of triangulated categories, any t-structure on $\mathcal{D}_1$ induces a t-structure on $\mathcal{D}_2$ in the obvious way; in particular, when we are dealing with derived categories, the standard t-structures may be mapped to non-standard ones;
	\item (see e.g.\ \cite[\S1.1]{Poli07Const}) if $(\mathcal{T},\mathcal{F})$ is a torsion pair in the heart $\mathcal{A}\subset\mathcal{D}$ of a bounded t-structure $(\mathcal{D}^{\leq0},\mathcal{D}^{\geq0})$, then we can define a new t-structure $(\mathcal{D}^{\leq0}_t,\mathcal{D}^{\geq0}_t)$ on $\mathcal{D}$ via a \emph{tilt}, i.e.\ by taking
\begin{equation*}\label{Eq: def of tilt t-str}
\begin{array}{c}
\Ob(\mathcal{D}^{\leq0}_t):=\{X\in\mathcal{D}\ |\ H^0_\mathcal{A}(X)\in\mathcal{T},\ H_\mathcal{A}^\ell(X)=0\ \forall\ell>0\}\,,\\
\Ob(\mathcal{D}^{\geq0}_t):=\{X\in\mathcal{D}\ |\ H_\mathcal{A}^{-1}(X)\in\mathcal{F},\ H_\mathcal{A}^\ell(X)=0\ \forall\ell<-1\}\,;
\end{array}
\end{equation*}
	moreover, we have that
	\begin{equation}\label{Eq: rel tilt t-str}
	\mathcal{D}^{\leq0}_t\subset\mathcal{D}^{\leq0}\subset\mathcal{D}^{\leq1}_t\,,
	\end{equation}
	and in fact this property characterizes all the t-structures which are obtained by tilting $(\mathcal{D}^{\leq0},\mathcal{D}^{\geq0})$.
\end{enumerate}\end{examples}

\subsection{Stability structures}\label{Stability structures}
Let $\mathcal{A}$ be an abelian category. We will discuss some notions of (semi)stability for the objects of $\mathcal{A}$. These are mostly standard techniques, although in some cases we introduce some new notation and definitions for later convenience.

\subsubsection{Weights and alternating forms}\label{Weigh alt forms}
The simplest notion we will use is that of stability with respect to a \emph{weight}, that is a $\mathbb{Z}$-linear map $\nu:K_0(\mathcal{A})\to R$ with values in an ordered abelian group $(R,\leq)$ (which will typically be $\mathbb{Z}$, $\mathbb{R}$ or the polynomial ring $\mathbb{R}[t]$ with lexicographical order). This was introduced in \cite{King94Mod}.

\begin{definition}\label{Defn: stab wrt weight}
A nonzero object $A$ in $\mathcal{A}$ is said to be \emph{$\nu$-(semi)stable} if $\nu(A)=0$ and any strict subobject $0\neq B\subsetneq A$ satisfies $\nu(B)\ggeq0$. $A$ is \emph{$\nu$-polystable} if it is a direct sum of $\nu$-stable objects.
\end{definition}

It is easily checked that the semistable subobjects form (adding the zero object) a full abelian subcategory $\mathcal{S}_\nu\subset\mathcal{A}$ closed under extensions. The (semi)simple objects of this category are the $\nu$-(poly)stable objects.

Second, we take an alternating $\mathbb{Z}$-bilinear form $\sigma:K_0(\mathcal{A})\times K_0(\mathcal{A})\to R$.

\begin{definition}\label{Defn: stab alt form}
A nonzero object $A$ in $\mathcal{A}$ is said to be \emph{$\sigma$-(semi)stable} if any strict subobject $0\neq B\subsetneq A$ satisfies $\sigma(B,A)\lleq0$.
\end{definition}

If we fix a class $v\in K_0(\mathcal{A})$, then we can define a weight $\nu_v:=\sigma(v,\cdot)$ and observe that an object $A$ of class $v$ is $\nu_v$-(semi)stable if and only if it is $\sigma$-(semi)stable. 

We have given these two basic definitions of stability mostly for later notational convenience, and because they will be useful when used on different hearts in a triangulated category (see \S\ref{Stab triang cat}). These definitions are very general and do not have particularly interesting properties, mainly because they are too weak to induce an order on the objects of $\mathcal{A}$. However, with $\sigma$ we can order the subobjects of a fixed object $A$, and we will use the following definition:

\begin{definition}\label{Defn: sigma-max sub}
Let $A$ be a nonzero object. A nonzero subobject $S\subset A$ is said to be \emph{$\sigma$-maximal} if for any subobject $S'\subset A$ we have $\sigma(S',A)\leq\sigma(S,A)$.
\end{definition}

\subsubsection{Polynomial stabilities}
Take a $\mathbb{Z}$-linear map $P:K_0(\mathcal{A})\to\mathbb{R}[t]$ and define an alternating form
$\sigma_P:K_0(\mathcal{A})\times K_0(\mathcal{A})\to\mathbb{R}$ by
\begin{equation}\label{Eq: sigmaP}
\sigma_P(v,w):=P_vP'_w-P_wP'_v\,,
\end{equation}
where $P_v'(t):=\frac{\mathrm{d}}{\mathrm{d}t}P_v(t)$.

\begin{definition}\label{Defn: P-(semi)st}
A nonzero object $A$ in $\mathcal{A}$ is said to be \emph{$P$-(semi)stable} if it is $\sigma_P$-(semi)stable, that is if $P_vP'_w-P_wP'_v\leq0$ for any $0\neq B\subsetneq A$.
\end{definition}

As usual, polynomials are ordered lexicographically. This definition does not assume anything on the map $P$, but it turns out to be much more interesting when $P$ maps the classes of nonzero objects into the set $\mathbb{R}[t]_+\subset\mathbb{R}[t]$ of polynomials with positive leading coefficient:

\begin{definition}
We call $P$ a \emph{polynomial stability} on $\mathcal{A}$ if $P_A\in\mathbb{R}[t]_+$ for any nonzero object $A$ of $\mathcal{A}$.
\end{definition}

Indeed, we can give $\mathbb{R}[t]_+$ an alternative total preorder $\prqG$ by setting
\begin{equation}\label{Eq: def Gies preor}
p\prqG q \iff pq'-p'q\leq 0
\end{equation}
for $p,q\in\mathbb{R}[t]_+$; we also write $p\equiv_\textup{G}q$ when $p\prqG q$ and $q\prqG p$. We have the following equivalent characterizations of $\preceq_\textup{G}$, which show that it is indeed a preorder (that is, a total, reflexive and transitive relation) and that it coincides with the preorder considered in \cite[\S2]{Ruda97Stab}:

\begin{lemma}\label{Lem: equiv char Gies ord}
Take two polynomials $p,q\in\mathbb{R}[t]_+$ and write them as $p(t)=\sum_{i=0}^{\deg p}a_it^i$ and $q(t)=\sum_{j=0}^{\deg q}b_jt^j$. Then the following statements are equivalent:
	\begin{itemize}
	\item[(i)] $p\pprG q$;
	\item[(ii)] we have
	\begin{equation*}\label{Eq: Gies ord equiv cond}
	\deg p>\deg q\textup{\ \ \ or\ \ \ }\left\{\begin{matrix}
\deg p=\deg q=:d\\ 
\frac{p(t)}{a_d}\lleq\frac{q(t)}{b_d}
\end{matrix}\right.\,.
	\end{equation*}
	\end{itemize}
If $\deg p\leq\deg q$, then they are also equivalent to
	\begin{itemize}
	\item[(iii)] $b_{\deg p}\, p(t)\lleq a_{\deg p}\,q(t)$.
	\end{itemize}
Moreover, we have $p\equiv_\textup{G}q$ if and only if $p$ and $q$ are proportional.
\end{lemma}

\begin{remark}\label{Rmk: pq'=p'q impl prop}
Notice that the last statement of the Lemma extends to any nonzero polynomials $p,q\in\mathbb{R}[t]$: if $pq'-p'q=0$, then $p$ and $q$ are proportional. Indeed, we can replace $p$ and $q$ by their opposites if necessary and then apply the Lemma to them.
\end{remark}
\begin{remark}\label{Rmk: simplif polyn stab}
Stability with respect to a (not necessarily positive) polynomial function $P=\sum_ia_it^i:K_0(\mathcal{A})\to\mathbb{R}[t]_{\leq d}$ taking values in polynomials of degree at most $d$ is unchanged if we replace each coefficient $a_i$ by a constant real combination $\sum_{j\geq i}M_{ij}a_i$ with $M_{ii}>0$. Indeed, take $p=\sum_ia_it^i,q=\sum_ib_it^i\in\mathbb{R}[t]_{\leq d}$ and an upper triangular matrix $(M_{ij})_{i,j=0}^d$ with positive diagonal entries, and define $\tilde{p}=\sum_i\tilde{a}_it^i,\tilde{q}=\sum_i\tilde{b}_it^i\in\mathbb{R}[t]_{\leq d}$ by $\tilde{a}_i=\sum_{j\geq i}M_{ij}a_j$ and $\tilde{b}_i=\sum_{j\geq i}M_{ij}b_j$. Then $pq'-p'q\leq0$ if and only if $\tilde{p}\tilde{q}'-\tilde{p}'\tilde{q}\leq0$. This is easily seen by using the previous lemma, after reducing to the case in which $p>0$ and $q>0$.
\end{remark}

The fact that a polynomial stability $P$ orders the nonzero objects of $\mathcal{A}$ has some interesting consequences \cite{Ruda97Stab}: first, when $A,B$ are $P$-semistable objects with $P_A\suG P_B$, then $\Hom_\mathcal{A}(A,B)=0$; second, the $P$-semistable objects $A$ such that $P_A$ is proportional to a fixed $p\in\mathbb{R}[t]_+$ form, after adding the zero objects, a full abelian subcategory
\begin{equation*}
\mathcal{S}_P(p)\subset\mathcal{A}\,,
\end{equation*}
closed under extensions. Finally, $P$ can be used in many cases to induce canonical filtrations of objects of $\mathcal{A}$ by semistable ones:

\begin{definition}\label{Defn: HN filtr P}
Let $P:K_0(\mathcal{A})\to\mathbb{R}[t]$ be a polynomial stability, and take a nonzero object $A$ in $\mathcal{A}$. We call \emph{Harder-Narasimhan (HN) filtration} of $A$ a filtration
\begin{equation*}
0=A_0\subsetneq A_1\subsetneq\cdots\subsetneq A_\ell=A
\end{equation*}
such that every quotient $A_i/A_{i-1}$ is $P$-semistable and $P_{A_1}\suG P_{A_2/A_1}\suG\cdots\suG P_{A/A_{\ell-1}}$. $P$ is said to have the \emph{HN property} when each nonzero object admits a HN filtration.
\end{definition}

HN filtrations, when they exist, are unique. \cite{Ruda97Stab} also gives sufficient conditions guaranteeing that $P$ has the HN property: for example, this is the case if $\mathcal{A}$ is Noetherian and $P$ only takes values in numerical polynomials. Moreover, in this case the categories $\mathcal{S}_P(p)$ are of finite length.

\subsubsection{Central charges}
Now consider a $\mathbb{Z}$-linear map $Z:K_0(\mathcal{A})\to\mathbb{C}$, and construct from it a bilinear form $\sigma_Z:K_0(\mathcal{A})\times K_0(\mathcal{A})\to\mathbb{R}$ by
\begin{equation}\label{Eq: sigmaZ}
\sigma_Z(v,w):=-\Re Z(v)\Im Z(w)+\Re Z(w)\Im Z(v)\,.
\end{equation}

\begin{definition}
A nonzero object $A$ in $\mathcal{A}$ is said to be \emph{$Z$-(semi)stable} if it is $\sigma_Z$-(semi)stable.
\end{definition}

Equivalently, we are asking that $A$ is (semi)stable with respect to the polynomial map $P_v(t):=t\Im Z(v)-\Re Z(v)$. Again, this notion of stability is most useful when the positive cone in $K_0(\mathcal{A})$ is mapped to a proper subcone of $\mathbb{C}$, as this allows to order the objects of $\mathcal{A}$ according to the phases of their images under $Z$. Commonly, one requires that the positive cone is mapped by $Z$ inside the semi-closed upper half-plane $\mathbb{H}\cup\mathbb{R}_{<0}$, which is as saying that $P$ is a polynomial stability:

\begin{definition}\cite{Brid07Stab}
$Z$ is called a \emph{stability function}, or \emph{central charge}, when for any nonzero object $A$ we have $\Im Z(A)\geq 0$, and we have $\Im Z(A)=0$ only if $\Re Z(A)<0$. $Z$ has the \emph{HN property} if the polynomial stability $P$ has.
\end{definition}

In this case we denote by $\phi_Z(A):=\arg Z(A)/\pi\in(0,1]$ the \emph{phase} of a nonzero object $A$. Note that now we have $P_A\prqG P_B$ if and only if $\phi_Z(A)\leq\phi_Z(B)$, and similarly if we replace phases by \emph{slopes} $\mu_Z(A):=-\cot\phi_Z(A)=-\Re Z(A)/\Im Z(A)\in(-\infty,+\infty]$.\\
Thus, objects are ordered by their slopes, and the above definitions of stability and HN filtrations take now the usual forms. We write $\mathcal{S}_Z(\phi)\subset\mathcal{A}$ for the abelian subcategory of $Z$-semistable objects of fixed phase $\phi$.


%
\subsubsection{Stability in triangulated categories}\label{Stab triang cat}
Finally, we extend the previous notions of stability to a triangulated category $\mathcal{D}$: by a stability structure (of any of the above types) on $\mathcal{D}$ we mean a stability structure on the heart $\mathcal{A}\subset\mathcal{D}$ of a bounded t-structure. Notice that fixing e.g.\ a bilinear form $\sigma:K_0(\mathcal{D})\times K_0(\mathcal{D})\to R$ gives a stability structure on any heart in $\mathcal{D}$; when an object $D\in\mathcal{D}$ lies in different hearts, it is necessary to specify with respect to which of them we are considering it being (semi)stable or not (being a subobject is a notion that depends on the heart).

A situation in which a stability structure behaves well when changing the heart is when a t-structure is built via a tilt with respect to a stability function:\footnote{The same argument works using a polynomial stability instead of $Z$, but we will not need this level of generality.} let $Z:K_0(\mathcal{A})\to\mathbb{C}$ be a stability function with the HN property on a heart $\mathcal{A}$, and take $\phi\in(0,1]$. Then, as in \cite[Lemma 6.1]{Brid08Stab}, $Z$ induces a torsion pair $(\mathcal{T}^Z_{\geq\phi},\mathcal{F}^Z_{<\phi})$ in $\mathcal{A}$, given by
\begin{equation}\label{Eqn: tors pair ind by Z}
\begin{array}{c}
\Ob(\mathcal{T}^Z_{\geq\phi})=\{A\in\mathcal{A}\ \textup{with all the HN phases}\geq\phi\}\,,\\
\Ob(\mathcal{F}^Z_{<\phi})=\{A\in\mathcal{A}\ \textup{with all the HN phases}<\phi\}\,,
\end{array}
\end{equation}
where by \emph{HN phases} we mean the phases $\phi_Z(A_i/A_{i-1})$ of the quotients in the HN filtration of $A$. Thus we can consider $Z$-stability with respect to either $\mathcal{A}$ or the heart $\mathcal{A}^\#$ of the tilted t-structure (although $Z$ does not map the positive cone of $K_0(\mathcal{A}^\#)$ in the upper half plane, so typically one rotates $Z$ accordingly; we do not perform this operation), and for objects in the intersection $\mathcal{A}\cap\mathcal{A}^\#=\mathcal{T}^Z_{\geq\phi}$ the two notions coincide.

\subsection{Compatibility of hearts under a stability structure}\label{Compat hearts stab}
Take a triangulated category $\mathcal{D}$, an alternating $\mathbb{Z}$-bilinear form $\sigma:K_0(\mathcal{D})\times K_0(\mathcal{D})\to\mathbb{R}[t]$, the hearts $\mathcal{A},\mathcal{B}\subset\mathcal{D}$ of two bounded t-structures, and $v\in K_0(\mathcal{D})$.

To relate $\sigma$-(semi)stable objects in the two hearts, we would like the following compatibility conditions to be satisfied. First, we want the $\sigma$-semistable objects in one heart to belong also to the other:
\begin{itemize}
\item[(C1)] For any object $D\in\mathcal{D}$ belonging to the class $v$, the following conditions hold:
	\begin{itemize}
	\item[(a)] if $D$ is a $\sigma$-semistable object of $\mathcal{A}$, then it also belongs to $\mathcal{B}$;
	\item[(b)] if $D$ is a $\sigma$-semistable object of $\mathcal{B}$, then it also belongs to $\mathcal{A}$.
	\end{itemize}
\end{itemize}
Second, we want that $\sigma$-(semi)stability can be equivalently checked in one heart or the other:
\begin{itemize}
\item[(C2)] For any object $D\in\mathcal{A}\cap\mathcal{B}$ belonging to the class $v$, we have that $D$ is $\sigma$-(semi)stable in $\mathcal{A}$ if and only if it is $\sigma$-(semi)stable in $\mathcal{B}$.
\end{itemize}

\begin{definition}\label{Defn: (sigma,v)-compat hearts}
We say that the hearts $\mathcal{A}$ and $\mathcal{B}$ are \emph{$(\sigma,v)$-compatible} when they satisfy the above conditions (C1) and (C2).
\end{definition}

\begin{remark}
Denote for now by $\mathcal{A}^\textup{st}_{\sigma,v}\subset\mathcal{A}^\textup{ss}_{\sigma,v}\subset\mathcal{A}$ the subcategories of $\sigma$-stable and $\sigma$-semistable objects in $\mathcal{A}$ of class $v$, and similarly with $\mathcal{A}$ replaced by $\mathcal{B}$. Then $\mathcal{A}$ and $\mathcal{B}$ are $(\sigma,v)$-compatible if and only if
\begin{equation*}
\mathcal{A}^\textup{ss}_{\sigma,v}=\mathcal{B}^\textup{ss}_{\sigma,v}\quad\textup{and}\quad
\mathcal{A}^\textup{st}_{\sigma,v}=\mathcal{B}^\textup{st}_{\sigma,v}\,.
\end{equation*}
In particular, notice that $(\sigma,v)$-compatibility is an equivalence relation between hearts of bounded t-structures in $\mathcal{D}$.
\end{remark}

Consider for example the case of the alternating form $\sigma_Z$ induced by a map $Z:K_0(\mathcal{D})\to\mathbb{C}$ as in Eq.\ \eqref{Eq: sigmaZ}:

\begin{lemma}\label{Lem: tilted hearts compatible}
Suppose that $Z$ is a central charge polynomial stability with the HN property on the heart $\mathcal{A}$, and let $\mathcal{A}^\#$ be the tilted heart at the torsion pair \eqref{Eqn: tors pair ind by Z} for some some $\phi\in(0,1]$. Then $\mathcal{A}$ and $\mathcal{A}_t$ are $(\sigma_P,v)$-compatible for any $v\in K_0(\mathcal{D})$ such that $\phi_Z(v)\in[\phi,1]$.
\end{lemma}


\begin{remark}\label{Rmk: mod stack comp herts coinc}
Typically (e.g.\ this is the case when $\mathcal{D}=D^b(X)$, as discussed in \S\ref{Families obj der cat}) there is some notion of families of objects in the hearts of $\mathcal{D}$, so that we have moduli stacks (or even moduli spaces) $\mathfrak{M}_{\mathcal{A},\sigma}(v)$, $\mathfrak{M}_{\mathcal{B},\sigma}(v)$ of $\sigma$-(semi)stable objects in $\mathcal{A},\mathcal{B}$ respectively, and belonging to the class $v$. Then we have $\mathfrak{M}_{\mathcal{A},\sigma}(v)=\mathfrak{M}_{\mathcal{B},\sigma}(v)$ if $\mathcal{A}$ and $\mathcal{B}$ are $(\sigma,v)$-compatible.
\end{remark}

\subsection{Quiver moduli}\label{Quiver moduli}
Here we briefly recall the main aspects of the geometric representation theory of quivers introduced in \cite{King94Mod}: this summary is essentially based on that paper and on the notes \cite{Rein08Mod}.

Our notation is as follows: a \emph{quiver} $Q=(I,\Omega)$, consists of a set $I$ of vertices, a collection $\Omega$ of arrows between them and source and target maps $s,t:\Omega\to I$. We only consider finite and acyclic (that is, without oriented loops) quivers. We denote by $\Rep^\textup{fd}_\mathbb{C}(Q)$ the abelian category of finite-dimensional complex representations of $Q$, which are identified with \emph{left} modules of finite dimension over the path algebra $\mathbb{C}Q$ (we adopt the convention in which arrows are composed like functions). When a $I$-graded $\mathbb{C}$-vector space $V=\oplus_{i\in I}V_i$ is fixed, we write
\begin{equation*}
R_V:=\oplus_{h\in\Omega}\Hom_\mathbb{C}(V_{s(h)},V_{t(h)})
\end{equation*}
for the vector space of representations of $Q$ on $V$, whose elements are collections $f=\{f_h\}_{h\in\Omega}$ of linear maps. The isomorphism classes of such representations are the orbits of the obvious action of $G_V:=\prod_{i\in I}\GL_\mathbb{C}(V_i)$ on $R_V$. The subgroup $\Delta:=\{(\lambda\Id_i)_{i\in I}\,,\ \lambda\in\mathbb{C}^\times\}$ acts trivially, so the action descends to $PG_V:=G_V/\Delta$.

Under our assumptions of finiteness and acyclicity of $Q$, the category $\Rep^\textup{fd}_\mathbb{C}(Q)$ is of finite length and hereditary, and its simple objects are the representations $S(i)$ with $\mathbb{C}$ at the $i$th vertex and zeroes elsewhere; in particular, the classes of these objects form a basis of the Grothendieck group $K_0(Q):=K_0(\Rep^\textup{fd}_\mathbb{C}(Q))$, which is then identified with the lattice $\mathbb{Z}^I$ by taking the dimension vector $\dimvec V=(\dim_\mathbb{C}V_i)_{i\in I}$ of a representation $V$. Hence, giving a $\mathbb{Z}$-valued weight as in \S\ref{Weigh alt forms} on the category $\Rep^\textup{fd}_\mathbb{C}(Q)$ is the same as giving an array $\theta\in\mathbb{Z}^I$: this defines $\nu_\theta:\mathbb{Z}^I\cong K_0(\Rep^\textup{fd}_\mathbb{C}(Q))\to\mathbb{Z}$ by
\begin{equation}\label{Eq: weight quiver theta}
\nu_\theta(d):=\theta\cdot d=\sum_{i\in I}\theta^id_i\,.
\end{equation}
Similarly, we can consider $\mathbb{R}$-valued or even polynomial-valued arrays $\theta$ and weights $\nu_\theta$.

\begin{definition}\label{Defn: theta-stab qui rep}
Fix $\theta\in\mathbb{R}^I$ or $\theta\in\mathbb{R}[t]^I$. We call a representation $(V,f)$ \emph{$\theta$-(semi)stable} when it is $\nu_\theta$-(semi)stable according to Def.\ \ref{Defn: stab wrt weight}, namely when we have $\theta\cdot\dimvec V=0$ and $\theta\cdot\dimvec W\ggeq 0$ for any subrepresentation $0\neq W\subsetneq V$.\footnote{This is the definition used in \cite{King94Mod}; the author of \cite{Rein08Mod} uses instead the opposite convention, and in fact he defines a slope $\mu_\theta$, to order the representations and to have a HN property.}\\
We denote by $R_{V,\theta}^\textup{st}\subset R_{V,\theta}^\textup{ss}\subset R_V$ the subsets of $\theta$-stable and $\theta$-semistable representations on $V$, and by
\begin{equation}\label{Eq: cat semis quiv rep}
\mathcal{S}_\theta\subset\Rep_\mathbb{C}^\textup{fd}(Q)
\end{equation}
the subcategory of $\theta$-semistable representations of $Q$ (of any dimension), including the zero representation.
\end{definition}
\begin{remark}\label{Rmk: coprime dim vec}
Suppose that a dimension vector $d\in\mathbb{N}^I$ such that $\theta\cdot d=0$ is \emph{$\theta$-coprime}, meaning that we have $\theta\cdot d'\neq0$ for any $0\neq d'<d$ (which means that $0\neq d'\neq d$ and $d'_i\leq d_i$ for all $i\in I$). Then a $d$-dimensional representation cannot be strictly $\theta$-semistable. Note also that if $d$ is $\theta$-coprime, then it is a primitive vector of $\mathbb{Z}^I$. Conversely, if $d$ is a primitive vector, $\theta\cdot d=0$ and the components $\theta^0,...,\theta^k\in\mathbb{R}^I$ of $\theta$ span a subspace of dimension at least equal to $\#I-1$, then $d$ is $\theta$-coprime.
\end{remark}

Now we fix a $I$-graded vector space $V$ of dimension vector $d$, and we want to consider a quotient space of $R_V$ by the reductive group $G_V$ to parameterize geometrically the representations of $Q$. The set-theoretical quotient often is not a variety, while the classical invariant theory quotient $R_V/\!/PG_V$ is just a point, so we consider a GIT quotient with respect to a linearization: given an integral array $\theta\in\mathbb{Z}^I$ such that $d\cdot\theta=0$, we construct a character $\chi_\theta:PG_V\to\mathbb{C}^\times$ by $\chi_\theta(g):=\prod_{i\in I}(\det g_i)^{\theta_i}$; this induces a linearization of the trivial line bundle on $R_V$ and then a notion of (semi)stability which is exactly the same as $\theta$-(semi)stability, and GIT quotients which we denote by
\begin{equation*}
\M^\textup{ss}_{Q,\theta}(d):=R_{V,\theta}^\textup{ss}/\!/\!_{\chi_\theta}PG_V\,,\quad\quad 
\M^\textup{st}_{Q,\theta}(d):=R_{V,\theta}^\textup{st}/\!/\!_{\chi_\theta}PG_V\,,
\end{equation*}
the latter being the stable quotient. These varieties corepresent the quotient stacks
\begin{equation*}
\mathfrak{M}^\textup{ss}_{Q,\theta}(d):=[R_{V,\theta}^\textup{ss}/G_V]\,,\quad\quad 
\mathfrak{M}^\textup{st}_{Q,\theta}(d):=[R_{V,\theta}^\textup{st}/G_V]\,,
\end{equation*}
which can be also defined as moduli stacks of families of representations (which are defined e.g.\ in \cite[\S5]{King94Mod}). This is why it is meaningful to call $\M^\textup{ss}_{Q,\theta}(d)$ and $\M^\textup{st}_{Q,\theta}(d)$ the \emph{moduli spaces} of semistable and stable representations.

\begin{remark}\label{Rmk: prop quiv mod}
Here we list the main properties of these moduli spaces:
	\begin{enumerate}
	\item $\M^\textup{ss}_{Q,\theta}(d)$ is a projective variety, and $\M^\textup{st}_{Q,\theta}(d)$ is an open set in it;
	\item $\M^\textup{st}_{Q,\theta}(d)$ is smooth of dimension
	\begin{equation*}
	\dim\M^\textup{st}_{Q,\theta}(d)=\sum_{h\in\Omega}d_{s(h)}d_{t(h)}-\sum_{i\in I}d_i^2+1=1-\chi(d,d)\,,
	\end{equation*}
	where $\chi$ is the Euler form on $K_0(\Rep_\mathbb{C}(Q))\simeq\mathbb{Z}^I$; the stacks $\mathfrak{M}^\textup{ss}_{Q,\theta}(d)$ and $\mathfrak{M}^\textup{st}_{Q,\theta}(d)$ are smooth of dimension $-\chi(d,d)$;
	\item $\M^\textup{ss}_{Q,\theta}(d)$ is a coarse moduli space for S-equivalence classes\footnote{Two representations are \emph{S-equivalent} when the closures of their $PG_V$-orbits in $R_V^\textup{ss}$ intersect, or equivalently when they have the same composition factors as elements of the subcategory $\mathcal{S}_{\nu_\theta}\subset\Rep^\textup{fd}_\mathbb{C}(Q)$ of $\theta$-semistable representations.} of $\theta$-semistable representations on $V$, while the points of $\M^\textup{st}_{Q,\theta}(d)$ correspond to isomorphism classes of $\theta$-stable representations;
	\item \cite[\S5.4]{Rein08Mod} if $d$ is primitive (that is, $\gcd(d_i)_{i\in I}=1$), then $\M^\textup{st}_{Q,\theta}(d)$ admits a universal family.
	\item if $d$ is $\theta$-coprime (see Remark \ref{Rmk: coprime dim vec}), then there are no strictly semistable representations, so $\M^\textup{ss}_{Q,\theta}(d)=\M^\textup{st}_{Q,\theta}(d)$ is smooth and projective, and it admits a universal family.
	\end{enumerate}
\end{remark}

After we have fixed a dimension vector $d=\dimvec V$, we can partition the hyperplane $d^\perp\subset\mathbb{R}^I$ into finitely many locally closed subsets where different $\theta$ give the same $\theta$-(semi)stable representations: call $\theta_1,\theta_2\in d^\perp$ \emph{numerically equivalent} when for any $d'\leq d$ (which means that $d'_i\leq d_i$ for all $i\in I$) $\theta_1\cdot d'$ and $\theta_1\cdot d'$ have the same sign ($\pm 1$ or $0$). Then we have a finite collection $\{W_j\}_{j\in J}$ of rational hyperplanes in $d^\perp$, called \emph{(numerical) walls}, of the form
\begin{equation*}
W(d')=\{\theta\in d^\perp\ |\ \theta\cdot d'=0\}\,,
\end{equation*}
where $d'\in\mathbb{N}^I$ is such that $d'\leq d$ but does not divide $d$. The numerical equivalence classes in $d^\perp$ are the connected components of the locally closed subsets $\cap_{j_1\in J_1}W_{j_1}\setminus\cup_{j_2\in J_2}W_{j_2}$, for some partition $J=J_1\sqcup J_2$ (for $J=J_2$ these are called \emph{(numerical) chambers}). By construction, the subsets $R_{V,\theta}^\textup{ss}$ and $R_{V,\theta}^\textup{st}$ do not change when $\theta$ moves inside a numerical equivalence class, and any such a class contains integral arrays, because it is a cone and the walls are rational.

This means that also for a real or polynomial array $\theta$ orthogonal to $d$ the moduli spaces $\M^{\textup{ss}}_{Q,\theta}(d)$ and $\M^{\textup{st}}_{Q,\theta}(d)$ make sense and are constructed as GIT quotients after choosing a numerically equivalent integral weight $\theta'\in\mathbb{Z}^I$: for example, if $\theta=t\theta_1+\theta_0\in\mathbb{R}[t]^I$, then we can choose $\epsilon>0$ small enough so that $\theta$-(semi)stability is equivalent to $\theta'$-(semi)stability, where $\theta'\in\mathbb{Z}^I$ is some integral array lying in the same numerical equivalence class as $\theta_1+\epsilon\theta_0$.

\begin{example}\label{Ex: Kron modules}
A recurring example in this paper will be the \emph{Kronecker quiver}\\
\begin{center}\begin{tikzcd}
K_n: &-1\arrow[bend left=60]{r}\arrow[bend left=30]{r}\arrow[bend right=30]{r}{\vdots}\arrow[bend right=60]{r} &0
\end{tikzcd}\end{center}
with $n$ arrows. Its representations can be seen as linear maps $f:V_{-1}\otimes Z\to V_0$, where $Z$ is a $n$-dimensional vector space with a fixed basis, or with left modules over the \emph{Kronecker algebra} $\mathbb{C}K_n=\left( \begin{smallmatrix} \mathbb{C}&Z\\ 0&\mathbb{C} \end{smallmatrix} \right)$. Notice that the only arrays $\theta\in\mathbb{Z}^{\{-1,0\}}$ giving nontrivial stability weights $\nu_\theta$ on the representations of $K_n$ are those with $\theta^0>0$, and these are all in the same chamber: a Kronecker module $f$ is, accordingly, (semi)stable if and only if for any subrepresentation $W\subset V$ with $W_0\neq 0$ we have
\begin{equation*}
\frac{\dim W_0}{\dim W_{-1}}\ggeq\frac{\dim V_0}{\dim V_{-1}}\,.
\end{equation*}
This is the usual notion of (semi)stability for Kronecker modules (see e.g.\ \cite[Prop.\ 15]{Drez87Fibr}). We denote by
\begin{equation*}
K(n;d_{-1},d_0):=\M^\textup{ss}_{K_n,(-d_0,d_{-1})}(d)
\end{equation*}
the moduli space of semistable Kronecker modules of dimension vector $d$, and by
\begin{equation*}
K_\textup{st}(n;d_{-1},d_0)\subset K(n;d_{-1},d_0)
\end{equation*}
the stable locus. Some useful facts on these spaces are:
	\begin{enumerate}
	\item if $nd_{-1}<d_0$ or $d_{-1}>nd_0$, then $K(n;d_{-1},d_0)=\emptyset$;
	\item $K(n;d_{-1},nd_{-1})=K(n;nd_0,d_0)=\point$;
	\item $\dim K_\textup{st}(n;d_{-1},d_0)=nd_{-1}d_0+1-d_{-1}^2-d_0^2$;
	\item \cite[Prop.\ 21-22]{Drez87Fibr} we have isomorphisms
	\begin{equation*}
	\begin{array}{c}
	K_{(\textup{st})}(n;d_{-1},d_0)\simeq K_{(\textup{st})}(n;nd_{-1}-d_0,d_{-1})\simeq K_{(\textup{st})}(n;d_0,nd_0-d_{-1})\,,\\
	K_{(\textup{st})}(n;d_{-1},d_0)\simeq K_{(\textup{st})}(n;d_0,d_{-1})\,;
	\end{array}
	\end{equation*}
	\item $K(n;1,k)\simeq K(n;k,1)\simeq\G_k(n)$, the Grassmannian of $k$-dimensional subspaces of $\mathbb{C}^n$;
	\item \cite[Lemme 25]{Drez87Fibr} $K(3;2,2)\simeq\mathbb{P}^5$.
	\end{enumerate}
\end{example}

Often we will consider representations of $Q$ subject to certain \emph{relations}, that is combinations of arrows of length $\geq 2$ generating an ideal $J\subset\mathbb{C}Q$. These form an abelian subcategory $\Rep^\textup{fd}_\mathbb{C}(Q;J)\subset\Rep^\textup{fd}_\mathbb{C}(Q)$ equivalent to left $\mathbb{C}Q/J$-modules of finite dimension. Given $V$ and $\theta$ as above, the representations on $V$ subject to the relations make a $G_V$-invariant closed subscheme $X_J\subset R_V$, and thus the $\theta$-(semi)stable ones are parameterized by moduli stacks $\mathfrak{M}^\textup{ss/st}_{Q,J,\theta}(d)=[R_{V,\theta}^\textup{ss/st}\cap X_J/G_V]$ or by moduli spaces obtained as closed subschemes $M^\textup{ss/st}_{Q,J,\theta}(d)\subset\M^\textup{ss/st}_{Q,\theta}(d)$.

\subsection{Moduli spaces of semistable sheaves}\label{Moduli of sheaves}
Let $X$ be a smooth projective irreducible complex variety polarized by an ample divisor $A\subset X$. $\Coh_{\mathcal{O}_X}$ denotes the abelian category of coherent $\mathcal{O}_X$-modules, and $K_0(X)$ its Grothendieck group.\\
Given a sheaf $\mathcal{E}\in\Coh_{\mathcal{O}_X}$, we denote by $\rk\mathcal{E}$, $\deg_A\mathcal{E}:=c_1(\mathcal{E})\cdot A^{\dim X-1}$ its rank and degree (where $c(\mathcal{E})=1+\sum_{i\geq 1}c_i(\mathcal{E})$ is the Chern class), by $\ch\mathcal{E}$ its Chern character, by
\begin{equation*}
P_{\mathcal{E},A}(t)=\sum_{i=0}^{\dim\mathcal{E}}\frac{\alpha_i(\mathcal{E})}{i!}t^i:=\chi(X;\mathcal{E}(tA))
\end{equation*}
its Hilbert polynomial and by $\chi(\mathcal{E}):=P_{\mathcal{E},A}(0)$ its Euler characteristic. Note that these quantities are additive on short exact sequences: given a class $v\in K_0(X)$, it makes thus sense to write $\rk v,\deg_Av,\ch v$ and $P_{v,A}$. We also write $\mu_A(\mathcal{E}):=\deg_A\mathcal{E}/\rk\mathcal{E}$ for the slope of $\mathcal{E}$, $\dim\mathcal{E}$ for the dimension of its support, and $p_{\mathcal{E},A}(t):=P_{\mathcal{E},A}(t)/\alpha_{\dim\mathcal{E}}(\mathcal{E})$ for the reduced Hilbert polynomial. When $\mathcal{E}$ is torsion-free, we have $\dim\mathcal{E}=\dim X$ and $\rk\mathcal{E}=\alpha_{\dim X}(\mathcal{E})/A^{\dim X}$. Finally, the Hilbert polynomial $P_{v,A}$ can be computed by the Hirzebruch-Riemann-Roch Theorem:
	\begin{enumerate}
	\item if $\dim X=1$ and $g(X)$ is the genus of $X$, then
	\begin{equation}\label{Eq: Hilb pol sh cur}
	P_{v,A}(t)=t\,\rk v\,\deg(A)+\deg v+\rk v(1-g(X))\,;
	\end{equation}
	\item if $\dim X=2$, then
	\begin{equation}\label{Eq: Hilb pol sh sur}
	P_{v,A}(t)=
	t^2\frac{\rk v\,A^2}{2}
	\ +\ t\left(\deg_Av-\rk v\,\frac{A\cdot K_X}{2}\right )+\chi(v)\,,
	\end{equation}
	where $\chi(v)=\rk v\,\chi(X;\mathcal{O}_X)+\left(\ch_2v+c_1(v)c_1(X)/2\right)$.
	\end{enumerate}

Now we recall the main aspects of moduli spaces of semistable coherent sheaves, mainly following \cite{HuyLeh10Geo}. 

\begin{definition}
$\mathcal{E}\in\Coh_{\mathcal{O}_X}$ is said to be \emph{Gieseker-(semi)stable} with respect to $A$ if it is $P_{\cdot,A}$-(semi)stable according to Def.\ \ref{Defn: P-(semi)st}.
\end{definition}

Here we are seeing the Hilbert polynomial as a polynomial stability $P_{\cdot,A}:K_0(X)\to\mathbb{R}[t]$. So $\mathcal{E}$ is Gieseker-(semi)stable if and only if for any coherent subsheaf $0\neq\mathcal{F}\subsetneq\mathcal{E}$ we have the inequality $P_{\mathcal{F},A}\prqG P_{\mathcal{E},A}$, where $\prqG$ is the preorder introduced in Eq.\ \eqref{Eq: def Gies preor}. Lemma \ref{Lem: equiv char Gies ord} says that this inequality is equivalent to
\begin{equation*}
\alpha_{\dim\mathcal{E}}(\mathcal{E})P_{\mathcal{F},A}(t)\lleq\alpha_{\dim\mathcal{E}}(\mathcal{F})P_{\mathcal{E},A}(t)
\end{equation*}
(where as usual $\leq$ is the lexicographical order), so our definition agrees with the standard one given in \cite[\S1.2]{HuyLeh10Geo}. This reformulation of Gieseker stability will turn out to be useful in the rest of the paper.

\begin{definition}
A torsion-free sheaf $\mathcal{E}\in\Coh_{\mathcal{O}_X}$ is \emph{slope-(semi)stable} if for any coherent subsheaf $\mathcal{F}\subsetneq\mathcal{E}$ with $0<\rk\mathcal{F}<\rk\mathcal{E}$ we have $\mu_A(\mathcal{F})\lleq\mu_A(\mathcal{E})$.
\end{definition}

Some remarks on the notion of Gieseker (semi)stability:
	\begin{enumerate}
	\item If $\mathcal{E}$ is Gieseker-semistable, then it is automatically pure (that is, all its subsheaves have the same dimension), and in particular it is torsion-free if and only if $\dim\mathcal{E}=\dim X$.
	\item The category $\Coh_{\mathcal{O}_X}$ is Noetherian and Hilbert polynomials are numerical; then, as discussed after Def.\ \ref{Defn: HN filtr P}, any coherent sheaf $\mathcal{E}$ has a unique Harder-Narasimhan filtration
	\begin{equation*}
	0=\mathcal{E}_0\subsetneq\mathcal{E}_1\subsetneq\cdots\subsetneq\mathcal{E}_\ell=\mathcal{E}
	\end{equation*}
	with Gieseker-semistable quotients $\mathcal{E}_i/\mathcal{E}_{i-1}$ of $\prqG$-decreasing Hilbert polynomials (when $\mathcal{E}$ is pure this simply means that $p_{\mathcal{E}_1,A}>p_{\mathcal{E}_2/\mathcal{E}_1,A}>\cdots>p_{\mathcal{E}/\mathcal{E}_{\ell-1},A}$). We write
	\begin{equation}\label{Eq: min max Hilb pol}
	P_{\mathcal{E},A,\max}:=P_{\mathcal{E}_1,A}\,,\ \ \ \ P_{\mathcal{E},A,\min}:=P_{\mathcal{E}/\mathcal{E}_{\ell-1},A}\,.
	\end{equation}
	Moreover, Gieseker-semistable sheaves with fixed reduced Hilbert polynomial $p\in\mathbb{Q}[t]$ form an abelian subcategory
	\begin{equation}\label{Eq: cat semis coh shvs}
	\mathcal{S}_A(p)\subset\Coh_{\mathcal{O}_X}
	\end{equation}
	of finite length and closed under extensions; two sheaves in $\mathcal{S}_A(p)$ are called \emph{S-equivalent} if they have the same composition factors.
	\item Suppose that $\dim X=1$: any $\mathcal{E}\in\Coh_{\mathcal{O}_X}$ is the direct sum of its torsion-free and torsion parts, so it is pure if and only if they are not both nonzero; a torsion-free $\mathcal{E}$ (which is also a vector bundle) is Gieseker-(semi)stable if and only if it is slope-(semi)stable, and the slope condition can be checked on vector subbundles only; on the other hand, any torsion sheaf is Gieseker-semistable, and it is Gieseker-stable if and only if it is a simple object in $\Coh_{\mathcal{O}_X}$, that is a skyscraper sheaf.
	\end{enumerate}

The main reason to introduce semistability was the construction of moduli spaces:

\begin{theorem}\label{Thm: mod sp sst shvs}
Fix a numerical class $v\in K_{\num}(X)$. There exists a projective $\mathbb{C}$-scheme $\M^\textup{ss}_{X,A}(v)$ which is a coarse moduli space for $S$-equivalence classes of coherent $\mathcal{O}_X$-modules in $v$ which are Gieseker-semistable with respect to $A$. It also has an open subscheme $\M^\textup{st}_{X,A}(v)$ parameterizing isomorphism classes of Gieseker-stable sheaves.
\end{theorem}

Again, these are \emph{moduli spaces} in that they corepresent moduli stacks $\mathfrak{M}^\textup{ss}_{X,A}(v)$ and $\mathfrak{M}^\textup{st}_{X,A}(v)$ of families of sheaves. See \cite{HuyLeh10Geo,Neum09Alg} for the details.

\begin{remark}\label{Rmk: prop mod sp sh surf}
When $\dim X=2$, some properties of these spaces proven from their construction are:
	\begin{enumerate}
	\item \cite[Theorem 3.4.1]{HuyLeh10Geo} if $\M^\textup{ss}_{X,A}(v)\neq\emptyset$, then the \emph{Bogomolov inequality} holds:
	\begin{equation}\label{Eq: Bogom ineq}
	\Delta(v):=c_1(v)^2-2\rk v\,\ch_2(v)\geq 0\,;
	\end{equation}
	\item if $\deg_H\omega_X<0$, then for any stable $\mathcal{F}$ the obstruction space
	\begin{equation*}
	\Ext^2(\mathcal{F},\mathcal{F})\simeq\Hom(\mathcal{F},\mathcal{F}\otimes\omega_X)^\vee
	\end{equation*}
	vanishes and $\End(\mathcal{F})\simeq\mathbb{C}$, thus the tangent space $\Ext^1(\mathcal{F},\mathcal{F})$ has dimension $1-\chi(\mathcal{F},\mathcal{F})$; hence, by \cite[Corollary 4.5.2]{HuyLeh10Geo}, $\M^\textup{st}_{X,A}(v)$ is smooth of dimension
	\begin{equation}\label{Eq: dim ms sh sur}
	\dim\M^\textup{st}_{X,A}(v)=1-\chi(v,v)=1-(\rk v)^2\chi(\mathcal{O}_X)+\Delta(v)\,;
	\end{equation}
	\item \cite[Corollary 4.6.7]{HuyLeh10Geo} if $\gcd(\rk v,\deg_Av,\chi(v))=1$, then $\M^\textup{ss}_{X,A}(v)$ is equal to $\M^\textup{st}_{X,A}(v)$ and it has a universal family.
	\end{enumerate}
\end{remark}
	
Finally, to simplify the computations it is useful to introduce two alternating forms $\sigma_{\M},\sigma_\chi:K_0(X)\times K_0(X)\to\mathbb{Z}$, given by
\begin{equation*}
\sigma_{\M}(v,w):=\deg_Av\rk w-\deg_Aw\rk v\,,\quad \sigma_\chi(v,w):=\chi(v)\rk w-\chi(w)\rk v\,,
\end{equation*}
and also the $\mathbb{Z}[t]$-valued form
\begin{equation}\label{Eq: sigmaG}
\sigma_{\G}:=t\sigma_{\M}+\sigma_\chi\,.
\end{equation}
Now we can express Gieseker stability on curves and surfaces as stability with respect to these forms, in the sense of Def.\ \ref{Defn: stab alt form}:


\begin{lemma}\label{Lem: Gies,slope-stab and sigma}\ 
	\begin{enumerate}
	\item If $\dim X=1$, then $\sigma_{\M}=\sigma_\chi$, and Gieseker (semi)stability and $\sigma_{\M}$-(semi)stability of sheaves are equivalent; for sheaves of positive rank these are also equivalent to slope-stability;
	\item if $\dim X=2$, then for sheaves of positive rank Gieseker (semi)stability is equivalent to $\sigma_{\G}$-(semi)stability; for torsion-free sheaves, slope semistability is equivalent to $\sigma_{\M}$-semistability.
	\end{enumerate}
\end{lemma}

For $\dim X=2$, the restriction to positive rank is necessary as $\sigma_{\G}$ vanishes identically on sheaves supported on points. Note also that $\mathcal{O}_X$ is slope-stable but not $\sigma_{\M}$-stable, as $\sigma_{\M}(\mathcal{I}_x,\mathcal{O}_X)=0$, where $\mathcal{I}_x\subset\mathcal{O}_X$ is the ideal sheaf of a point.

\begin{proof}\ 
	\begin{enumerate}
	\item The first statement is just the observation that the alternating form induced by the Hilbert polynomial as in eq.\ \eqref{Eq: sigmaP} is $\sigma_{P_{\cdot,A}}=\deg A\,\sigma_{\M}=\deg A\,\sigma_\chi$. The second statement is also obvious.
	\item In this case we have
	\begin{equation*}
	\sigma_{P_{\cdot,A}}=\frac{t^2}{2}A^2\sigma_{\M}+\left(t\,A^2-\frac{A\cdot K_X}{2}\right)\sigma_\chi+\sigma_0\,,
	\end{equation*}
	where $\sigma_0(v,w):=\chi(v)\deg w-\chi(w)\deg v$. But if $\rk w\neq 0$, then $\sigma_0$ is irrelevant as $\sigma_{\M}(v,w)=\sigma_\chi(v,w)=0$ implies $\sigma_0(v,w)=0$, so $\sigma_{P_{\cdot,A}}$ can be replaced by $\sigma_{\G}=t\sigma_{\M}+\sigma_{\chi}$.\\
	The final claim follows from the equality
	\begin{equation*}
	\mu_A(v)-\mu_A(w)=\frac{1}{\rk v\rk w}\sigma_{\M}(v,w)
	\end{equation*}
	and from the fact that any coherent subsheaf $\mathcal{F}\subsetneq\mathcal{E}$ with $\rk\mathcal{F}=\rk\mathcal{E}$ gives
	\begin{equation*}
	\sigma_{\M}(\mathcal{F},\mathcal{E})=-\deg(\mathcal{E}/\mathcal{F})\rk\mathcal{E}\leq 0\,.
	\end{equation*}
	\end{enumerate}
\end{proof}

\begin{remark}
In fact, the same arguments apply to any heart $\mathcal{A}\subset D^b(X)$ of a bounded t-structure: if $\dim X=1$, then $P_{\cdot,A}$-(semi)stability and $\sigma_{\M}$-(semi)stability in $\mathcal{A}$ are equivalent; if $\dim X=2$, then $P_{\cdot,A}$-(semi)stability and $\sigma_{\G}$-(semi)stability are equivalent for objects of nonzero rank in $\mathcal{A}$.
\end{remark}


%
\subsection{Exceptional sequences}\label{Exceptional collections}
Let $\mathcal{D}$ be a $\mathbb{C}$-linear triangulated category of finite type.

\begin{definition}
An object $E\in\Ob(\mathcal{D})$ is called \emph{exceptional} when, for all $\ell\in\mathbb{Z}$,
\begin{equation*}
\Hom_\mathcal{D}(E,E[\ell])=\left\{\begin{matrix}
\mathbb{C} &\textup{if }\ell=0\,, \\ 
0 &\textup{if }\ell\neq0\,.
\end{matrix}\right.
\end{equation*}
A sequence $\mathfrak{E}=(E_0,...,E_n)$ of exceptional objects is called an \emph{exceptional sequence}, or \emph{exceptional collection}, if
\begin{equation*}
\Hom_\mathcal{D}(E_i,E_j[\ell])=0
\end{equation*}
for all $i>j$ and all $\ell\in\mathbb{Z}$. The exceptional sequence is said to be \emph{strong} if in addition $\Hom_\mathcal{D}(E_i,E_j[\ell])=0$ for all $i,j$ and all $\ell\in\mathbb{Z}\setminus\{0\}$; it is said to be \emph{full} if the smallest triangulated subcategory containing $E_0,...,E_n$ is $\mathcal{D}$.\\
Finally, exceptional collections ${}^\vee\!\mathfrak{E}=({}^\vee\!E_n,...,{}^\vee\!E_0)$ and $\mathfrak{E}^\vee=(E_n^\vee,...,E_0^\vee)$ are respectively called \emph{left dual} and \emph{right dual} to $\mathfrak{E}$ if
\begin{equation*}
\Hom_\mathcal{D}({}^\vee\!E_i,E_j[\ell])=\left\{\begin{matrix}
\mathbb{C} &\textup{if }i=j=n-\ell\,, \\ 
0 &\textup{otherwise}\,,
\end{matrix}\right.\quad\quad
\Hom_\mathcal{D}(E_i,E_j^\vee[\ell])=\left\{\begin{matrix}
\mathbb{C} &\textup{if }i=j=\ell\,, \\ 
0 &\textup{otherwise}\,.
\end{matrix}\right.
\end{equation*}
\end{definition}

Given a full exceptional collection $\mathfrak{E}$, its left and right dual always exist and are unique and full, and they can be realized by repeated \emph{mutations} \cite[\S2]{GorKul04Hel}. Notice also that if a full exceptional collection exists, then the Euler form $\chi$ is nondegenerate, and $K_0(\mathcal{D})=K_{\num}(\mathcal{D})$ is freely generated by the elements of the collection.

\begin{examples}\label{Exs: exc coll}\ 
\begin{enumerate}
\item $D^b(\mathbb{P}^2)$ has a full exceptional collection $\mathfrak{E}=(\mathcal{O}_{\mathbb{P}^2}(-1),\mathcal{O}_{\mathbb{P}^2},\mathcal{O}_{\mathbb{P}^2}(1))$ which has left dual ${}^\vee\!\mathfrak{E}=(\mathcal{O}_{\mathbb{P}^2}(1),\uptau_{\mathbb{P}^2},\wedge^2\uptau_{\mathbb{P}^2}(-1))$ and right dual $\mathfrak{E}^\vee=(\Omhol_{\mathbb{P}^2}^2(1),\Omhol_{\mathbb{P}^2}^1,\mathcal{O}_{\mathbb{P}^2}(-1))$. All these collections are strong.
\item $D^b(\mathbb{P}^1\times\mathbb{P}^1)$ has a full exceptional collection
\begin{equation*}
(\mathcal{O}_X(0,-1)[-1],\mathcal{O}_X[-1],\mathcal{O}_X(1,-1),\mathcal{O}_X(1,0))
\end{equation*}
with left dual given by
\begin{equation*}
\begin{array}{c}
(\mathcal{O}_X(1,0),\mathcal{O}_{\mathbb{P}^1}(1)\boxtimes\uptau_{\mathbb{P}^1}(-1),\uptau_{\mathbb{P}^1}\boxtimes\mathcal{O}_{\mathbb{P}^1},\uptau_{\mathbb{P}^1}\boxtimes\uptau_{\mathbb{P}^1}(-1))\\
\quad\simeq(\mathcal{O}_X(1,0),\mathcal{O}_X(1,1),\mathcal{O}_X(2,0),\mathcal{O}_X(2,1))\,.
\end{array}
\end{equation*}
The latter collection is strong, while the former is not.
\item Let $Q$ be an ordered quiver with relations $J$, whose vertices are labeled by $0,1,...,n$ (this means that there are no arrows from $i$ to $j$ if $j\leq i$). Then we have full exceptional collections $\mathfrak{E},{}^\vee\!\mathfrak{E}$ on the bounded derived category $D^b(Q;J):=D^b(\Rep^\textup{fd}_\mathbb{C}(Q;J))$ made by the objects
\begin{equation*}
E_i=S(i)[i-n]\,,\quad {}^\vee\!E_i=P(i)\,,
\end{equation*}
where $S(i)$ and $P(i)$ denote the standard simple and projective representations associated to each vertex $i$. Moreover, the collection ${}^\vee\!\mathfrak{E}$ is obviously strong and it is left dual to $\mathfrak{E}$ because of the formula
\begin{equation*}
\Ext^\ell(P(i),S(j))=\left\{\begin{matrix}
\mathbb{C} &\textup{if }i=j,\ \ell=0 \\ 
0 &\textup{otherwise}
\end{matrix}\right.\,.
\end{equation*}
\end{enumerate}\end{examples}

In the last example, the full strong collection made by the projective representations is somehow prototypical: suppose that ${}^\vee\!\mathfrak{E}=({}^\vee\!E_n,...,{}^\vee\!E_0)$ is a full and strong exceptional collection on $\mathcal{D}$, and let $T:=\oplus_{i=0}^n{}^\vee\!E_i$; the endomorphism algebra
\begin{equation*}
A:=\End_\mathcal{D}(T)=\begin{pmatrix}
\Hom({}^\vee\!E_n,{}^\vee\!E_n)) &  & \\ 
\vdots & \ddots & \\ 
\Hom({}^\vee\!E_n,{}^\vee\!E_0)) & \cdots & \Hom({}^\vee\!E_0,{}^\vee\!E_0))
\end{pmatrix}
\end{equation*}
is basic, and hence it can be identified with $(\mathbb{C}Q/J)^{\op}$ for some ordered quiver $Q$ with vertices $I=\{0,1,...,n\}$ and relations $J\subset\mathbb{C}Q$;\footnote{See \cite[\S II.3]{AsSiSk06Elem} for details, but note that the opposite convention for path algebras is used there.} in particular, we identify right $A$-modules of finite dimension with representations of $(Q,J)$. Then we have (under some additional hypotheses on $\mathcal{D}$ which are satisfied e.g.\ when $\mathcal{D}=D^b(X)$ for a smooth projective variety $X$):

\begin{theorem}\cite[Thm 6.2]{Bond89Repr}\label{Thm: Baer-Bondal}
${}^\vee\!\mathfrak{E}$ induces a triangulated equivalence
\begin{equation*}
\Phi_{{}^\vee\!\mathfrak{E}}=R\Hom_\mathcal{D}(T,\cdot):\mathcal{D}\longrightarrow D^b(Q;J)\,.
\end{equation*}
\end{theorem}

More explicitly, $\Phi_{{}^\vee\!\mathfrak{E}}$ maps an object $D$ of $\mathcal{D}$ to a complex of representations which at the vertex $i\in\{0,...,n\}$ of $Q$ has the graded vector space $R\Hom_\mathcal{D}({}^\vee\!E_i,D)$.

\begin{remark}\label{Rmk: heart ind by Phi}
Notice that $\Phi_{{}^\vee\!\mathfrak{E}}$ maps each ${}^\vee\!E_i$ to the projective representation $P(i)$ of $Q$ and each dual $E_i$ to the simple $S(i)[i-n]$, and the standard heart $\Rep^\textup{fd}_\mathbb{C}(Q;J)\subset D^b(Q;J)$ is the extension closure of the simple modules $S(i)$. Hence, ${}^\vee\!\mathfrak{E}$ induces a bounded t-structure on $\mathcal{D}$ whose heart is the extension closure of the objects $E_i[n-i]$, $i=0,...,n$.
\end{remark}

\subsection{Families of objects in the derived category}\label{Families obj der cat}
Let $X$ be a smooth projective irreducible complex variety polarized by an ample divisor $A\subset X$.

Take a $\mathbb{C}$-scheme $S$ of finite type and the heart $\mathcal{A}\subset D^b(X)$ of a bounded t-structure. Following \cite[Def.\ 3.7]{Brid02Flops}, by a \emph{family} over $S$ of objects of $\mathcal{A}$ having a common property P we mean an object $\mathcal{F}$ of $D^b(X\times S)$ such that, for any (closed) point $s\in S$, the object $\mathcal{F}_s:=L\iota_s^*\mathcal{F}$ is in $\mathcal{A}$ and has the property P, where $\iota_s:X\to X\times S$ maps $x$ to $(x,s)$. 

We are mostly interested in two kinds of families of objects:
	\begin{enumerate}
	\item Denote by $\mathcal{C}=\Coh_{\mathcal{O}_X}\subset D^b(X)$ the heart of the standard t-structure. If $\mathcal{F}$ is a family of objects of $\mathcal{C}$, then by \cite[Lemma 3.31]{Huyb06Four} it is isomorphic to a coherent $\mathcal{O}_{X\times S}$-module flat over $S$. In particular, the moduli space $\M^\textup{ss}_{X,A}(v)$ of Theorem \ref{Thm: mod sp sst shvs} corepresents the moduli functor of families of $\sigma_{\G}$-semistable objects of class $v\in K_0(X)$ in the heart $\mathcal{C}$.
	\item Suppose that $D^b(X)$ has a full strong exceptional collection ${}^\vee\!\mathfrak{E}$, and consider the equivalence $\Phi_{{}^\vee\!\mathfrak{E}}:D^b(X)\to D^b(Q;J)$ of Theorem \ref{Thm: Baer-Bondal} together with the induced isomorphism $\phi:K_0(X)\to K_0(Q;J)$ and the induced heart $\mathcal{K}:=\Phi_{{}^\vee\!\mathfrak{E}}^{-1}(\Rep^\textup{fd}_\mathbb{C}(Q;J))$. Then families of objects in $\mathcal{K}$ correspond, via $\Phi_{{}^\vee\!\mathfrak{E}}$, to families of $\mathbb{C}Q/J$-modules in the sense of \cite[Def.\ 5.1]{King94Mod} (see e.g.\ \cite[Prop.\ 4.4]{Ohka10Mod} and \cite[\S7.3]{BaCrZh17Nef} for details). In particular, if for $v\in K_0(X)$ we write $d^v:=\dimvec\phi(v)$ and denote by $\theta_{\G,v}\in\mathbb{Z}[t]^I$ the array such that $\theta_{\G,v}\cdot d^w=\sigma_{\G}(v,w)$ for all $w\in K_0(X)$, then the moduli space $\M_{Q,J,\theta_{\G,v}}(d^v)$ corepresents the moduli functor of families of $\sigma_{\G}$-semistable objects of class $v$ in the heart $\mathcal{K}$.
	\end{enumerate}

%% file: P1.tex

%
%
\section{Sheaves on \texorpdfstring{$\mathbb{P}^1$}{P1} and Kronecker modules}\label{Semshequirep P1}
In this section the well-known classification of coherent sheaves on $\mathbb{P}^1$ is deduced via the representation theory of the Kronecker quiver $K_2$, as an easy anticipation of the ideas introduced in the next sections.

\subsection{Representations of \texorpdfstring{$K_2$}{K2} and Kronecker complexes on \texorpdfstring{$\mathbb{P}^1$}{P1}}\label{Moduli of reps of K2}

Let $Z$ be a 2-dimensional $\mathbb{C}$-vector space with a basis $\{e_0,e_1\}$, and consider the complex projective line $\mathbb{P}^1:=\mathbb{P}_\mathbb{C}(Z)$. Fix also an integer $k\in\mathbb{Z}$.

We are interested in the finite-dimensional representations of the \emph{Kronecker quiver}\\
\begin{center}$K_2$: \begin{tikzcd}
-1\arrow[bend left=50]{r}\arrow[bend right=50]{r} &0
\end{tikzcd},\end{center}
that is Kronecker modules $f\in\Hom_\mathbb{C}(V_{-1}\otimes Z,V_0)$ (see Example \ref{Ex: Kron modules}), and their relations with sheaves on $\mathbb{P}^1$. The couple $\mathfrak{E}_k=(E_{-1},E_0):=(\mathcal{O}_{\mathbb{P}^1}(k-1),\mathcal{O}_{\mathbb{P}^1}(k))$ is a full strong exceptional collection in $D^b(\mathbb{P}^1)$, and so is its left dual collection, which is given by ${}^\vee\!\mathfrak{E}_k=({}^\vee\!E_0,{}^\vee\!E_{-1}):=(\mathcal{O}_{\mathbb{P}^1}(k),\uptau_{\mathbb{P}^1}(k-1))$, where $\uptau_{\mathbb{P}^1}\simeq\mathcal{O}_{\mathbb{P}^1}(2)$ denotes the tangent sheaf. Hence, the tilting sheaf
\begin{equation*}
T_k:=\oplus_{i=-1}^0{}^\vee\!E_i=\mathcal{O}_{\mathbb{P}^1}(k)\oplus\uptau_{\mathbb{P}^1}(k-1)
\end{equation*}
induces by Theorem \ref{Thm: Baer-Bondal} a derived equivalence
\begin{equation*}
\Psi_k:=\Phi_{{}^\vee\!\mathfrak{E}_k}:D^b(\mathbb{P}^1)\to D^b(K_2)\,,
\end{equation*}
as $\End_{\mathcal{O}_{\mathbb{P}^1}}(T_k)$ may be identified with $\mathbb{C}K_2^{\op}$ via the isomorphism $H^0(\mathbb{P}^1;\uptau_{\mathbb{P}^1}(-1))\cong Z$. $\Psi_k$ sends a complex $\mathcal{F}^\bullet$ of coherent sheaves to the complex of representations
\begin{equation}\label{Eqn: images of cplx under der eq}
R\Hom_{\mathcal{O}_{\mathbb{P}^1}}(\uptau_{\mathbb{P}^1}(k-1),\mathcal{F}^\bullet)\rightrightarrows
R\Hom_{\mathcal{O}_{\mathbb{P}^1}}(\mathcal{O}_{\mathbb{P}^1}(k),\mathcal{F}^\bullet)\,.
\end{equation}
We denote by $\mathcal{C}\subset D^b(\mathbb{P}^1)$ the heart of the standard t-structure and by $\mathcal{K}_k\subset D^b(\mathbb{P}^1)$ the heart of the t-structure induced from the standard one in $D^b(K_2)$ via the equivalence $\Psi_k$.

\begin{lemma}
The objects of $\mathcal{K}_k$ are, up to isomorphism in $D^b(\mathbb{P}^1)$, the \emph{Kronecker complexes}
\begin{equation}\label{Eq. Kron cplx P1}
V_{-1}\otimes\mathcal{O}_{\mathbb{P}^1}(k-1)\longrightarrow V_0\otimes\mathcal{O}_{\mathbb{P}^1}(k)\,.
\end{equation}
\end{lemma}

\begin{proof}
Let $A:=\End_{\mathcal{O}_{\mathbb{P}^1}}(T_k)$. $\Psi_k$ maps the exceptional objects ${}^\vee\!E_i$, $i=0,-1$, to the standard projective right $A$-modules $\Id_{{}^\vee\!E_i}A$, which correspond to the Kronecker modules $P_0=(0\to\mathbb{C})$ and $P_{-1}=(\mathbb{C}\otimes Z\overset{\Id}{\to}Z)$; now the heart $\Rep^\textup{fd}_\mathbb{C}(K_2)$, which is the extension closure of the simple modules $S_{-1},S_0$, is mapped to the extension closure $\mathcal{K}_k$ of $E_{-1}[1],E_0$ (see Remark \ref{Rmk: heart ind by Phi}), whose objects are Kronecker complexes.
\end{proof}

$\Psi_k$ induces an isomorphism $\psi_k:K_0(\mathbb{P}^1)\to K_0(K_2)$ between the Grothendieck groups, which are free of rank 2. Hence, coordinates of an element $v\in K_0(\mathbb{P}^1)$ are provided either by the couple $(\rk v,\deg v)$ or by the dimension vector
\begin{equation*}
d^v=(d^v_{-1},d^v_0):=\dimvec(\psi_k(v))\,.
\end{equation*}
The simple representations $S(-1)$ and $S(0)$, whose dimension vectors are $(1,0)$ and $(0,1)$ respectively, correspond to the complexes $\mathcal{O}_{\mathbb{P}^1}(k-1)[1]$, with $(\rk,\deg)=(-1,1-k)$, and $\mathcal{O}_{\mathbb{P}^1}(k)$, with $(\rk,\deg)=(1,k)$. So we deduce that the linear transformation between the two sets of coordinates is given by
	\begin{equation}\label{Eq: coord trans Gr grp P1}
	\begin{array}{c}
	\begin{pmatrix}
	\rk v\\ 
	\deg v
	\end{pmatrix}=\begin{pmatrix}
	-1 & 1\\ 
	1-k & k
	\end{pmatrix}
	\begin{pmatrix}
	d^v_{-1} \\ 
	d^v_0
	\end{pmatrix}\,,\ \ \ \ 
	\begin{pmatrix}
	d^v_{-1} \\ 
	d^v_0
	\end{pmatrix}
	=\begin{pmatrix}
	-k & 1\\ 
	1-k & 1
	\end{pmatrix}
	\begin{pmatrix}
	\rk v\\ 
	\deg v
	\end{pmatrix}
	\end{array}\,.
	\end{equation}

\subsection{Semistable sheaves and Kronecker complexes}
As in the end of \S\ref{Moduli of sheaves}, we consider the alternating form $\sigma_{\M}:K_0(\mathbb{P}^1)\times K_0(\mathbb{P}^1)\to\mathbb{Z}$ given by
\begin{equation*}
\sigma_{\M}(v,w):=\deg v\rk w-\deg w\rk v\,.
\end{equation*}
This is also the alternating form $\sigma_Z$ induced by the central charge $Z=-\deg+i\rk$ as in equation \eqref{Eq: sigmaZ}. We have seen in Lemma \ref{Lem: Gies,slope-stab and sigma} that, on the standard heart $\mathcal{C}=\Coh_{\mathcal{O}_{\mathbb{P}^1}}$, $\sigma_{\M}$ reproduces Gieseker-stability. Now we also consider $\sigma_{\M}$-stability on the heart $\mathcal{K}_k$:

\begin{definition}
A Kronecker complex $K_V$ is said to be \emph{(semi-)stable} when it is $\sigma_{\M}$-(semi)stable in $\mathcal{K}_k$ (Def.\ \ref{Defn: stab alt form}), that is when for any nonzero Kronecker subcomplex $K_W\subset K_V$ we have
\begin{equation*}
\deg K_V\rk K_W-\deg K_W\rk K_V\lleq 0\,.
\end{equation*}
\end{definition}

If we fix $v\in K_0(\mathbb{P}^1)$, then we can write
\begin{equation*}
\nu_{\M,v}(w):=\sigma_{\M}(v,w)=-d^v_0d^w_{-1}+d^v_{-1}d^w_0=\theta_{\M,v}\cdot d^w\,,
\end{equation*}
where the dot is the standard scalar product in $\mathbb{Z}^{\{-1,0\}}$ and 
\begin{equation*}
\theta_{\M,v}:=\binom{-d^v_0}{d^v_{-1}}=\binom{(k-1)\rk v-\deg v}{-k \rk v+\deg v}\,.
\end{equation*}
So, via the equivalence $\Psi_k$, (semi)stability of a Kronecker complex $K_V$ with $[K_V]=v$ is equivalent to $\theta_{\M,v}$-(semi)stability of the corresponding representation $V$ of $K_2$; being $\theta_{\M,v}^0=d^v_{-1}\geq0$, this is the usual definition of (semi)stable Kronecker module (see Ex.\ \ref{Ex: Kron modules}).

Consider an object in the intersection of the hearts $\mathcal{K}_k$ and $\mathcal{C}$ in $D^b(\mathbb{P}^1)$: this can be seen either as an injective Kronecker complex or as the sheaf given by its cokernel. The following observation shows that for such an object the two notions of stability coincide:

\begin{proposition}\label{Prop: semist sh-Kk on P1}
$\mathcal{K}_k$ is the heart obtained by tilting the standard heart $\mathcal{C}$ with respect to the central charge $Z=-\deg+i\rk$ at phase $\phi_k:=\arg(-k+i)/\pi$, as in \S\ref{Stab triang cat}. In particular, for any $\phi\in[\phi_k,1]$ the categories of $Z$-semistable objects with phase $\phi$ in the two hearts coincide: $\mathcal{S}^{(\mathcal{C})}_Z(\phi)=\mathcal{S}^{(\mathcal{K}_k)}_Z(\phi)$.
\end{proposition}

We denote by $\mathpzc{R}_k\subset K_0(\mathbb{P}^1)$ the cone spanned by the objects of $\mathcal{C}\cap\mathcal{K}_k$, that is
\begin{equation}\label{Eq: region Rk on P1}
\begin{array}{rcl}
\mathpzc{R}_k &:=&\{v\in K_0(\mathbb{P}^1)\ |\ \rk v\geq 0\textup{ and }\deg v\geq k\rk v\}\\
  &=&\{v\in K_0(\mathbb{P}^1)\ |\ d^v_0\geq d^v_{-1}\geq 0\}
\end{array}
\end{equation}

The proposition implies (as a special case of Lemma \ref{Lem: tilted hearts compatible}) that for any class $v\in\mathpzc{R}_k$ the hearts $\mathcal{C},\mathcal{K}$ are $(\sigma_{\G},v)$-compatible (Def.\ \ref{Defn: (sigma,v)-compat hearts}). Namely, we have:
	\begin{itemize}
	\item[(C1)] a slope-(semi)stable sheaf $\mathcal{F}\in\mathcal{C}$ with $[\mathcal{F}]=v$ belongs to $\mathcal{K}_k$, that is, it is isomorphic to the cokernel of an injective Kronecker complex $K_V\in\mathcal{K}_k$; similarly, a (semi)stable Kronecker complex $K_V\in\mathcal{K}_k$ with $[K_V]=v$ belongs to $\mathcal{C}$, which means that it is injective;
	\item[(C2)] an object $K_V\simeq\mathcal{F}$ of class $v$ in $\mathcal{C}\cap\mathcal{K}_k$ is (semi)-stable as a Kronecker complex if and only if it is (semi)-stable as a sheaf.
	\end{itemize}

\begin{proof}
The heart $\mathcal{K}_k$ lies in $\langle\mathcal{C},\mathcal{C}[1]\rangle_{\ext}$ and then by \cite[Lemma 1.1.2]{Poli07Const} it is obtained by tilting $\mathcal{C}$ at the torsion pair $(\mathcal{T}_k,\mathcal{F}_k)$ given by $\mathcal{T}_k:=\mathcal{C}\cap\mathcal{K}_k$ and $\mathcal{F}_k:=\mathcal{C}\cap\mathcal{K}_k[-1]$.
Now fix $k\in\mathbb{Z}$, take the phase $\phi_k=\arg(-k+i)/\pi$ of $\mathcal{O}_{\mathbb{P}^1}(k)$, and consider the torsion pair $(\mathcal{T}^Z_{\geq\phi_k},\mathcal{F}^Z_{<\phi_k})$ induced by $Z$, eqn.\ \eqref{Eqn: tors pair ind by Z}. Now, using the explicit form \eqref{Eqn: images of cplx under der eq} of $\Psi_k$, we will see that $\mathcal{T}^Z_{\geq\phi_k}\subset\mathcal{T}_k$ and $\mathcal{F}^Z_{<\phi_k}\subset\mathcal{F}_k$, which implies that the two torsion pairs must coincide: a sheaf $\mathcal{G}\in\mathcal{T}^Z_{\geq\phi_k}$ satisfies $\Ext^1(\uptau_{\mathbb{P}^1}(k-1),\mathcal{G})=\Ext^1(\mathcal{O}_{\mathbb{P}^1}(k),\mathcal{G})=0$ by Serre duality, and thus it belongs to $\mathcal{K}_k$, and hence to $\mathcal{T}_k$. On the other hand, for a sheaf $\mathcal{F}\in\mathcal{F}^Z_{<\phi_k}$ we have $\Hom(\uptau_{\mathbb{P}^1}(k-1),\mathcal{F})=\Hom(\mathcal{O}_{\mathbb{P}^1}(k),\mathcal{F})=0$, which means that it belongs to $\mathcal{K}_k[-1]$, and hence to $\mathcal{F}_k$.
\end{proof}

\begin{figure}
\centering
\includegraphics[height=3cm]{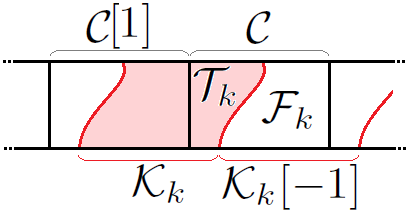}
\caption{The hearts $\mathcal{C},\mathcal{K}_k\subset D^b(\mathbb{P}^1)$}
\end{figure}

\begin{figure}
\centering
\includegraphics[height=7cm]{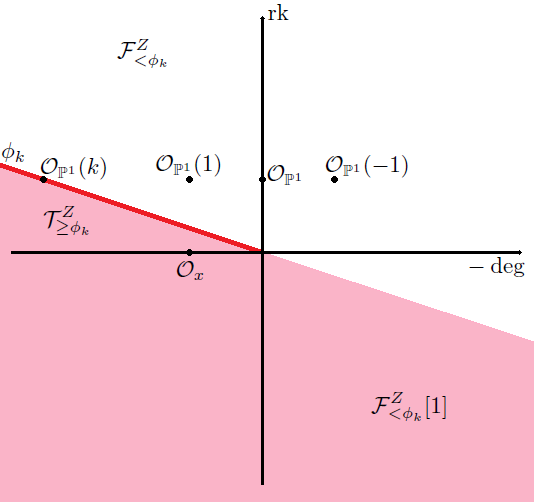}
\caption{The Grothendieck group $K_0(\mathbb{P}^1)$}
\end{figure}

\begin{corollary}\label{Cor: Bir-Gro Thm} (Birkhoff-Grothendieck Theorem)
Every coherent sheaf $\mathcal{F}\in\Coh_{\mathcal{O}_{\mathbb{P}^1}}$ is a direct sum of line bundles $\mathcal{O}_{\mathbb{P}^1}(\ell)$ and structure sheaves of fat points.
\end{corollary}

\begin{proof}
For an object in $\mathcal{T}^Z_{\geq k}$, being indecomposable is the same when considered in $\mathcal{C}$ or $\mathcal{K}_k$. All the indecomposable representations of $K_2$ are listed below for $n\geq 1$ (see e.g.\ \cite[Theorem 4.3.2]{Bens98Repr}):
\begin{equation*}
\mathbb{C}^n\overset{\mathbb{I}_n}{\underset{J_n(\lambda)}{\rightrightarrows}}\mathbb{C}^n\,,\quad
\mathbb{C}^n\overset{J_n(0)^t}{\underset{\mathbb{I}_n}{\rightrightarrows}}\mathbb{C}^n\,,\quad
\mathbb{C}^n\overset{(\mathbb{I}_n\ 0)^t}{\underset{(0\ \mathbb{I}_n)^t}{\rightrightarrows}}\mathbb{C}^{n+1}\,,\quad
\mathbb{C}^{n+1}\overset{(\mathbb{I}_n\ 0)}{\underset{(0\ \mathbb{I}_n)}{\rightrightarrows}}\mathbb{C}^n\,,
\end{equation*}
where $J_n(\lambda)$ is the $n$-dimensional Jordan matrix with eigenvalue $\lambda\in\mathbb{C}$. The first three representations correspond to injective Kronecker complexes whose cokernels are, respectively, a torsion sheaf with length $n$ support at the point $[-\lambda:1]$, a torsion sheaf with length $n$ support at $[1:0]$ and the line bundle $\mathcal{O}_{\mathbb{P}^1}(k+n)$. The last representation gives a Kronecker complex which is not in $\mathcal{C}$.\\
Now take any $\mathcal{F}\in\Coh_{\mathcal{O}_{\mathbb{P}^1}}$ and choose $k\in\mathbb{Z}$ such that the minimum HN phase of $\mathcal{F}$ is at least $\phi_k=\arg(-k+i)/\pi$. If $\mathcal{F}=\oplus_i\mathcal{F}_i$ is the decomposition of $\mathcal{F}$ in indecomposables, then every $\mathcal{F}_i$ has HN phases $\geq\phi_k$, so $\mathcal{F}_i\in\mathcal{T}^Z_{\geq\phi_k}$, and then it is also an indecomposable object in $\mathcal{K}_k$, which means that it is isomorphic to one of the three sheaves listed above.
\end{proof}

\subsection{Moduli spaces}
Fix $k\in\mathbb{Z}$ and a class $v\in\mathpzc{R}_k$ (see Eq.\ \eqref{Eq: region Rk on P1}). By Proposition \ref{Prop: semist sh-Kk on P1} (see also \S\ref{Families obj der cat}), the moduli spaces $\M^\textup{ss}_{\mathbb{P}^1}(v)$ and $\M^\textup{ss}_{K_2,\theta_{\M,v}}(d^v)=K(2;d^v_{-1},d^v_0)$ are isomorphic, as well as the subspaces of stable objects:
\begin{equation*}
\M^\textup{ss}_{\mathbb{P}^1}(v)\simeq K(2;d^v_{-1},d^v_0)\,,\ \ \ \M^\textup{st}_{\mathbb{P}^1}(v)\simeq K_\textup{st}(2;d^v_{-1},d^v_0)\,.
\end{equation*}
In this subsection we will describe explicitly these moduli spaces for all values of $v\in\mathpzc{R}_k$, that is for all $d\in\mathbb{Z}^{\{-1,0\}}$ with $d_0\geq d_{-1}\geq 0$.

First of all, as already mentioned in Ex.\ \ref{Ex: Kron modules}, we have:

\begin{lemma}\cite[Prop.\ 21-22]{Drez87Fibr}
There are isomorphisms
\begin{equation*}
\begin{array}{c}
K(2;d_{-1},d_0)\simeq K(2;2d_{-1}-d_0,d_{-1})\simeq K(2;d_0,2d_0-d_{-1})\,,\\
K(2;d_{-1},d_0)\simeq K(2;d_0,d_{-1})\,,
\end{array}
\end{equation*}
restricting to isomorphisms of the stable loci.
\end{lemma}

We can visualize these isomorphisms as follows: consider the linear transformation $M = \left( \begin{smallmatrix} 2&-1\\ 1&0 \end{smallmatrix} \right)$ acting in the $(d_{-1},d_0)$ plane; the orbits of $M$ are on lines of slope 1. The region $\mathpzc{R}=\{d_0\geq d_{-1}>0\}$ and the diagonal $d_{-1}=d_0$ are invariant under $M$. The lemma says that integral points in $\mathpzc{R}$ lying in the same $M$-orbit, as well as symmetric points with respect to the diagonal $d_{-1}=d_0$, give isomorphic moduli spaces.

Thus it is enough to consider the wedge $d_0\geq 2d_{-1}$ and the diagonal $d_{-1}=d_0$. We start by analyzing the diagonal:

\begin{lemma}\label{Lem: isom Kron=Pd}
$K(2;1,1)=K_\textup{st}(2;1,1)\simeq\mathbb{P}^1$ and $K(2;m,m)\simeq\mathbb{P}^m$, $K_\textup{st}(2;m,m)=\emptyset$ for $m\geq 2$.
\end{lemma}

Given a Kronecker module $f\in\Hom_\mathbb{C}(V_{-1}\otimes Z,V_0)$, we will often use the notation $f_z:=f(\cdot\otimes z)\in\Hom_\mathbb{C}(V_{-1},V_0)$ for $z\in Z$, and $f_j:=f_{e_j}$ for $j=0,1$ (here $\{e_0,e_1	\}$ is the basis of $Z$ that we fixed from the beginning); the index $j$ is tacitly summed when repeated.

\begin{proof}
Clearly, $f:\mathbb{C}\otimes Z\to\mathbb{C}$ is semistable if and only if it is stable if and only if $f\neq 0$; we can identify thus $R^\textup{ss}\cong\mathbb{C}^2\setminus\{0\}$; $PG_V\simeq\mathbb{C}^\times$ acts by scalar multiplication, hence the quotient is $\mathbb{P}^1$.\\
Now let $m\geq 2$: first observe that the semistable and stable loci in the representation space $R=R_{\mathbb{C}^m\oplus\mathbb{C}^m}=\Hom_\mathbb{C}(\mathbb{C}^m\otimes Z,\mathbb{C}^m)$ are
\begin{equation*}
R^\textup{ss}=\left\{f\in R\ \bigg|\ \max_{z\in Z}\rk f_z=m\right\}\,,\ \ \ R^\textup{st}=\emptyset\,.
\end{equation*}
Indeed, the set $U:=\left\{f\in R\ |\ \max\rk f_z=m\right\}$ is open and $G_V$-invariant, and it is contained in $R^\textup{ss}$ because there are no subrepresentations of dimension $(d'_{-1},d'_0)$ with $d'_{-1}>d'_0$, being the generic $f_z$ an isomorphism. Moreover, any polystable representation $f$ can be written, up to isomorphism, as $f_j=\diag(a^1_j,\cdots,a^m_j)$ for $[a^1],...,[a^m]\in\mathbb{P}^1_\mathbb{C}$ (unique up to permutations), and thus $f_z=z^jf_j$ has nonvanishing determinant for general $z\in Z$. So $U$ contains the polystable locus $R^\textup{ps}$, which implies that $U=R^\textup{ss}$.\\
Now let $\mathbb{C}[Z]_m$ be the vector space of homogeneous polynomial functions $Z\to\mathbb{C}$ of degree $m$, and $\mathbb{P}_\mathbb{C}(\mathbb{C}[Z]_m)$ be the projective space of lines in it. We can consider the $G_V$-invariant morphism $\phi:R^\textup{ss}\to\mathbb{P}_\mathbb{C}(\mathbb{C}[Z]_m)$ sending a module $f$ to the class of the polynomial function $z\mapsto\det f_z$. $\phi$ sends the polystable representation $f_j=\diag(a^1_j,\cdots,a^m_j)$ to the class $[\prod_{\ell=1}^m a^\ell_je^{*j}]$, where $\{e^{*0},e^{*1}\}$ is the dual basis of $\{e_0,e_1\}$; this shows that $\phi$ maps non-isomorphic polystable representations to distinct classes and that $\phi$ is surjective, because every element $h\in\mathbb{C}[Z]_m$ can be factored as $h=\prod_{\ell=1}^m a^\ell_je^{*j}$. This is enough to conclude that $\phi$ is the categorical quotient map.
\end{proof}

Now we turn to dimension vectors with $d_0\geq 2d_{-1}$:

\begin{lemma}
Let $d_1,d_0$ be positive integers:
	\begin{enumerate}
	\item If $d_0>2d_{-1}$, then $K(2;d_{-1},d_0)=\emptyset$;
	\item $K(2;1,2)=K_\textup{st}(2;1,2)=\point$;
	\item for $m>1$ we have $K(2;m,2m)=\point$ and $K_\textup{st}(2;m,2m)=\emptyset$.
	\end{enumerate}
\end{lemma}

\begin{proof}
Take a Kronecker module $f:V_{-1}\otimes Z\to V_0$ of dimension vector $(d_{-1},d_0)$, and consider the submodule $W$ with $W_{-1}=V_{-1}$ and $W_0=\im f_0+\im f_1$. If $d_0>2d_{-1}$, then $f$ is always unstable, as $W$ destabilizes it.\\
So we assume now on that $d_0=2d_{-1}$. $f$ is semistable if and only if $f_0,f_1$ are both injective and have complementary images, that is $V_0=\im f_1\oplus\im f_2$ (otherwise $W$ is again destabilizing); moreover, all semistable representations are always isomorphic, as each of them is completely described by the images $f_j(e_k)$ of the basis vectors of $V_{-1}$, and these form a basis of $V_0$. Finally, a semistable $f$ is stable if and only if $(d_{-1},d_0)=(1,2)$.
\end{proof}

Collecting the last three lemmas we can now describe explicitly all the moduli spaces $K(2;d_{-1},d_0)$: for $p,q\in\mathbb{Z}$ we define lines $\ell_p,r_q$ in the $(d_{-1},d_0)$-plane as follows: $\ell_p$ is the line $\{pd_0=(p+1)d_{-1}\}$ if $p>0$, the diagonal $\{d_0=d_{-1}\}$ for $p=0$ and the line $\{pd_{-1}=(p-1)d_0\}$ for $p<0$, while $r_q:=\{d_0=d_{-1}+q\}$.

\begin{theorem}
We assume that $d_{-1}>0$ and $d_0>0$.
	\begin{enumerate}
	\item $K(2;d_{-1},d_0)$ is empty unless $(d_{-1},d_0)$ lies on a line $\ell_p$;
	\item if $(d_{-1},d_0)\in\ell_p\cap r_q$ for some $p,q\in\mathbb{Z}$ with $q\neq 0,\pm 1$, then $K(2;d_{-1},d_0)=\point$, while $K_\textup{st}(2;d_{-1},d_0)=\emptyset$;
	\item if there is some $p\in\mathbb{Z}$ such that $(d_{-1},d_0)\in\ell_p\cap r_1$ or $(d_{-1},d_0)\in\ell_p\cap r_{-1}$, then $K(2;d_{-1},d_0)=K_\textup{st}(2;d_{-1},d_0)=\point$;
	\item if $(d_{-1},d_0)\in\ell_0=r_0$, then $K(2;d_{-1},d_0)\simeq\mathbb{P}^{d_0}$; moreover $K_\textup{st}(2;1,1)\simeq\mathbb{P}^1$, while $K_\textup{st}(2;m,m)=\emptyset$ for $m\geq 2$.
	\end{enumerate}
\end{theorem}

\begin{figure}
\centering
\includegraphics[height=7cm]{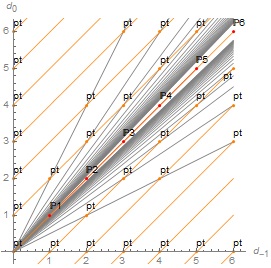}
\caption{The moduli spaces $K(2;d_{-1},d_0)$ for all values of $d_{-1},d_0\in\mathbb{N}$.}
\label{Fig: mod spa K(2,v)}
\end{figure}

This is summarized in Figure \ref{Fig: mod spa K(2,v)}. Now we can translate this into a classification of moduli of sheaves on $\mathbb{P}^1$ (depicted in Figure \ref{Fig: mod spa P1}):

\begin{corollary}
Fix $v\in K_0(\mathbb{P}^1)$.
	\begin{enumerate}
	\item Suppose that $\rk v>0$ and $\deg v$ is a multiple of $\rk v$; then $\M^\textup{ss}_{\mathbb{P}^1}(v)$ is a point, while $\M^\textup{st}_{\mathbb{P}^1}(v)$ is a point if $\rk v=1$ and empty otherwise;
	\item if $\rk v=0$ and $\deg v\geq 0$, then $\M^\textup{ss}_{\mathbb{P}^1}(v)\simeq\mathbb{P}^{\deg v}$; moreover $\M^\textup{st}_{\mathbb{P}^1}(v)\simeq\mathbb{P}^1$ for $\deg v=1$, while $\M^\textup{st}_{\mathbb{P}^1}(v)=\emptyset$ if $\deg v\geq 2$;
	\item in all the other cases $\M^\textup{ss}_{\mathbb{P}^1}(v)$ is empty.
	\end{enumerate}
\end{corollary}

\begin{proof}
Choose $k\in\mathbb{Z}$ so that $v\in\mathpzc{R}_k$. The statements immediately follow from the theorem, by noticing that the transformation \eqref{Eq: coord trans Gr grp P1} maps the lines $\ell_p$ with $p>0$ and the lines $r_q$ respectively to the lines $(p+k)\rk v=\deg v$ and the horizontal lines $\rk v=q$ in the $(-\deg v,\rk v)$ plane.
\end{proof}

\begin{figure}
\centering
\includegraphics[height=6cm]{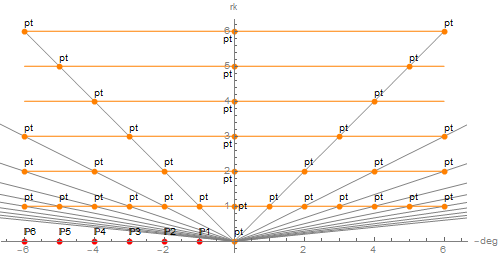}
\caption{The moduli spaces $\M^\textup{ss}_{\mathbb{P}^1}(v)$ for all values of $v\in K_0(\mathbb{P}^1)$ with $\rk v\geq 0$.}
\label{Fig: mod spa P1}
\end{figure}

\begin{remark}
The statements of the corollary can be easily explained in sheaf-theoretic terms via Birkhoff-Grothendieck Theorem:
	\begin{enumerate}
	\item A semistable sheaf of rank $r>0$ must be a direct sum of $r$ copies of the same line bundle $\mathcal{O}_{\mathbb{P}^1}(\ell)$, so it has degree $r\ell$; it is stable if and only if $r=1$.
	\item The polystable sheaves of rank 0 and degree $d$ are direct sums $\mathcal{O}_{x_1}\oplus\cdots\oplus\mathcal{O}_{x_d}$ of skyscraper sheaves and as such they are in 1-1 correspondence with points of the $d$th symmetric product $\mathbb{P}^d$ of $\mathbb{P}^1$; in particular, they can be stable if and only if $d=1$. The structure sheaves of fat points are also semistable, but they degenerate to direct sums of skyscraper sheaves on the reduced points of their support.
	\end{enumerate}
\end{remark}

%% file: Surfaces.tex

%
%
\section{Gieseker stability and quiver stability on surfaces}\label{Gies-quiv stab surf}

In this section we discuss how to relate Gieseker-semistable sheaves on a surface $X$ whose bounded derived category satisfies certain assumptions to semistable representations of a quiver associated to $X$. The idea is analogous to what we saw in the previous section, but the situation becomes now more involved and requires a different analysis.

Let $X$ be a smooth irreducible projective complex surface with an ample divisor $A$.

First of all we assume that $X$ has a strong full exceptional collection ${}^\vee\!\mathfrak{E}=({}^\vee\!E_n,...,{}^\vee\!E_0)$ of vector bundles, so that by Theorem \ref{Thm: Baer-Bondal} we get an equivalence (for convenience we include now a shift)
\begin{equation}\label{Eq: equiv Psi for surf}
\Psi:=\Phi_{{}^\vee\!\mathfrak{E}}[1]:D^b(X)\longrightarrow D^b(Q;J)\,.
\end{equation}
Recall that $\Psi$ maps a complex $\mathcal{F}^\bullet$ in $D^b(X)$ to a complex of representations of $Q$ given, at a vertex $i\in\{0,...,n\}$ of $Q$, by the graded vector space
\begin{equation}\label{Eq: expl form equiv Psi}
R\Hom({}^\vee\!E_i,\mathcal{F}^\bullet)[1]\,.
\end{equation}
$\Psi$ induces in particular an isomorphism $\psi:K_0(X)\to K_0(Q;J)$, and a t-structure on $D^b(X)$ whose heart is denoted by $\mathcal{K}:=\Psi^{-1}(\Rep^\textup{fd}_\mathbb{C}(Q;J))$ and equals the extension closure of the objects $E_i[n-i-1]$, where $\mathfrak{E}=(E_0,...,E_n)$ is the right dual collection to ${}^\vee\!\mathfrak{E}$ (see Remark \ref{Rmk: heart ind by Phi}, but recall that now $\mathcal{K}$ is also shifted by one place to the right).


The polynomial-valued alternating form $\sigma_{\G}=t\sigma_{\M}+\sigma_\chi$ on $K_0(X)$, defined in Eq.\ \eqref{Eq: sigmaG}, reproduces Gieseker stability when regarded as a stability structure on the standard heart $\mathcal{C}=\Coh_{\mathcal{O}_X}$ (Lemma \ref{Lem: Gies,slope-stab and sigma}). On the other hand, if we see $\sigma_{\G}$ as a stability structure on $\mathcal{K}$, then an object $K_V\in\mathcal{K}$ in a class $v\in K_0(X)$ and corresponding via $\Psi$ to a representation $V\in\Rep^\textup{fd}_\mathbb{C}(Q;J)$ is $\sigma_{\G}$-(semi)stable if and only if $V$ is $\theta_{\G,v}$-(semi)stable in the sense of Def.\ \ref{Defn: theta-stab qui rep}, where the polynomial array
\begin{equation*}
\theta_{\G,v}=t\theta_{\M,v}+\theta_{\chi,v}\in\mathbb{Z}[t]^I
\end{equation*}
is defined by (the dot denotes the standard scalar product in $\mathbb{Z}^I$)
\begin{equation}\label{Eq: thetaM,thetachi}
\nu_{\M,v}(w)=\sigma_{\M}(v,w)=\theta_{\M,v}\cdot\dimvec\psi(w)\,,\quad
\nu_{\chi,v}(w)=\sigma_\chi(v,w)=\theta_{\chi,v}\cdot\dimvec\psi(w)\,.
\end{equation}

\begin{figure}
\centering
\includegraphics[height=3cm]{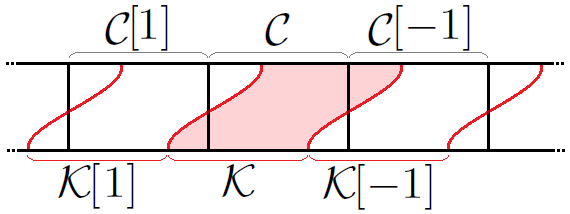}
\caption{The hearts $\mathcal{C},\mathcal{K}\subset D^b(X)$}
\label{Fig: hearts C,K surf}
\end{figure}

Unlike what happened for $\mathbb{P}^1$, now $\mathcal{K}$ is not obtained as a tilt of the standard heart $\mathcal{C}$ with respect to a stability condition (it never satisfies Eq.\ \eqref{Eq: rel tilt t-str} because it intersects three shifts of $\mathcal{C}$, see Figure \ref{Fig: hearts C,K surf}). So there seems to be no reason to expect a priori any relation between stability on one heart and on the other. Nevertheless, we will see that under certain hypotheses this kind of compatibility exists; more precisely, we discuss when the hearts $\mathcal{C},\mathcal{K}$ are $(\sigma_{\G},v)$-compatible in the sense of Def.\ \ref{Defn: (sigma,v)-compat hearts}. Doing this requires the following extra hypotheses on the collections $\mathfrak{E},{}^\vee\!\mathfrak{E}$, which will be always assumed in this section:

\begin{definition}\label{Defn: mod-fr exc coll}
The strong exceptional sequence $\mathfrak{E}$ will be called \emph{monad-friendly} (with respect to the ample divisor $A$) if the following assumptions are satisfied:
	\begin{itemize}\label{Assumptions A1-A2}
	\item[(A1)] the objects ${}^\vee\!E_i$ are locally free sheaves which are Gieseker-semistable with respect to $A$; 
	\item[(A2)] every element of $\mathcal{K}$ is isomorphic to a complex $K_V$ of locally free sheaves concentrated in degrees $-1,0,1$.
	\end{itemize}
By analogy with the situation discussed in the Introduction, we will call such complexes $K_V$ \emph{Kronecker complexes}.
\end{definition}
	
\begin{remark}
In the specific cases that will be examined in the next sections, assumption (A2) will follow from the fact that the objects $E_i[n-i-1]$ generating $\mathcal{K}$ turn out to group into three \emph{blocks}, where the objects in each block are orthogonal to each other, and they are all isomorphic to vector bundles $\tilde{E}_i$ shifted to degree $-1,0$ or $1$, depending on the block, and with vanishing positive Ext spaces between them. Because of this, the complex $K_V\in\mathcal{K}$ corresponding to some representation $V$ of $(Q,J)$ consists, in each degree $\ell=-1,0,1$, of a direct sum of vector bundles of the form $V_i\otimes\tilde{E}_i$. This means in particular that we can write down explicitly the cohomological functors $H_\mathcal{K}^\ell$ of the non-standard t-structure as functors mapping a complex $\mathcal{F}^\bullet\in D^b(X)$ to a complex $K_V\in\mathcal{K}$ with
\begin{equation}\label{Eq: cohom funct HK expl}
V_i=R^{\ell+1}\Hom({}^\vee\!E_i,\mathcal{F}^\bullet)\,.
\end{equation}
\end{remark}

\subsection{Condition (C1)}\label{Cond C1}
In this subsection we will study condition (C1) of \S\ref{Compat hearts stab}. First of all, we want to show that a semistable sheaf in a class $v$ also belongs to the heart $\mathcal{K}$, that is, it is isomorphic to the middle cohomology of a certain monad $K_V$ (recall that a \emph{monad} is a complex with zero cohomology in degrees $\ell\neq 0$). This amounts to checking the vanishing on $\mathcal{F}$ of the cohomological functors $H_\mathcal{K}^\ell$ for $\ell\neq 0$, which in turn reduces, by Eq.\ \eqref{Eq: expl form equiv Psi}, to verifying the vanishing of some $\Ext$ spaces. For this to work we need to choose $v$ appropriately: we denote by $\mathpzc{R}_A\subset\mathpzc{R}^\textup{G}_A\subset K_0(X)$ the regions
\begin{equation}\label{Eq: sets RA,RGA}
\begin{array}{rcl}
\mathpzc{R}_A&:=&\{v\in K_0(X)\ |\ \rk v>0\,,\ \max_i\mu_A({}^\vee\!E_i\otimes\omega_X)<\mu_A(v)<\min_i\mu_A({}^\vee\!E_i)\}\,,\\
\mathpzc{R}^\textup{G}_A&:=&\{v\in K_0(X)\ |\ \rk v>0\,,\ \max^\textup{G}_iP_{{}^\vee\!E_i\otimes\omega_X,A}\prG P_{v,A}\prG \min^\textup{G}_iP_{{}^\vee\!E_i,A}\}\\
&=&\{v\in K_0(X)\ |\ \rk v>0\,,\ \max_ip_{{}^\vee\!E_i\otimes\omega_X,A}<p_{v,A}<\min_ip_{{}^\vee\!E_i,A}\}
\end{array}
\end{equation}
(recall that $p_{v,A}$ denotes the \emph{reduced} Hilbert polynomial). For these regions to be nonempty, the exceptional sheaves ${}^\vee\!E_i$ must have their slopes concentrated in a sufficiently narrow region, and the anticanonical bundle $\omega_X^\vee$ must be sufficiently positive (Figure \ref{Fig: slop exc sh}).

\begin{figure}
\centering
\includegraphics[height=6cm]{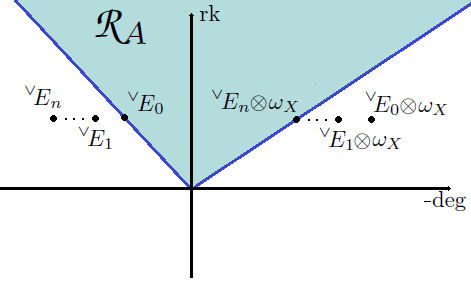}
\caption{The region $\mathpzc{R}_A$.}
\label{Fig: slop exc sh}
\end{figure}

\begin{remark}
We could also twist the collection $\mathfrak{E}$ by a line bundle, to shift the regions $\mathpzc{R}_A,\mathpzc{R}^\textup{G}_A$ accordingly: if these are wide enough and the line bundle has small but nonzero degree, then with such twists we can cover the whole region $\rk v>0$. When this is the case, like in the examples that we will consider, we are thus free to start with any class $v\in K_0(X)$ with positive rank, provided that we choose $\mathfrak{E}$ appropriately.
\end{remark}

\begin{lemma}\label{Lem: sem shvs are in K}
Suppose that $v\in\mathpzc{R}^\textup{G}_A$ (resp.\ $v\in\mathpzc{R}_A$). Then any Gieseker-semistable (resp.\ slope-semistable) sheaf $\mathcal{F}\in v$ belongs to the heart $\mathcal{K}$.
\end{lemma}

\begin{proof}
Since each ${}^\vee\!E_i$ is Gieseker-semistable by assumption $(A1)$, we have $\Hom({}^\vee\!E_i,\mathcal{F})=0$ because of the i\-ne\-qua\-li\-ty $p_{\mathcal{F},A}<p_{{}^\vee\!E_i,A}$; on the other hand, the inequality $p_{{}^\vee\!E_i\otimes\omega_X,A}<p_{\mathcal{F},A}$ and Serre duality give $\Ext^2({}^\vee\!E_i,\mathcal{F})=\Hom(\mathcal{F},{}^\vee\!E_i\otimes\omega_X)=0$. So $H^{-1}_\mathcal{K}(\mathcal{F})=H^1_\mathcal{K}(\mathcal{F})=0$ by Eq.\ \eqref{Eq: expl form equiv Psi}. For the case of slope semistability the proof is the same.
\end{proof}

Now we deal with the same problem with the two hearts $\mathcal{C},\mathcal{K}$ exchanged: we want a $\sigma_{\G}$-semistable Kronecker complex $K_V\in\mathcal{K}$ of class $v$ to be a monad, that is to belong to $\mathcal{C}$. To obtain this, we observe that when $K_V$ is not a monad, we can construct a destabilizing subcomplex or quotient complex using the following idea from \cite[\S2]{FiGiIoKu16Int}: consider the skyscraper sheaf $\mathcal{O}_x$ over some point $x\in X$. Clearly $H_\mathcal{K}^\ell(\mathcal{O}_x)=0$ for all $\ell\neq -1$, which means that there is a Kronecker complex $K_x\in\mathcal{K}$ which has cohomology $\mathcal{O}_x$ in degree 1, and zero elsewhere, that is to say that $K_x\simeq\mathcal{O}_x[-1]$ in $D^b(X)$. Observe that this complex is self-dual: we have
\begin{equation*}
K_x^\vee\simeq\mathcal{O}_x^\vee[1]\simeq\mathcal{O}_x[-1]\simeq K_x
\end{equation*}
in $D^b(X)$, where $\mathcal{O}_x^\vee\simeq\mathcal{O}_x[-2]$ is the derived dual of $\mathcal{O}_x$.

\begin{proposition}\label{Prop: morphs to Vx}
If the second map in a Kronecker complex $K_V$ is not surjective at some point $x\in X$, then there is a nonzero morphism $K_V\to K_x$. If the first map in $K_V$ is not injective at $x$, then there is a nonzero morphism $K_x\to K_V$.
\end{proposition}

\begin{proof}
Suppose that the second map $b:K_V^0\to K_V^1$ in $K_V$ is not surjective at some $x\in X$: we have then a surjective morphism $c:K_V^1\to\mathcal{O}_x$ such that $c\circ b=0$, and this gives a cochain map $K_V\to\mathcal{O}_x[-1]$, and thus a nonzero morphism $K_V\to K_x$ in $\mathcal{K}$.\\
Now suppose that the first map is not injective at $x$: then we can apply the previous argument to the complex $K_V^\vee$ to get a nonzero map $K_V^\vee\to K_x$, hence a nonzero $K_x\simeq K_x^\vee\to K_V$.
\end{proof}

In following two lemmas we prove that for any Kronecker complex $K_V\in\mathcal{K}$ of class $v$ and any $\sigma_{\G}$-maximal subobject $K_W\subset K_V$ in $\mathcal{K}$ (see Def.\ \ref{Defn: sigma-max sub}), we have some vanishings in the cohomologies of $K_W$ and $K_V/K_W$, provided that the class $v\in K_0(X)$ chosen satisfies some constraints imposed by the complex $K_x$. Notice that when $K_V$ is $\sigma_{\G}$-semistable, then it is a $\sigma_{\G}$-maximal subobject of itself, and thus $K_V$ will turn out in Cor.\ \ref{Cor: ss KV monad} to be a monad.

\begin{lemma}\label{Lem: max KW has H1=0}
Take $K_V\in\mathcal{K}$ of class $v$. Suppose that for any $x\in X$ and any nonzero subobject $S\subset K_x$ in $\mathcal{K}$ we have $\nu_{\G,v}(S):=\sigma_{\G}(v,S)>0$. Then any $\sigma_{\G}$-maximal subobject $K_W\subset K_V$ satisfies $H_\mathcal{C}^1(K_W)=0$.
\end{lemma}

\begin{proof}
If $H^1_\mathcal{C}(K_W)\neq 0$, which means that the second map in $K_W$ is not surjective at some point $x\in X$, then there is a nonzero morphism $f:K_W\to K_x$ by Prop.\ \ref{Prop: morphs to Vx}. So we have $\nu_{\G,v}(K_W)=\nu_{\G,v}(\ker f)+\nu_{\G,v}(\im f)$ and, by hypothesis, $\nu_{\G,v}(\im f)>0$. If $\ker f=0$ then $\nu_{\G,v}(K_W)>0$, while if $\ker f\neq0$ then $\nu_{\G,v}(\ker f)=\nu_{\G,v}(K_W)-\nu_{\G,v}(\im f)<\nu_{\G,v}(K_W)$; in both cases, $\sigma_{\G}(K_W,K_V)=-\nu_{\G,v}(K_W)$ is not maximal.
\end{proof}

\begin{lemma}\label{Lem: min KU has H-1=0}
Take $K_V\in\mathcal{K}$ of class $v$. Suppose that, for any $x\in X$, $K_x$ is $\nu_{\M,v}$-semistable and every quotient $Q$ of $K_x$ with $\nu_{\M,v}(Q)=0$ satisfies $H^{-1}_\mathcal{C}(Q)=0$. Then for any $\sigma_{\M}$-maximal subobject $K_W\subset K_V$ we have $H_\mathcal{C}^{-1}(K_V/K_W)=0$.
\end{lemma}
Notice that if $K_W\subset K_V$ is $\sigma_{\G}$-maximal then it is also $\sigma_{\M}$-maximal. It is also worth mentioning here that $\nu_{\M,v}(K_x)=0$ and $\nu_{\G,v}(K_x)=\nu_{\chi,v}(K_x)=\rk v$.

\begin{proof}
Let $K_W\subset K_V$ be a $\sigma_{\M}$-maximal subobject, which means that the quotient $K_U:=K_V/K_W$ maximizes $\nu_{\M,v}=\sigma_{\M}(v,\cdot)$. We have to prove that the first map in $K_U$ is injective. This is clearly true if such a map is injective at every point of $X$; thus suppose now that it is not injective at some point $x\in X$, so that we have a nonzero morphism $g:K_x\to K_U$ by Prop.\ \ref{Prop: morphs to Vx}. Now $\nu_{\M,v}(K_U)=\nu_{\M,v}(K_x/\ker g)+\nu_{\M,v}(\coker g)$, and $\nu_{\M,v}(K_x/\ker g)\leq 0$ by hypothesis.

If $\coker g=0$, then $0\leq\nu_{\M,v}(K_U)=\nu_{\M,v}(K_x/\ker g)\leq 0$, implying $H_\mathcal{C}^{-1}(K_U)=0$. On the other hand, if the cokernel
\begin{equation*}
K_U\overset{c^{(0)}}{\to}K_U^{(1)}:=\coker g
\end{equation*}
is nonzero, then looking at the exact sequence
\begin{equation*}
0\to\ker g\to K_x\overset{g}{\to}K_U\overset{c^{(0)}}{\to}K_U^{(1)}\to 0
\end{equation*}
we see that $\nu_{\M,v}(K_x/\ker g)=0$ (otherwise $\nu_{\M,v}(K_U)<\nu_{\M,v}(\coker g)$, contradicting maximality of $\nu_{\M,v}(K_U)$), so that $\nu_{\M,v}(K_U^{(1)})=\nu_{\M,v}(K_U)$. Hence $K_U^{(1)}$ is also a quotient of $K_V$ of maximal $\nu_{\M,v}$, and $H^{-1}_\mathcal{C}(\ker c^{(0)})=H^{-1}_\mathcal{C}(K_x/\ker g)=0$.

By applying the same argument to $K_U^{(1)}$ we see that either we can immediately conclude that $H_\mathcal{C}^{-1}(K_U^{(1)})=0$, in which case we stop here, or we can construct a further quotient
\begin{equation*}
K_U^{(1)}\overset{c^{(1)}}{\to}K_U^{(2)}
\end{equation*}
with maximal $\nu_{\M,v}$ and such that $H^{-1}_\mathcal{C}(\ker c^{(1)})=0$. After finitely many steps ($\mathcal{K}$ has finite lenght) we end up with a chain
\begin{equation*}
K_U=K_U^{(0)}\overset{c^{(0)}}{\to}K_U^{(1)}\overset{c^{(1)}}{\to}K_U^{(2)}\overset{c^{(2)}}{\to}\cdots\overset{c^{(\ell-1)}}{\to} K_U^{(\ell)}
\end{equation*}
of surjections with $H^{-1}_\mathcal{C}(\ker c^{(i)})=0$ for all $i\geq 0$ and $H^{-1}_\mathcal{C}(K_U^{(\ell)})=0$. This implies that $H_\mathcal{C}^{-1}(K_U)=0$.
\end{proof}

\begin{remark}\label{Rmk: srt cond hyp lemmas}
Notice that the hypotheses of Lemmas \ref{Lem: max KW has H1=0} and \ref{Lem: min KU has H-1=0} are verified under the stronger assumptions that $\rk v>0$ and for all $x\in X$, $K_x$ is $\nu_{\M,v}$-stable.\footnote{As for the hypothesis of Lemma \ref{Lem: max KW has H1=0}, note that $\nu_{\G,v}(K_x)=\rk v>0$.}
\end{remark}

It is convenient to gather the conditions on $v$ imposed by the hypotheses of Lemmas \ref{Lem: max KW has H1=0} and \ref{Lem: min KU has H-1=0} or by Remark \ref{Rmk: srt cond hyp lemmas} in the definition of two regions $\mathpzc{S}^\circ_A\subset\mathpzc{S}_A\subset K_0(X)$:
\begin{equation}\label{Eq: sets SA,S0A}
\begin{array}{c}
\mathpzc{S}_A:=\left\{v\in K_0(X)\ \bigg| \begin{array}{l}
\textup{ for any }x\in X\textup{ we have }\nu_{\G,v}(S)>0 \textup{ for any }0\neq S\subset K_x\,, \textup{ and }\\
H^{-1}_\mathcal{C}(Q)=0\textup{ for any quotient }K_x\to Q\textup{ with }\nu_{\M,v}(Q)=0
\end{array}\right\}\,,\\
\ \\
\mathpzc{S}^\circ_A:=\left\{v\in K_0(X)\ | \rk v>0\textup{ and }K_x\textup{ is }\nu_{\M,v}\textup{-stable for all }x\in X\right\}\,.
\end{array}
\end{equation}

Again, in the examples it will be enough to twist the collection $\mathfrak{E}$ by a line bundle to have any $v\in K_0(X)$ of positive rank inside such a region.

\begin{corollary}\label{Cor: ss KV monad}
Take $K_V\in\mathcal{K}$ of class $v\in\mathpzc{S}_A$. If $K_V$ is $\sigma_{\G}$-semistable, then it is a monad, that is $K_V\in\mathcal{C}$.
\end{corollary}

\begin{proof}
If a nonzero $K_V$ is $\sigma_{\G}$-semistable (hence $\nu_{\G,v}$-semistable), then it has minimal $\nu_{\G,v}$ between its subobjects, and maximal $\nu_{\M,v}$ between its quotients. So we can apply Lemmas \ref{Lem: max KW has H1=0} and \ref{Lem: min KU has H-1=0} to deduce that $H_\mathcal{C}^{-1}(K_V)=H_\mathcal{C}^1(K_V)=0$.
\end{proof}

Summing up, Lemma \ref{Lem: sem shvs are in K} and Corollary \ref{Cor: ss KV monad} tell us that:

\begin{proposition}
Assume that $\mathfrak{E}$ is monad-friendly (Def.\ \ref{Defn: mod-fr exc coll}). Then condition (C1) of \S\ref{Compat hearts stab} is verified for Gieseker stability $\sigma_{\G}$, the hearts $\mathcal{C},\mathcal{K}$ and for all $v\in\mathpzc{R}^{\G}_A\cap\mathpzc{S}_A$.
\end{proposition}

\subsection{Condition (C2)}\label{Cond C2}
Now we turn to the analysis of condition (C2) of \S\ref{Compat hearts stab}: we want to show that a monad $K_V\in\mathcal{K}$ of class $v$ is $\sigma_{\G}$-(semi)stable as an object of $\mathcal{K}$ if and only if its middle cohomology is $\sigma_{\G}$-(semi)stable as an object of $\mathcal{C}$, that is, a Gieseker-(semi)stable sheaf. First we prove the ``only if" direction:

\begin{lemma}
Suppose that $v\in\mathpzc{R}^\textup{G}_A$, and let $K_V\in v$ be monad which is a $\sigma_{\G}$-(semi)stable object of $\mathcal{K}$. Then its middle cohomology $H^0_\mathcal{C}(K_V)$ is a Gieseker-(semi)stable sheaf.
\end{lemma}

\begin{proof}
Suppose that $\mathcal{F}:=H^0_\mathcal{C}(K_V)$ is not Gieseker-semistable. Let $\mathcal{F}_1\subset\mathcal{F}$ be the maximally destabilizing subsheaf (i.e.\ the first nonzero term in the HN filtration of $\mathcal{F}$), which is semistable and satisfies $P_{\mathcal{F}_1,A}\suG P_{\mathcal{F},A}\suqG P_{\mathcal{F}/\mathcal{F}_1,A,\max}$ (notation as in Eq.\ \eqref{Eq: min max Hilb pol}). Then, as in the proof of Lemma \ref{Lem: sem shvs are in K}, we deduce that
\begin{equation*}
\Hom({}^\vee\!E_i,\mathcal{F}/\mathcal{F}_1)=0\,,\quad \Ext^2({}^\vee\!E_i,\mathcal{F}_1)=\Hom(\mathcal{F}_1,{}^\vee\!E_i\otimes\omega_X)=0
\end{equation*}
for all $i$. These vanishings mean that $H_\mathcal{K}^\ell(\mathcal{F}/\mathcal{F}_1)=0$ for all $\ell\neq0,1$ and $H_\mathcal{K}^\ell(\mathcal{F}_1)=0$ for all $\ell\neq-1,0$, so we get a long exact sequence
\begin{equation}\label{Eq: ex seq HK(F)}
0\to H^{-1}_\mathcal{K}(\mathcal{F}_1)\to 0\to 0\to H^0_\mathcal{K}(\mathcal{F}_1)\to H^0_\mathcal{K}(\mathcal{F})\to H^0_\mathcal{K}(\mathcal{F}/\mathcal{F}_1)\to 0\to 0\to H^1_\mathcal{K}(\mathcal{F}/\mathcal{F}_1)\to 0\,,
\end{equation}
showing that $H^{-1}_\mathcal{K}(\mathcal{F}_1)=H^1_\mathcal{K}(\mathcal{F}/\mathcal{F}_1)=0$, that is $\mathcal{F},\mathcal{F}/\mathcal{F}_1\in\mathcal{K}$, and the remaining short exact sequence means that $K_V=H^0_\mathcal{K}(\mathcal{F})$ is not $\sigma_{\G}$-semistable.

Finally, if $\mathcal{F}$ is strictly $\sigma_{\G}$-semistable, then we take $\mathcal{F}_1\subsetneq\mathcal{F}$ with $P_{\mathcal{F}_1,A}\equiv_\textup{G}P_{\mathcal{F},A}\equiv_\textup{G}P_{\mathcal{F}/\mathcal{F}_1,A}$ (hence $\mathcal{F}_1$ and $\mathcal{F}/\mathcal{F}_1$ are semistable) and again we get a short exact sequence as in Eq.\ \eqref{Eq: ex seq HK(F)}, showing that $K_V$ is not $\sigma_{\G}$-stable.
\end{proof}

Now we prove the ``if'' direction with a specular argument:

\begin{lemma}
Suppose that $K_V\in\mathcal{K}$ is a monad of class $v\in\mathpzc{S}_A$ whose middle cohomology $H^0_\mathcal{C}(K_V)$ is a Gieseker-(semi)stable sheaf. Then $K_V$ is $\sigma_{\G}$-(semi)stable as an object of $\mathcal{K}$.
\end{lemma}

\begin{proof}
Suppose that $K_V$ is $\sigma_{\G}$-unstable, take a $\sigma_{\G}$-maximal subobject $0\neq K_W\subsetneq K_V$ in $\mathcal{K}$ (this exists as the subobjects of $K_V$ can only belong to finitely many classes in $K_0(X)$) and apply Lemmas \ref{Lem: max KW has H1=0} and \ref{Lem: min KU has H-1=0}, to get the vanishings $H^1_\mathcal{C}(K_W)=H^{-1}_\mathcal{C}(K_V/K_W)=0$ and then an exact sequence
	\begin{equation}\label{Eq: ex seq HC(K)}
	0\to H^{-1}_\mathcal{C}(K_W) \to 0\to 0\to
	H^0_\mathcal{C}(K_W)\to H^0_\mathcal{C}(K_V)\to H^0_\mathcal{C}(\nicefrac{K_V}{K_W})\to
	0\to 0\to H^1_\mathcal{C}(\nicefrac{K_V}{K_W})\to 0
	\end{equation}
showing that $K_W,K_V/K_W\in\mathcal{C}$ and that $\mathcal{F}:=H^0_\mathcal{C}(K_V)$ is also $\sigma_{\G}$-unstable as an object of $\mathcal{C}$. Now suppose that $K_V$ is strictly $\sigma_{\G}$-semistable: we have again a $0\neq K_W\subsetneq K_V$ maximizing $\nu_{\G,v}$, so that the lemmas apply and we end up with a short exact sequence as in \eqref{Eq: ex seq HC(K)}, showing that $\mathcal{F}$ is not $\sigma_{\G}$-stable.
\end{proof}

So we can conclude that:

\begin{proposition}
Assume that $\mathfrak{E}$ is monad-friendly (Def.\ \ref{Defn: mod-fr exc coll}). Then condition (C2) of \S\ref{Compat hearts stab} is verified for Gieseker stability $\sigma_{\G}$, the hearts $\mathcal{C},\mathcal{K}$ and for all $v\in\mathpzc{R}^{\G}_A\cap\mathpzc{S}_A$.
\end{proposition}

\subsection{Conclusions}
We summarize the results of \S\ref{Cond C1} and \S\ref{Cond C2}. We recall that $X$ is a smooth projective surface, $A$ is an ample divisor, and we are supposing that $D^b(X)$ has a full strong exceptional collection ${}^\vee\!\mathfrak{E}$ which is monad-friendly with respect to $A$ (Def.\ \ref{Defn: mod-fr exc coll}). Recall also that ${}^\vee\!\mathfrak{E}$ determines a quiver $Q$ with relations $J$, together with an equivalence $\Psi:D^b(X)\to D^b(Q;J)$ (Eq.\ \eqref{Eq: equiv Psi for surf}), a heart $\mathcal{K}\subset D^b(X)$ and an isomorphism $\psi:K_0(X)\to K_0(Q;J)$. For any class $v\in K_0(X)$ we denote by $d^v:=\dimvec\psi(v)\in\mathbb{Z}^I$ the corresponding dimension vector, and by $\theta_{\G,v}=t\theta_{\M,v}+\theta_{\chi,v}\in\mathbb{Z}[t]^I$ the array of polynomials defined in Eq.\ \eqref{Eq: thetaM,thetachi}.

Now consider the conical region
\begin{equation*}
\tilde{\mathpzc{R}}_{A,\mathfrak{E}}:=\mathpzc{R}^\textup{G}_A\cap\mathpzc{S}_A\subset K_0(X)
\end{equation*}
defined as the intersection of the regions in Equations \eqref{Eq: sets RA,RGA} and \eqref{Eq: sets SA,S0A}.

\begin{theorem}\label{Thm: C1,C2 on surf}
For all $v\in\tilde{\mathpzc{R}}_{A,\mathfrak{E}}$, the hearts $\mathcal{C}$ and $\mathcal{K}$ are $(\sigma_{\G},v)$-compatible (Def.\ \ref{Defn: (sigma,v)-compat hearts}). Thus, $\Psi$ restricts to an equivalence between the category of Gieseker-(semi)stable sheaves of class $v$ on $X$ and the category of $d^v$-dimensional $\theta_{\G,v}$-(semi)stable representations of $(Q,J)$.
\end{theorem}

As already observed in Remark \ref{Rmk: mod stack comp herts coinc}, this theorem implies that the moduli stack $\mathfrak{M}^\textup{ss}_{X,A}(v)$ of $\sigma_G$-semistable objects in $\mathcal{C}$ with class $v$ coincides with the moduli stack of $\sigma_G$-semistable objects in $\mathcal{K}$ with class $v$, which (recall the discussion of \S\ref{Families obj der cat}) is isomorphic to the quiver moduli stack $\mathfrak{M}^\textup{ss}_{Q,J,\theta_{\G,v}}(d^v)$. Similar arguments apply to the stable loci. Hence:

\begin{corollary}\label{Cor: isom mod st/sp}
For all $v\in\tilde{\mathpzc{R}}_{A,\mathfrak{E}}$ we have isomorphisms
\begin{equation*}
\mathfrak{M}^\textup{ss}_{X,A}(v)\simeq\mathfrak{M}^\textup{ss}_{Q,J,\theta_{\G,v}}(d^v)\quad\textup{and}\quad
\mathfrak{M}^\textup{st}_{X,A}(v)\simeq\mathfrak{M}^\textup{st}_{Q,J,\theta_{\G,v}}(d^v)\,.
\end{equation*}
In particular, we have isomorphisms
\begin{equation*}
\M^\textup{ss}_{X,A}(v)\simeq\M^\textup{ss}_{Q,J,\theta_{\G,v}}(d^v)\quad\textup{and}\quad
\M^\textup{st}_{X,A}(v)\simeq\M^\textup{st}_{Q,J,\theta_{\G,v}}(d^v)
\end{equation*}
between the coarse moduli spaces.
\end{corollary}

Recall that the construction of the moduli space $\M^\textup{ss}_{Q,J,\theta_{\G,v}}(d^v)$ for a polynomial array $\theta_{\G,v}\in\mathbb{Z}[t]^I$ was explained in \S\ref{Quiver moduli}, just before Example \ref{Ex: Kron modules}.

%% file: P2.tex

%
%
\section{Application to \texorpdfstring{$\mathbb{P}^2$}{P2}}\label{Appl to P2}
In this section we apply the previous results taking $X$ to be the complex projective plane $\mathbb{P}^2=\mathbb{P}_\mathbb{C}(Z)$, where $Z$ is a 3-dimensional $\mathbb{C}$-vector space. We choose the ample divisor as $A=H$, the divisor of a line, and we write $\deg:=\deg_H$ for simplicity.

Take $v\in K_0(\mathbb{P}^2)=K_{\num}(\mathbb{P}^2)$. By the Hirzebruch-Riemann-Roch formula \eqref{Eq: Hilb pol sh sur} we have
\begin{equation*}
P_{v,H}(t)=(t^2+3t)\frac{\rk v}{2}+t\deg v+\chi(v)\,,
\end{equation*}
with $\chi(v)=P_v(0)=\rk v+(3/2)\deg v+\ch_2 v$. 

\subsection{The first equivalence}
Take, as in Ex.\ \ref{Exs: exc coll}(1), the full strong collections
\begin{equation*}
\begin{array}{c}
\mathfrak{E}=(E_{-1},E_0,E_1)=(\mathcal{O}_{\mathbb{P}^2}(-1),\mathcal{O}_{\mathbb{P}^2},\mathcal{O}_{\mathbb{P}^2}(1))\,,\\
{}^\vee\!\mathfrak{E}=({}^\vee\!E_1,{}^\vee\!E_0,{}^\vee\!E_{-1})=(\mathcal{O}_{\mathbb{P}^2}(1),\uptau_{\mathbb{P}^2},\wedge^2\uptau_{\mathbb{P}^2}(-1))
\end{array}
\end{equation*}
(note that $\wedge^2\uptau_{\mathbb{P}^2}(-1)\simeq\mathcal{O}_{\mathbb{P}^2}(2)$). We apply Theorem \ref{Thm: Baer-Bondal} to the collection ${}^\vee\!\mathfrak{E}$: the tilting sheaf $T=\mathcal{O}_{\mathbb{P}^2}(1)\oplus\uptau_{\mathbb{P}^2}\oplus\wedge^2\uptau_{\mathbb{P}^2}(-1)$ has endomorphism algebra
\begin{equation*}
\End_{\mathcal{O}_{\mathbb{P}^2}}(T)=\begin{pmatrix}
\mathbb{C} &  & \\ 
Z & \mathbb{C} & \\ 
\wedge^2Z & Z & \mathbb{C}
\end{pmatrix}
\end{equation*}
which is identified, after fixing a $\mathbb{C}$-basis $e_0,e_1,e_2$ of $Z$, with the opposite of the bound quiver algebra $\mathbb{C}B_3/J$ of the \emph{Beilinson quiver}
\begin{center}$B_3$: \begin{tikzcd}
-1\arrow[bend left=50]{r}{a_1}\arrow{r}{a_2}\arrow[bend right=50]{r}{a_3} &0\arrow[bend left=50]{r}{b_1}\arrow{r}{b_2}\arrow[bend right=50]{r}{b_3} &1
\end{tikzcd}\end{center}
with quadratic relations $J=(b_ia_j+b_ja_i,\ i,j=1,2,3)$. So we get a triangulated equivalence
\begin{equation*}
\Psi:=\Phi_{{}^\vee\!\mathfrak{E}}[1]:D^b(\mathbb{P}^2)\to D^b(B_3;J)\,.
\end{equation*}
This maps a complex $\mathcal{F}^\bullet\in D^b(\mathbb{P}^2)$ to the complex of representations
\begin{equation*}
R\Hom_{\mathcal{O}_{\mathbb{P}^2}}(\wedge^2\uptau_{\mathbb{P}^2}(-1),\mathcal{F}^\bullet)[1]\triplerightarrow
R\Hom_{\mathcal{O}_{\mathbb{P}^2}}(\uptau_{\mathbb{P}^2},\mathcal{F}^\bullet)[1]\triplerightarrow
R\Hom_{\mathcal{O}_{\mathbb{P}^2}}(\mathcal{O}_{\mathbb{P}^2}(1),\mathcal{F}^\bullet)[1]\,.
\end{equation*}
The standard heart of $D^b(B_3;J)$ is sent to the heart
\begin{equation*}
\mathcal{K}:=\langle\mathcal{O}_{\mathbb{P}^2}(-1)[1],\mathcal{O}_{\mathbb{P}^2},\mathcal{O}_{\mathbb{P}^2}(1)[-1]\rangle_{\ext}
\end{equation*}
whose objects are \emph{Kronecker complexes}
\begin{equation*}
K_V:V_{-1}\otimes\mathcal{O}_{\mathbb{P}^2}(-1)\longrightarrow V_0\otimes\mathcal{O}_{\mathbb{P}^2}\longrightarrow V_1\otimes\mathcal{O}_{\mathbb{P}^2}(1)\,,
\end{equation*}
where the middle sheaf is in degree $0$. Moreover, the objects of ${}^\vee\!\mathfrak{E}$ are semistable bundles. Thus the sequence $\mathfrak{E}$ is monad-friendly with respect to $H$ (Def.\ \ref{Defn: mod-fr exc coll}).

The equivalence $\Psi$ also gives an isomorphism $\psi:K_0(\mathbb{P}^2)\to K_0(B_3;J)$; coordinates on the Grothendieck groups are provided by the isomorphisms
\begin{equation*}
K_0(\mathbb{P}^2)\overset{(\rk,\deg,\chi)}{\longrightarrow}\mathbb{Z}^3\,,\quad
K_0(B_3;J)\overset{\dimvec}{\longrightarrow}\mathbb{Z}^3\,,
\end{equation*}
and we denote by
\begin{equation*}
(d^v_{-1},d^v_0,d^v_1)=d^v:=\dimvec\psi(v)
\end{equation*}
the coordinates of $\psi(v)\in K_0(B_3;J)$ with respect to the basis of simple representations $S(-1),S(0),S(1)$; using the fact that these are mapped to $\mathcal{O}_{\mathbb{P}^2}(-1)[1],\mathcal{O}_{\mathbb{P}^2},\mathcal{O}_{\mathbb{P}^2}(1)[-1]$, we find that the base-change matrices between the two coordinate sets are
\begin{equation}\label{Eq: trans coord K0(P2)}
\begin{pmatrix}
d^v_{-1}\\
d^v_0\\
d^v_1
\end{pmatrix}=
\begin{pmatrix}
1 &2 &-1\\
3 &3 &-2\\
1& 1& -1
\end{pmatrix}
\begin{pmatrix}
\rk v\\
\deg v\\
\chi(v)
\end{pmatrix}\,,\quad
\begin{pmatrix}
\rk v\\
\deg v\\
\chi(v)
\end{pmatrix}=
\begin{pmatrix}
-1& 1& -1\\
1 &0 &-1\\
0 &1 &-3
\end{pmatrix}
\begin{pmatrix}
d^v_{-1}\\
d^v_0\\
d^v_1
\end{pmatrix}\,.
\end{equation}
So, given $v\in K_0(\mathbb{P}^2)$, the arrays $\theta_{\M,v},\theta_{\chi,v}\in\mathbb{Z}^{\{-1,0,1\}}$ associated to the alternating forms $\sigma_{\M},\sigma_\chi$ as in equation \eqref{Eq: thetaM,thetachi} are given by
\begin{equation*}
\begin{array}{c}
\theta_{\M,v}=
\begin{pmatrix}
-\rk v-\deg v\\
\deg v\\
\rk v-\deg v
\end{pmatrix}=
\begin{pmatrix}
-d^v_0+2d^v_1\\
d^v_{-1}-d^v_1\\
-2d^v_{-1}+d^v_0
\end{pmatrix}\,,\\
\theta_{\chi,v}=
\begin{pmatrix}
-\chi(v)\\
-\rk v+\chi(v)\\
3\rk v-\chi(v)
\end{pmatrix}=
\begin{pmatrix}
-d^v_0+3d^v_1\\
d^v_{-1}-2d^v_1\\
-3d^v_{-1}+2d^v_0
\end{pmatrix}\,.
\end{array}
\end{equation*}
The regions $\mathpzc{R}_H,\mathpzc{S}_H^\circ\subset K_0(\mathbb{P}^2)$ of equations \eqref{Eq: sets RA,RGA} and \eqref{Eq: sets SA,S0A} read now
\begin{equation*}
\mathpzc{R}_H=\mathpzc{S}_H^\circ=\{|\deg v|<\rk v\}=\{d^v_0>2d^v_{-1}\textup{ and }d^v_0>2d^v_1\}\,.
\end{equation*}
Indeed, take $x\in\mathbb{P}^2$ and let $p,q\in Z^\vee$ be linear forms whose common zero is $x$, and notice that the Kronecker complex
\begin{equation*}
K_x:\mathcal{O}_{\mathbb{P}^2}(-1)\overset{\binom{p}{q}}{\longrightarrow} \mathbb{C}^2\otimes\mathcal{O}_{\mathbb{P}^2}\overset{(q\ -p)}{\longrightarrow}\mathcal{O}_{\mathbb{P}^2}(1)
\end{equation*}
is quasi-isomorphic to $\mathcal{O}_x[-1]$, and its only nontrivial subcomplexes have dimension vectors $w$ equal to $(0,2,1),(0,1,1)$ and $(0,0,1)$; the inequalities $\theta_{\M,v}\cdot w>0$ give the above formula for $\mathpzc{S}_H^\circ$.

Notice that, after twisting by a line bundle, every sheaf of positive rank can be brought inside the region $\mathpzc{R}_H$. Hence it is enough to consider this region to describe all the moduli spaces $\M^\textup{ss}_{\mathbb{P}^2,H}(v)$ with $\rk v>0$.

We can now apply Corollary \ref{Cor: isom mod st/sp}:

\begin{theorem}\label{Thm: isom sh=qui on P2 (1)}
For any $v\in\mathpzc{R}_H$ we have isomorphisms
\begin{equation*}
\M^\textup{ss}_{\mathbb{P}^2,H}(v)\simeq\M^\textup{ss}_{B_3,J,\theta_{\G,v}}(d^v)\quad\textup{and}\quad
\M^\textup{st}_{\mathbb{P}^2,H}(v)\simeq \M^\textup{st}_{B_3,J,\theta_{\G,v}}(d^v)\,.
\end{equation*}
\end{theorem}

Many of the known properties of $\M^\textup{ss}_{\mathbb{P}^2,H}(v)$ can be recovered from these isomorphisms. We observe for example that:
	\begin{enumerate}
	\item $v\in\mathpzc{R}_H$ is primitive if and only if $\gcd(\rk v,\deg v,\chi(v))=1$. In this case we have that $\M^\textup{ss}_{\mathbb{P}^2,H}(v)=\M^\textup{st}_{\mathbb{P}^2,H}(v)$ and there is a universal family, either by Remark \ref{Rmk: prop mod sp sh surf} or Remark \ref{Rmk: prop quiv mod}.\footnote{Notice that for $v\in\mathpzc{R}_H$ the arrays $\theta_{\M,v},\theta_{\chi,v}$ are linearly independent, so $d^v$ is $\theta_{\G,v}$-coprime by Remark \ref{Rmk: coprime dim vec}.}
	\item By general arguments about monads (see \cite[Lemma 4.1.7]{OSS80Ve} and \cite[Prop.\ 2.3 and 2.6]{DreLeP85Fibr} and note that the fact that a semistable sheaf $\mathcal{E}$ has $\Ext^2(\mathcal{E},\mathcal{E})=0$ is used), the variety $X_J\subset R_{d^v}(B_3)$ cut by the relations $J$ intersects the semistable locus $R_{d^v,\theta_{\G,v}}^\textup{ss}(B_3)$ in a smooth complete intersection. In particular $\M^\textup{st}_{\mathbb{P}^2,H}(v)\simeq\M^\textup{st}_{B_3,J,\theta_{\G,v}}(d^v)$ is smooth and we can compute its dimension as the dimension of the quotient $R_{d^v,\theta_{\G,v}}^\textup{st}(B_3)/PG_{d^v}$ minus the number $6d^v_{-1}d^v_1$ of relations imposed; the result is
	\begin{equation*}
	\dim\M^\textup{st}_{\mathbb{P}^2,H}(v)=1-\rk v^2+\Delta(v)\,,
	\end{equation*}
	in agreement with eq.\ \eqref{Eq: dim ms sh sur}.
	\item If $\theta_{\M,v}^{-1}>0$ or $\theta_{\M,v}^1<0$ then every $d^v$-dimensional representation is $\theta_{\G,v}$-unstable, so $\M^\textup{ss}_{B_3,J,\theta_{\G,v}}(d^v)$ is empty. But for all $v\in\mathpzc{R}_H$ we have
	\begin{equation*}
	\theta_{\M,v}^{-1}=-\rk v-\deg v<0\quad\textup{and} \quad\theta_{\M,v}^1=\rk v-\deg v>0\,.
	\end{equation*}
	\end{enumerate}

Notice also that the existence of a semistable sheaf $\mathcal{F}$ in $v\in\mathpzc{R}_H$ implies that all the components of the array $\dimvec v$ are nonnegative. Thus for example $2\ch_2v=-d^v_{-1}-d^v_1\leq 0$, with the equality holding only when $\mathcal{F}$ is trivial. From this simple observation we can easily deduce the Bogomolov inequality \eqref{Eq: Bogom ineq}:

\begin{proposition}\label{Prop: Bogom ineq from quiv mod}
If $\M^\textup{ss}_{\mathbb{P}^2,H}(v)\neq\emptyset$ for some $v\in K_0(\mathbb{P}^2)$, then
\begin{equation*}
\Delta(v):=(\deg v)^2-2\rk v \ch_2(v)\geq 0\,.
\end{equation*}
\end{proposition}

\begin{proof}
For $\rk v=0$ the statement is obvious. If $\rk v>0$, then after twisting by a line bundle (which does not change the discriminant $\Delta$) we can reduce to the case $v\in\mathpzc{R}_H$: for such $v$ we have just observed that $\ch_2 v\leq0$, and hence $\Delta(v)\geq0$.
\end{proof}

Finally, observe that whether a class $v\in K_0(\mathbb{P}^2)$ belongs to the region $\mathpzc{R}_H$ only depends on the ray generated by the Hilbert polynomial $P_{v,H}$ in $\mathbb{R}[t]_{\leq 2}$. Thus we can extend the equivalence of Thm \ref{Thm: C1,C2 on surf} to whole abelian categories of semistable sheaves with fixed reduced Hilbert polynomial:

\begin{theorem}\label{Thm: equiv ab cat semist sh and quiv rep P2}
If $p\in\mathbb{R}[t]$ is the Hilbert polynomial of a class $v\in\mathpzc{R}_H$, then $\Psi$ restricts to an equivalence between the abelian categories $\mathcal{S}_H(p)$ and $\mathcal{S}_{\theta_{\G,v}}$ (defined in Equations \eqref{Eq: cat semis coh shvs} and \eqref{Eq: cat semis quiv rep}).
\end{theorem}

\begin{proof}
Identify $\mathcal{S}_{\theta_{\G,v}}$ with the category of $\nu_{\textup{G},v}$-semistable Kronecker complexes, via $\Psi$. The inclusion $\mathcal{S}_H(p)\subset\mathcal{S}_{\theta_{\G,v}}$ is clear, as any nonzero $\mathcal{F}\in\mathcal{S}_H(p)$ has class $[\mathcal{F}]\in\mathpzc{R}_H$. For the converse, take a $\theta_{\G,v}$-semistable representation $(V,f)$ of $(B_3,J)$, and let $w:=\psi^{-1}[V,f]\in K_0(\mathbb{P}^2)$. By definition of $\theta_{\G,v}$-stability we have $\sigma_{\G}(v,w)=0$ and then $pP_{w,H}'-p'P_{w,H}=0$, which by Remark \ref{Rmk: pq'=p'q impl prop} implies that $P_{w,H}=\alpha p$ for some $\alpha\in\mathbb{R}$. In fact, $\alpha\neq 0$ since $P_{w,H}\neq 0$ by Eq.\ \eqref{Eq: trans coord K0(P2)}, and $\alpha$ cannot be negative because otherwise we would have $-P_{w,H}\in\mathpzc{R}_H$, and then $\theta_{\M,w}^{-1}>0$ and $\theta_{\M,w}^1<0$, which (as observed above) would contradict the existence of the semistable representation $(V,f)$ in $w$. Thus $w\in\mathpzc{R}_H$, and hence $\Psi^{-1}(V,f)\in\mathcal{S}_H(p)$ by Thm \ref{Thm: C1,C2 on surf}. This concludes the proof that $\mathcal{S}_H(p)\supset\mathcal{S}_{\theta_{\G,v}}$.
\end{proof}

\subsection{The second equivalence}
Now we will use instead the full strong collections
\begin{equation*}
\begin{array}{c}
\mathfrak{E}'=(E'_{-1},E'_0,E'_1)=(\Omhol_{\mathbb{P}^2}^2(2),\Omhol_{\mathbb{P}^2}^1(1),\mathcal{O}_{\mathbb{P}^2})\,,\\
{}^\vee\!\mathfrak{E}'=({}^\vee\!E'_1,{}^\vee\!E'_0,{}^\vee\!E'_{-1})=(\mathcal{O}_{\mathbb{P}^2},\mathcal{O}_{\mathbb{P}^2}(1),\mathcal{O}_{\mathbb{P}^2}(2))\,.
\end{array}
\end{equation*}
The tilting sheaf $T'=\mathcal{O}_{\mathbb{P}^2}\oplus\mathcal{O}_{\mathbb{P}^2}(1)\oplus\mathcal{O}_{\mathbb{P}^2}(2)$ has endomorphism algebra
\begin{equation*}
\End_{\mathcal{O}_{\mathbb{P}^2}}(T')=\begin{pmatrix}
\mathbb{C} &  & \\ 
Z^\vee & \mathbb{C} & \\ 
S^2Z^\vee & Z^\vee & \mathbb{C}
\end{pmatrix}
\end{equation*}
which is identified, after fixing a $\mathbb{C}$-basis $e_0,e_1,e_2$ of $Z$, to the opposite of the bound quiver algebra $\mathbb{C}B_3/J'$, where now $J'=(b_ia_j-b_ja_i,\ i,j=1,2,3)$. The new equivalence
\begin{equation*}
\Psi':=\Phi_{{}^\vee\!\mathfrak{E}'}[1]:D^b(\mathbb{P}^2)\to D^b(B_3;J')
\end{equation*}
sends a complex $\mathcal{F}^\bullet\in D^b(\mathbb{P}^2)$ to the complex of representations
\begin{equation*}
R\Hom_{\mathcal{O}_{\mathbb{P}^2}}(\mathcal{O}_{\mathbb{P}^2}(2),\mathcal{F}^\bullet)[1]\triplerightarrow
R\Hom_{\mathcal{O}_{\mathbb{P}^2}}(\mathcal{O}_{\mathbb{P}^2}(1),\mathcal{F}^\bullet)[1]\triplerightarrow
R\Hom_{\mathcal{O}_{\mathbb{P}^2}}(\mathcal{O}_{\mathbb{P}^2},\mathcal{F}^\bullet)[1]
\end{equation*}
and the standard heart of $D^b(B_3;J')$ is now sent to the heart
\begin{equation*}
\mathcal{K}':=\langle\Omhol_{\mathbb{P}^2}^2(2)[1],\Omhol_{\mathbb{P}^2}^1(1),\mathcal{O}_{\mathbb{P}^2}[-1]\rangle_{\ext}
\end{equation*}
whose objects are complexes
\begin{equation*}
K'_V:V_{-1}\otimes\Omhol_{\mathbb{P}^2}^2(2)\longrightarrow V_0\otimes\Omhol_{\mathbb{P}^2}^1(1)\longrightarrow V_1\otimes\mathcal{O}_{\mathbb{P}^2}
\end{equation*}
with the middle term in degree $0$. These are the Kronecker complexes originally used in \cite{DreLeP85Fibr}, and we see again that $\mathfrak{E}'$ is monad-friendly with respect to $A$ (Def.\ \ref{Defn: mod-fr exc coll}).

$\Psi'$ induces a different isomorphism $\psi':K_0(\mathbb{P}^2)\to K_0(B_3;J')$. Given $v\in K_0(\mathbb{P}^2)$, we write now
\begin{equation*}
({d'}^v_{-1},{d'}^v_0,{d'}^v_1)={d'}^v:=\dimvec\psi'(v)
\end{equation*}
for the coordinates with respect to the basis of simple representations $S(-1),S(0),S(1)$; these are mapped to the objects $\Omhol_{\mathbb{P}^2}^2(2)[1]\simeq\mathcal{O}_{\mathbb{P}^2}(-1)[1],\Omhol_{\mathbb{P}^2}^1(1),\mathcal{O}_{\mathbb{P}^2}[-1]$, for which the triple $(\rk,\deg,\chi)$ is equal to $(-1,1,0),$ $(2,-1,0)$ and $(-1,0,-1)$ respectively. This gives the linear transformations
\begin{equation}\label{Eq: trans coord K0(P2), 2}
\begin{pmatrix}
{d'}^v_{-1}\\
{d'}^v_0\\
{d'}^v_1
\end{pmatrix}=
\begin{pmatrix}
1 &2 &-1\\
1 &1 &-1\\
0& 0& -1
\end{pmatrix}
\begin{pmatrix}
\rk v\\
\deg v\\
\chi(v)
\end{pmatrix}\,,\quad
\begin{pmatrix}
\rk v\\
\deg v\\
\chi(v)
\end{pmatrix}=
\begin{pmatrix}
-1& 2& -1\\
1 &-1 &0\\
0 &0 &-1
\end{pmatrix}
\begin{pmatrix}
{d'}^v_{-1}\\
{d'}^v_0\\
{d'}^v_1
\end{pmatrix}\,.
\end{equation}

For $v\in K_0(\mathbb{P}^2)$ define now $\theta'_{\G,v}=t\theta'_{\M,v}+\theta'_{\chi,v}$ by
\begin{equation*}
\nu_{\M,v}(w)=\sigma_{\M}(v,w)=\theta'_{\M,v}\cdot{d'}^w\,,\quad
\nu_{\chi,v}(w)=\sigma_\chi(v,w)=\theta'_{\chi,v}\cdot{d'}^w\,,
\end{equation*}
where the new arrays $\theta'_{\M,v},\theta'_{\chi,v}\in\mathbb{Z}^{\{-1,0,1\}}$ are given by
\begin{equation}\label{Eq: thetas on P2, 2}
\begin{array}{c}
\theta'_{\M,v}=
\begin{pmatrix}
-\rk v-\deg v\\
2\deg v+\rk v\\
-\deg v
\end{pmatrix}=
\begin{pmatrix}
-{d'}^v_0+{d'}^v_1\\
{d'}^v_{-1}-{d'}^v_1\\
-{d'}^v_{-1}+{d'}^v_0
\end{pmatrix}\,,\\
\theta'_{\chi,v}=
\begin{pmatrix}
-\chi(v)\\
2\chi(v)\\
\rk v-\chi(v)
\end{pmatrix}=
\begin{pmatrix}
{d'}^v_1\\
-2{d'}^v_1\\
-{d'}^v_{-1}+2{d'}^v_0
\end{pmatrix}\,.
\end{array}
\end{equation}
The regions of interest are now
\begin{equation*}
\begin{array}{c}
\mathpzc{R}_H{}'=\mathpzc{S}_H^\circ{}'=\{v\in K_0(\mathbb{P}^2)\ |\ 0<-\deg v<\rk v\}=\{v\in K_0(\mathbb{P}^2)\ |\ {d'}^v_0>{d'}^v_{-1}\textup{ and }{d'}^v_0>{d'}^v_1\}\,,\\
\mathpzc{R}^\textup{G}_H{}'=\{v\in K_0(\mathbb{P}^2)\ |\ -t\rk v<t\deg v+\chi(v)<\rk v\}\,,\\
\mathpzc{S}_H{}'=\{v\in K_0(\mathbb{P}^2)\ |\ -(t+1)\rk v<t\deg v+\chi(v)<\rk v\,,\ \ \deg v\neq-\rk v\}\,.
\end{array}
\end{equation*}
To find these expressions for $\mathpzc{S}_H^\circ{}'$ and $\mathpzc{S}_H{}'$ we observe that for any $x\in\mathbb{P}^2$ we can take a section $s\in H^0(\Omhol_{\mathbb{P}^2}^1(1))$ whose zero locus is $x$, and define the Kronecker complex
\begin{equation*}
K'_x:\Omhol_{\mathbb{P}^2}^2(2)\overset{\iota_s}{\longrightarrow}\Omhol_{\mathbb{P}^2}^1(1)\overset{\iota_s}{\longrightarrow}\mathcal{O}_{\mathbb{P}^2}\,,
\end{equation*}
so that $K'_x\simeq\mathcal{O}_x[-1]$. Then we notice that the nontrivial quotients $Q$ of $K'_x$ have dimensions $(1,1,0)$ and $(1,0,0)$, and only for the latter we have $H^{-1}_\mathcal{C}(Q)\neq 0$.

This time the cone $\mathpzc{R}_H{}'$ is not wide enough to describe all moduli spaces for positive rank: if a torsion-free sheaf $\mathcal{F}$ has $\mu_H(\mathcal{F})\in\mathbb{Z}$, then no twist of it is in $\mathpzc{R}_H{}'$. However, if $\mathcal{F}$ is non-trivial and Gieseker-semistable, then it has a twist $\mathcal{F}(k)$ of zero slope and $\chi(\mathcal{F}(k))=\rk\mathcal{F}+3\deg_H\mathcal{F}/2+\ch_2\mathcal{F}<\rk\mathcal{F}$ (because $\ch_2\mathcal{F}<0$, as observed just before Prop.\ \ref{Prop: Bogom ineq from quiv mod}), so that it is contained in
\begin{equation*}
\tilde{\mathpzc{R}}_{H,\mathfrak{E}'}=\mathpzc{R}^\textup{G}_H{}'\cap\mathpzc{S}_H{}'=\{v\in K_0(\mathbb{P}^2)\ |\ -t\rk v<t\deg v+\chi(v)<\rk v\,,\ \ \deg v\neq-\rk v\}\,.
\end{equation*}

So we can apply Corollary \ref{Cor: isom mod st/sp} to the collection $\mathfrak{E}'$:

\begin{theorem}\label{Thm: isom sh=qui on P2 (2)}
Let $v\in\tilde{\mathpzc{R}}_{H,\mathfrak{E}'}$. Then we have isomorphisms
\begin{equation*}
\M^\textup{ss}_{\mathbb{P}^2,H}(v)\simeq\M^\textup{ss}_{B_3,J',\theta'_{\G,v}}({d'}^v)\quad\textup{and}\quad 
\M^\textup{st}_{\mathbb{P}^2,H}(v)\simeq\M^\textup{st}_{B_3,J',\theta'_{\G,v}}({d'}^v)\,.
\end{equation*}
\end{theorem}

Moreover, remarks analogous to those at the end of the previous subsection apply to this situation, and similarly we also deduce an equivalence of abelian categories of semistable objects:

\begin{theorem}
If $p\in\mathbb{R}[t]$ is the Hilbert polynomial of a class $v\in\tilde{\mathpzc{R}}_{H,\mathfrak{E}'}$, then $\Psi$ restricts to an equivalence between the abelian categories $\mathcal{S}_H(p)$ and $\mathcal{S}_{\theta'_{\G,v}}$ (defined in Equations \eqref{Eq: cat semis coh shvs} and \eqref{Eq: cat semis quiv rep}).
\end{theorem}

\subsection{Examples}
Now we will see some examples in which $\M^\textup{ss}_{\mathbb{P}^2,H}(v)$ can be determined more or less explicitly using the isomorphisms of Theorems \ref{Thm: isom sh=qui on P2 (1)} and \ref{Thm: isom sh=qui on P2 (2)}.

The first observation is that we can choose $v\in K_0(\mathbb{P}^2)$ so that at least one of the invariants $d^v_{-1},d^v_1,{d'}^v_{-1},{d'}^v_1$ vanish (and via equations \eqref{Eq: trans coord K0(P2)} and \eqref{Eq: trans coord K0(P2), 2} each of these conditions turns into a linear relation on $\rk v$, $\deg v$ and $\chi(v)$). In these cases, the representations of $B_3$ under consideration reduce to representations of the Kronecker quiver $K_3$, the relations $J$ and $J'$ are trivially satisfied and in any case the stability conditions reduce to the standard one for Kronecker modules. This means that $\M^\textup{ss}_{\mathbb{P}^2,H}(v)$ is isomorphic to some Kronecker moduli space $K(3;m,n)$, for which we can use the properties listed in Ex.\ \ref{Ex: Kron modules}. More precisely, as special cases of Theorems \ref{Thm: isom sh=qui on P2 (1)} and \ref{Thm: isom sh=qui on P2 (2)} we have:

\begin{corollary}\label{Cor: isom mod sh P2 Kron mod}
First, let $v\in\mathpzc{R}_H$:
	\begin{itemize}
	\item[a)] if $d_{-1}^v=\rk v-2\deg v-\chi(v)=0$, then $\M^\textup{ss}_{\mathbb{P}^2,H}(v)\simeq K(3;d_0^v,d_1^v)$;
	\item[b)] if $d_1^v=\rk v+\deg v+\chi(v)=0$, then $\M^\textup{ss}_{\mathbb{P}^2,H}(v)\simeq K(3;d_{-1}^v,d_0^v)$.
	\end{itemize}
Now let $v\in\tilde{\mathpzc{R}}_{H,\mathfrak{E}'}$:
\begin{itemize}
	\item[c)] if ${d'}_{-1}^v=\rk v+2\deg v-\chi(v)=0$, then $\M^\textup{ss}_{\mathbb{P}^2,H}(v)\simeq K(3;{d'}_0^v,{d'}_1^v)$;
	\item[d)] if ${d'}_1^v=-\chi(v)=0$, then $\M^\textup{ss}_{\mathbb{P}^2,H}(v)\simeq K(3;{d'}_{-1}^v,{d'}_0^v)$.
	\end{itemize}
Similar isomorphisms hold for the stable loci.
\end{corollary}

Recall also that twisting by $\mathcal{O}_{\mathbb{P}^2}(1)$ gives isomorphic moduli spaces. In the examples we will only consider classes $v$ normalized as before, that is belonging to the regions $\mathpzc{R}_H$ or $\mathpzc{R}^\textup{G}_H{}'\cap\mathpzc{S}_H{}'$.

Since we have an isomorphism $K_0(\mathbb{P}^2)\overset{(\rk,\deg,\chi)}{\longrightarrow}\mathbb{Z}^3$, we will often use the notation
\begin{equation*}
\M^\textup{ss}_{\mathbb{P}^2,H}(\rk v,\deg v,\chi(v))\quad\textup{and}\quad
\M^\textup{st}_{\mathbb{P}^2,H}(\rk v,\deg v,\chi(v))
\end{equation*}
to indicate $\M^\textup{ss}_{\mathbb{P}^2,H}(v)$ and $\M^\textup{st}_{\mathbb{P}^2,H}(v)$.

\begin{examples}\ 
\begin{enumerate}
\item Let $r>0$ and $(d^v_{-1},d^v_0,d^v_1)=(0,r,0)$, so that $(\rk v,\deg v,\chi(v))=(r,0,r)$. For this choice there is a unique representation of $B_3$, which is always semistable, and stable only for $r=1$. $\M^\textup{ss}_{\mathbb{P}^2,H}(r,0,r)$ is a point, and $\M^\textup{st}_{\mathbb{P}^2,H}(r,0,r)$ is a point for $r=1$ and empty for $r>1$. So the only Gieseker-semistable sheaf with these invariants is the trivial bundle $\mathcal{O}_{\mathbb{P}^2}^{\oplus r}$.
\item Let $m>0$ and $({d'}^v_{-1},{d'}^v_0,{d'}^v_1)=(0,m,0)$, so that $(\rk v,\deg v,\chi(v))=(2m,-m,0)$. Again, $\M^\textup{ss}_{\mathbb{P}^2,H}(2m,-m,0)$ is a point, and $\M^\textup{st}_{\mathbb{P}^2,H}(2m,-m,0)$ is a point for $m=1$ and empty for $m>1$: the only Gieseker-semistable sheaf with these invariants is $\Omhol_{\mathbb{P}^2}^1(1)^{\oplus m}$.
\item Let $m$ be a positive integer. We have $\M^\textup{ss}_{\mathbb{P}^2,H}(5m,-2m,0)\simeq K(3;m,3m)\simeq\point$, with empty stable locus for $m>1$: the only Gieseker-semistable sheaf with these invariants is the right mutation $R_{\Omhol_{\mathbb{P}^2}^1(1)}\mathcal{O}_{\mathbb{P}^2}(-1)^{\oplus m}[1]$.
\item $\M^\textup{ss}_{\mathbb{P}^2,H}(2,0,0)\simeq K(3;2,2)\simeq\mathbb{P}^5$, having used Theorem \ref{Thm: isom sh=qui on P2 (2)} and Example \ref{Ex: Kron modules} (see also \cite[Ch.2, \S4.3]{OSS80Ve} for a sheaf-theoretical proof of this isomorphism).
\item Since $\Pic^0(\mathbb{P}^2)$ is trivial, by sending a 0-dimensional subscheme $Z\subset X$ of length $\ell$ to its ideal sheaf $\mathcal{I}_Z\subset\mathcal{O}_X$ we get an isomorphism $\Hilb^\ell(\mathbb{P}^2)\simeq\M^\textup{ss}_{\mathbb{P}^2,H}(1,0,1-\ell)$, where $\Hilb^\ell(\mathbb{P}^2)$ is the Hilbert scheme of $\ell$ points in $\mathbb{P}^2$. In particular, $\Hilb^1(\mathbb{P}^2)\simeq\mathbb{P}^2$ must be isomorphic (by using the above formulas to compute $d^v,\theta_{\G,v},{d'}^v,$ and $\theta'_{\G,v}$) to the moduli spaces $\M^\textup{ss}_{B_3,J,(-t,-1,t+3)}(1,3,1)$ and $\M^\textup{ss}_{B_3,J',(-t,t,1)}(1,1,0)$, and also to their stable loci.

We can obtain these isomorphisms directly from the representation theory of $B_3$: for the second isomorphism we just observe that
\begin{equation*}
\M^\textup{ss}_{B_3,J',(-t,t,1)}(1,1,0)=K(3;1,1)=K_\textup{st}(3;1,1)\simeq\G_1(3)=\mathbb{P}^2\,.
\end{equation*}
To see the isomorphism $\M^\textup{ss}_{B_3,J,(-t,-1,t+3)}(1,3,1)\simeq\mathbb{P}^2$, first note that
\begin{equation*}
\theta_{\G,v}=(-t,-1,t+3)\quad\textup{and}\quad \tilde\theta=(-1,-1,4)
\end{equation*}
are equivalent by looking at the walls in $(1,3,1)^\perp$ (see Figure \ref{Fig: ChB3}). Then by the symmetry $B_3\simeq B_3^{\op}$ we also see that $\M^\textup{ss}_{B_3,J,(-1,-1,4)}(1,3,1)\simeq\M^\textup{ss}_{B_3,J,(-4,1,1)}(1,3,1)$. So we are interested in understanding $(-4,1,1)$-stability for representations\\
\begin{center}\begin{tikzcd}
&\mathbb{C}\arrow[bend left=50]{r}{a_1}\arrow{r}{a_2}\arrow[bend right=50]{r}{a_3} &\mathbb{C}^3\arrow[bend left=50]{r}{b_1}\arrow{r}{b_2}\arrow[bend right=50]{r}{b_3} &\mathbb{C}
\end{tikzcd}\,.\end{center}
We also write $a=(a_1,a_2,a_3),b^t=(b_1^t,b_2^t,b_3^t)\in\M_3(\mathbb{C})$. Such a representation $(a,b)$ is $(-4,1,1)$-unstable if and only if it admits a subrepresentation of dimension $(1,2,1)$ or $(1,w_0,0)$ for some $w_0\in\{0,1,2,3\}$, and this happens if and only if $\rk a\leq 2$ or there is a $w_0$-dimensional subspace $W_0\subset\mathbb{C}^3$ such that $\im a\subset W_0\subset\ker b$.\\
Hence (note also that $(1,3,1)$ is $(-4,1,1)$-coprime) the $(-4,1,1)$-(semi)stable locus in $R:=R_{\mathbb{C}\oplus\mathbb{C}^3\oplus\mathbb{C}}(B_3)\cong\M_3(\mathbb{C})^{\oplus 2}$ is
\begin{equation*}
R^\textup{ss}=R^\textup{st}=\{(a,b)\in\M_3(\mathbb{C})^{\oplus 2}\ |\ \rk a=3\textup{ and }b\neq 0\}\,.
\end{equation*}
The map $R\to\M_3(\mathbb{C})$ given by $(a,b)\mapsto ba=(b_ja_i)_{i,j=1,2,3}$ descends to an isomorphism
\begin{equation}\label{Eq: isom MB3=P8}
\M^\textup{ss}_{B_3,(-4,1,1)}(1,3,1)=R^\textup{ss}/PG_{(1,3,1)}\to\mathbb{P}(\M_3(\mathbb{C}))\simeq\mathbb{P}^8\,.
\end{equation}
Finally, the relations $J$ cut down the subvariety $X_J=\{(a,b)\in R\ |\ a_ib_j+a_jb_i=0\}$, thus the previous isomorphism restricts to
\begin{equation*}
\M^\textup{ss}_{B_3,J,(-4,1,1)}(1,3,1)=(X_J\cap R^\textup{ss})/PG_{(1,3,1)}\simeq\mathbb{P}(\Ant_3(\mathbb{C}))\simeq\mathbb{P}^2\,,
\end{equation*}
where $\Ant_3(\mathbb{C})\subset\M_3(\mathbb{C})$ is the subspace of antisymmetric matrices.
\item For $({d'}^v_{-1},{d'}^v_0,{d'}^v_1)=(1,3,1)$ we have $\theta'_{\G,v}=(-2t+1,-2,2t+5)$, which is also equivalent to $\tilde\theta=(-1,-1,4)$ (see Figure \ref{Fig: ChB3}). Imposing the symmetric relations $J'$ instead, the isomorphism \eqref{Eq: isom MB3=P8} restricts to $\M^\textup{ss}_{B_3,J',(-4,1,1)}(1,3,1)\simeq\mathbb{P}(\Sym_3(\mathbb{C}))\simeq\mathbb{P}^5$, where $\Sym_3(\mathbb{C})\subset\M_3(\mathbb{C})$ is the subspace of symmetric matrices. Hence
\begin{equation*}
\M^\textup{ss}_{\mathbb{P}^2,H}(4,-5,1)\simeq\M^\textup{ss}_{B_3,J',(-2t+1,-2,2t+5)}(1,3,1)\simeq\mathbb{P}^5\,.
\end{equation*}
\end{enumerate}\end{examples}

\begin{figure}
\centering
\includegraphics[height=7cm]{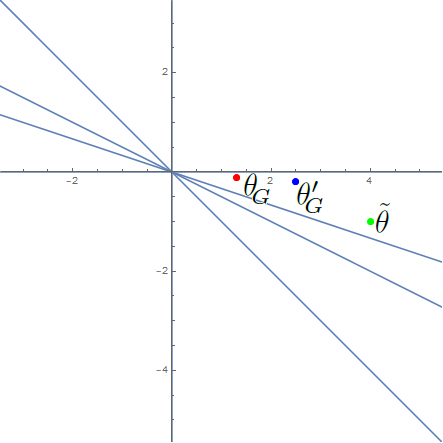}
\caption{The plane $(1,3,1)^\perp$ in $K_0(B_3)\simeq\mathbb{Z}^3$, represented with respect to the basis $\{(-1,0,1),(-3,1,0)\}$. The lines are the numerical walls, while the dots are the points $\theta_{G,(1,3,1)}=(-1,0,1)+\epsilon(3,-2,3)$, $\theta_{G,(1,3,1)}'=(-2,0,2)+\epsilon(1,-2,5)$ and $\tilde\theta:=(-1,-1,4)$, for $\epsilon=0.1$.}
\label{Fig: ChB3}
\end{figure}

%% file: P1xP1.tex

%
%
\section{Application to \texorpdfstring{$\mathbb{P}^1\times\mathbb{P}^1$}{P1xP1}}\label{Appl to P1xP1}
Let $Z$ be a 2-dimensional $\mathbb{C}$-vector space and set $X:=\mathbb{P}_\mathbb{C}(Z)\times\mathbb{P}_\mathbb{C}(Z)$.

Recall that $\Pic(X)=\mathbb{Z}H\oplus\mathbb{Z}F$, where $H,F$ are inverse images of a point under the first and second projections $X\to\mathbb{P}^1$ respectively.\\
Take a divisor $A=aH+bF$. $A$ is ample if and only if $a,b$ are both positive, and by the Hirzebruch-Riemann-Roch formula \eqref{Eq: Hilb pol sh sur} we have
\begin{equation*}
P_{v,A}(t)=t^2ab\rk v+t(a\deg_Hv+b\deg_Fv+\rk v(a+b))+\chi(v)
\end{equation*}
and $\chi(v)=P_{v,A}(0)=\rk v+\deg_Hv+\deg_Fv+\ch_2 v$. 

Consider the exceptional collections
\begin{equation*}
\begin{array}{c}
\mathfrak{E}=(E_{(0,-1)},E_{(0,0)},E_{(1,-1)},E_{(1,0)}):=(\mathcal{O}_X(0,-1)[-1],\mathcal{O}_X[-1],\mathcal{O}_X(1,-1),\mathcal{O}_X(1,0))\,,\\
{}^\vee\!\mathfrak{E}=({}^\vee\!E_{(1,0)},{}^\vee\!E_{(1,-1)},{}^\vee\!E_{(0,0)},{}^\vee\!E_{(0,-1)})=
(\mathcal{O}_X(1,0),\mathcal{O}_{\mathbb{P}^1}(1)\boxtimes\uptau_{\mathbb{P}^1}(-1),
\uptau_{\mathbb{P}^1}\boxtimes\mathcal{O}_{\mathbb{P}^1},\uptau_{\mathbb{P}^1}\boxtimes\uptau_{\mathbb{P}^1}(-1))
\end{array}
\end{equation*}
seen in Example \ref{Exs: exc coll}(2) (note that the objects of ${}^\vee\!\mathfrak{E}$ are isomorphic to $\mathcal{O}_X(1,0)$, $\mathcal{O}_X(2,0)$, $\mathcal{O}_X(1,1)$, and $\mathcal{O}_X(2,1)$). We apply Theorem \ref{Thm: Baer-Bondal} to the full strong collection ${}^\vee\!\mathfrak{E}$. We have now the tilting bundle $T:=\oplus_{i\in I}{}^\vee\!E_i$ (here $I=\{(0,-1),(0,0),(1,-1),(1,0)\}$) and its endomorphism algebra
\begin{equation*}
\End_{\mathcal{O}_X}(T)=\begin{pmatrix}
\mathbb{C} &  &  & \\ 
\mathbb{C}\otimes Z & \mathbb{C} &  & \\ 
Z\otimes\mathbb{C} & 0 & \mathbb{C} & \\ 
Z\otimes Z & Z\otimes\mathbb{C} & \mathbb{C}\otimes Z & \mathbb{C}
\end{pmatrix}\,.
\end{equation*}
Choosing a basis $\{e_1,e_2\}$ of $Z$, $\End_{\mathcal{O}_X}(T)$ identifies with the opposite of the bound quiver algebra $\mathbb{C}Q_4/J$, where\\
\begin{center}$Q_4$: \begin{tikzcd}
&(0,0)\arrow[shift left]{dr}{b^1_1}\arrow[shift right,swap]{dr}{b^1_2} &\\
(0,-1)\arrow[shift left]{ur}{a^1_1}\arrow[shift right,swap]{ur}{a^1_2}\arrow[shift left]{dr}{a^2_1}\arrow[shift right,swap]{dr}{a^2_2} &&(1,0)\\
&(1,-1)\arrow[shift left]{ur}{b^2_1}\arrow[shift right,swap]{ur}{b^2_2}&
\end{tikzcd}\end{center}
and $J=(b^1_ia^1_j+b^2_ja^2_i,\ i=1,2)$. So we have again an equivalence
\begin{equation*}
\Psi:=\Phi_{{}^\vee\!\mathfrak{E}}[1]:D^b(X)\to D^b(Q_4;J)
\end{equation*}
which sends a complex $\mathcal{F}^\bullet\in D^b(X)$ to the complex of representations\\
\begin{center}\begin{tikzcd}[column sep=tiny]
&&R\Hom_{\mathcal{O}_X}(\mathcal{O}_X(2,0),\mathcal{F}^\bullet)[1]\arrow[shift left]{dr}\arrow[shift right,swap]{dr} &\\
&R\Hom_{\mathcal{O}_X}(\mathcal{O}_X(2,1),\mathcal{F}^\bullet)[1]\arrow[shift left]{ur}\arrow[shift right,swap]{ur}\arrow[shift left]{dr}\arrow[shift right,swap]{dr}&&R\Hom_{\mathcal{O}_X}(\mathcal{O}_X(1,0),\mathcal{F}^\bullet)[1]\\
&&R\Hom_{\mathcal{O}_X}(\mathcal{O}_X(1,1),\mathcal{F}^\bullet)[1]\arrow[shift left]{ur}\arrow[shift right,swap]{ur}&
\end{tikzcd}\end{center}
and the standard heart in $D^b(Q_4;J)$ corresponds to the heart
\begin{equation*}
\mathcal{K}:=\langle\mathcal{O}_X(0,-1)[1],\mathcal{O}_X,\mathcal{O}_X(1,-1),\mathcal{O}_X(1,0)[-1]\rangle_{\ext}\,,
\end{equation*}
whose objects are \emph{Kronecker complexes}
\begin{equation*}
K_V:V_{0,-1}\otimes\mathcal{O}_X(0,-1)\to V_{0,0}\otimes\mathcal{O}_X\oplus V_{1,-1}\otimes\mathcal{O}_X(1,-1)\to V_{1,0}\otimes\mathcal{O}_X(1,0)
\end{equation*}
with the middle bundle in degree $0$. Also in this case we see immediately that $\mathfrak{E}$ is always monad-friendly with respect to $A$ (Def.\ \ref{Defn: mod-fr exc coll}).

Let $\psi:K_0(X)\to K_0(Q_4;J)$ be the isomorphism induced by the equivalence $\Psi$; we have coordinates on these Grothendieck groups given by the isomorphisms
\begin{equation*}
K_0(X)\overset{(\rk,\deg_H,\deg_F,\chi)}{\longrightarrow}\mathbb{Z}^4\,,\quad
K_0(Q_4;J)\overset{\dimvec}{\longrightarrow}\mathbb{Z}^4\,,
\end{equation*}
and as usual we write
\begin{equation*}
(d^v_{0,-1},d^v_{0,0},d^v_{1,-1},d^v_{1,0})=d^v:=\dimvec\psi(v)
\end{equation*}
for the coordinates of $\psi(v)\in K_0(Q_4;J)$ with respect to the basis of simple representations $S(i)$, where $i\in I=\{(0,-1),(0,0),(1,-1),(1,0)\}$; these are mapped to the objects $\mathcal{O}_X(0,-1)[1]$, $\mathcal{O}_X$, $\mathcal{O}_X(1,-1)$, and $\mathcal{O}_X(1,0)[-1]$, so we find the transformations
\begin{equation*}
\begin{pmatrix}
d^v_{0,-1}\\
d^v_{0,0}\\
d^v_{1,-1}\\
d^v_{1,0}
\end{pmatrix}=
\begin{pmatrix}
1 &1 &2 &-1\\
2 &0 &2 &-1\\
1 &1 &1 &-1\\
1 &0 &1 &-1
\end{pmatrix}
\begin{pmatrix}
\rk v\\
\deg_H v\\
\deg_F v\\
\chi(v)
\end{pmatrix}\,,\quad
\begin{pmatrix}
\rk v\\
\deg_H v\\
\deg_F v\\
\chi(v)
\end{pmatrix}=
\begin{pmatrix}
-1& 1& 1& -1\\
0 &0 &1 &-1\\
1 &0 &-1 &0\\
0 &1 &0 &-2
\end{pmatrix}
\begin{pmatrix}
d^v_{0,-1}\\
d^v_{0,0}\\
d^v_{1,-1}\\
d^v_{1,0}
\end{pmatrix}\,.
\end{equation*}

The arrays $\theta_{\M,v},\theta_{\chi,v}\in\mathbb{Z}^I$ of equation \eqref{Eq: thetaM,thetachi} are given by
\begin{equation*}
\theta_{\M,v}=
\begin{pmatrix}
-\deg_Av-b\rk v\\
\deg_Av\\
\deg_Av-(a-b)\rk v\\
-\deg_Av+a\rk v
\end{pmatrix}\,,\quad
\theta_{\chi,v}=
\begin{pmatrix}
-\chi(v)\\
-\rk v+\chi(v)\\
\chi(v)\\
2\rk v-\chi(v)
\end{pmatrix}\,.
\end{equation*}

The region $\mathpzc{R}_A\subset K_0(X)$ of eq.\ \eqref{Eq: sets RA,RGA} reads
\begin{equation*}
\begin{array}{c}
\mathpzc{R}_A=\{v\in K_0(X)\ |\ \rk v>0,\ -b\rk v<\deg_A(v)<a\rk v\}\,.\\
\end{array}
\end{equation*}
Given $x=([z_1],[z_2])\in X$, take $p_1,p_2\in Z^\vee$ vanishing on $z_1$ and $z_2$ respectively. We have $K_x\simeq\mathcal{O}_x[-1]$, where now
\begin{equation*}
K_x:\mathcal{O}_X(0,-1)\overset{\binom{p_1}{p_2}}{\to}\mathcal{O}_X\oplus\mathcal{O}_X(1,-1)\overset{\binom{p_2}{-p_1}}{\to}\mathcal{O}_X(1,0)\,.
\end{equation*}
Then we can check that $\mathpzc{S}^\circ_A=\mathpzc{R}_A$, and twisting by line bundles we can bring any sheaf of positive rank inside this region. Hence Corollary \ref{Cor: isom mod st/sp} describes again all moduli spaces of semistable sheaves of positive rank:

\begin{theorem}
Let $v\in\mathpzc{R}_A$. We have isomorphisms
\begin{equation*}
\M^\textup{ss}_{X,A}(v)\simeq\M^\textup{ss}_{Q_4,J,\theta_{\G,v}}(d^v)\quad\textup{and}\quad
\M^\textup{st}_{X,A}(v)\simeq\M^\textup{st}_{Q_4,J,\theta_{\G,v}}(d^v)\,.
\end{equation*}
\end{theorem}

Like after Theorem \ref{Thm: isom sh=qui on P2 (1)}, we have some immediate remarks:
	\begin{enumerate}
	\item If for $v\in\mathpzc{R}_A$ the dimension vector $d^v$ is $\theta_{\G,v}$-coprime, then $\gcd(\rk v,\deg_Av,\chi(v))=1$. In this case $\M^\textup{ss}_{X,A}(v)=\M^\textup{st}_{X,A}(v)$ and there is a universal family (by Remarks \ref{Rmk: prop mod sp sh surf} and \ref{Rmk: prop quiv mod}).
	\item As we observed for $\mathbb{P}^2$, also in this case $\M^\textup{st}_{X,A}(v)\simeq\M^\textup{st}_{Q_4,J,\theta_{\G,v}}(d^v)$ is smooth and its dimension is $\dim\M^\textup{st}_{Q_4,\theta_{\G,v}}(d^v)$ minus the number $4d^v_{0,-1}d^v_{1,0}$ of relations imposed, which gives
	\begin{equation*}
	\dim\M^\textup{st}_{X,A}(v)=1-\rk v^2+\Delta(v)\,,
	\end{equation*}
	in agreement with eq.\ \eqref{Eq: dim ms sh sur}.
	\item For all $v\in\mathpzc{R}_A$ we have $\theta_{\M,v}^{(0,-1)}=-b\rk v-\deg_Av<0$ and $\theta_{\M,v}^{(1,0)}=a\rk v-\deg_A v>0$.
	\item Notice that in this case there may be semistable Kronecker complexes in classes $w\in K_0(X)$ such that $P_{w,A}=0$, so we do not have an analogue of Theorem \ref{Thm: equiv ab cat semist sh and quiv rep P2}.
	\end{enumerate}

\begin{examples} We use the notation $\M^{\textup{ss}/\textup{st}}_{X,A}(\rk v,\deg_Hv,\deg_Fv,\chi(v)):=\M^{\textup{ss}/\textup{st}}_{X,A}(v)$.
\begin{enumerate}
\item Let $r$ be a positive integer. Taking $d^v=(0,r,0,0)$ we get $\M^\textup{ss}_{X,A}(r,0,0,r)=\{\mathcal{O}_X^{\oplus r}\}$, while for $d^v=(0,0,r,0)$ we find $\M^\textup{ss}_{X,A}(r,r,-r,0)=\{\mathcal{O}_X(1,-1)^{\oplus r}\}$.
\item Let $\ell$ be a positive integer. The choice $d^v=(\ell,\ell+1,\ell,\ell)$ gives the Hilbert scheme of points:
\begin{equation*}
\Hilb^\ell(X)=\M^\textup{ss}_{X,A}(v)\simeq\M^\textup{ss}_{Q_4,J,\theta_{\G,v}}(\ell,\ell+1,\ell,\ell)\,,
\end{equation*}
where $\theta_{\G,v}=(-tb+(\ell-1),-\ell,t(b-a)+(1-\ell),ta+(\ell+1))$.
\item If we choose $v\in K_0(X)$ with at least one between $d^v_{0,-1}$ and $d^v_{1,0}$ vanishing, then the representations we are considering reduce to representations of the quivers\\
\begin{tikzcd}
&&(0,0)\arrow[shift left]{dr}{b^1_1}\arrow[shift right,swap]{dr}{b^1_2} &\\
&&&(1,0)\\
&&(1,-1)\arrow[shift left]{ur}{b^2_1}\arrow[shift right,swap]{ur}{b^2_2}&
\end{tikzcd}\,,
\begin{tikzcd}
&&(0,0)&\\
&(0,-1)\arrow[shift left]{ur}{a^1_1}\arrow[shift right,swap]{ur}{a^1_2}\arrow[shift left]{dr}{a^2_1}\arrow[shift right,swap]{dr}{a^2_2} &&\\
&&(1,-1)&
\end{tikzcd}\\
respectively, and the relations $J$ are trivially satisfied. These are the cases considered in \cite{Kule97OnMod}.
\end{enumerate}\end{examples}

%% file: MSSQM.bbl
\begin{thebibliography}{ADHM78}

\bibitem[AB13]{ArcBer13Brid}
Daniele Arcara and Aaron~J. Bertram.
\newblock Bridgeland-stable moduli spaces for {K}-trivial surfaces.
\newblock {\em J. Eur. Math. Soc.}, 15(1):1--38, 2013.

\bibitem[ADHM78]{ADHM78Constr}
Michael~F. Atiyah, Vladimir~G. Drinfeld, Nigel~J. Hitchin, and Yuri~I. Manin.
\newblock Construction of instantons.
\newblock {\em Phys. Lett. A}, 65(3):185--187, 1978.

\bibitem[AM17]{ArcMil17Proj}
Daniele Arcara and Eric Miles.
\newblock Projectivity of {Bridgeland} moduli spaces on {Del Pezzo} surfaces of
  {Picard} rank 2.
\newblock {\em Int. Math. Res. Not. IMRN}, 2017(11):3426--3462, 2017.

\bibitem[ASS06]{AsSiSk06Elem}
Ibrahim Assem, Daniel Simson, and Andrzej Skowronski.
\newblock {\em Elements of the Representation Theory of Associative Algebras:
  Volume 1: Techniques of Representation Theory}, volume~65.
\newblock Cambridge University Press, 2006.

\bibitem[Bar77]{Barth77Mod}
Wolf Barth.
\newblock Moduli of vector bundles on the projective plane.
\newblock {\em Invent. Math.}, 42(1):63--91, 1977.

\bibitem[BBD82]{BeBeDe82Fais}
Alexander~A. Beilinson, Joseph Bernstein, and Pierre Deligne.
\newblock Faisceaux pervers.
\newblock In {\em Analysis and topology on singular spaces, {I} ({L}uminy,
  1981)}, volume 100 of {\em Ast\'erisque}, pages 5--171. Soc. Math. France,
  Paris, 1982.

\bibitem[BCZ17]{BaCrZh17Nef}
Arend Bayer, Alastair Craw, and Ziyu Zhang.
\newblock Nef divisors for moduli spaces of complexes with compact support.
\newblock {\em Selecta Math.}, 23(2):1507--1561, 2017.

\bibitem[Bei78]{Beil78Coh}
Alexander~A. Beilinson.
\newblock Coherent sheaves on {$\mathbb{P}^n$} and problems of linear algebra.
\newblock {\em Funct. Anal. Appl.}, 12(3):214--216, 1978.

\bibitem[Ben98]{Bens98Repr}
D.J. Benson.
\newblock {\em Representations and cohomology. I: Basic representation theory
  of finite groups and associative algebras}, volume~30 of {\em Cambridge
  Studies in Advanced Mathematics}.
\newblock Cambridge U. Press,, 1998.

\bibitem[BH78]{BaHu78Mo}
Wolf Barth and Klaus Hulek.
\newblock Monads and moduli of vector bundles.
\newblock {\em Manuscripta Math.}, 25(4):323--347, 1978.

\bibitem[Bon89]{Bond89Repr}
Aleksei~I. Bondal.
\newblock Representation of associative algebras and coherent sheaves.
\newblock {\em Izv. Ross. Akad. Nauk Ser. Mat.}, 53(1):25--44, 1989.

\bibitem[Bri02]{Brid02Flops}
Tom Bridgeland.
\newblock Flops and derived categories.
\newblock {\em Invent. Math.}, 147(3):613--632, 2002.

\bibitem[Bri07]{Brid07Stab}
Tom Bridgeland.
\newblock Stability conditions on triangulated categories.
\newblock {\em Ann. of Math.}, 166(2):317--345, 2007.

\bibitem[Bri08]{Brid08Stab}
Tom Bridgeland.
\newblock Stability conditions on ${K3}$ surfaces.
\newblock {\em Duke Math. J.}, 141(2):241--291, 2008.

\bibitem[DLP85]{DreLeP85Fibr}
Jean-Marc Dr{\'e}zet and Joseph Le~Potier.
\newblock Fibr{\'e}s stables et fibr{\'e}s exceptionnels sur $\mathbb{P}_2$.
\newblock In {\em Ann. Sci. {\'E}c. Norm. Sup{\'e}r.}, volume~18, pages
  193--243, 1985.

\bibitem[Don84]{Don84Ins}
Simon~K. Donaldson.
\newblock Instantons and geometric invariant theory.
\newblock {\em Comm. Math. Phys.}, 93(4):453--460, 1984.

\bibitem[Dr{\'e}87]{Drez87Fibr}
Jean-Marc Dr{\'e}zet.
\newblock Fibr{\'e}s exceptionnels et vari{\'e}t{\'e}s de modules de faisceaux
  semi-stables sur $\mathbb{P}_2(\mathbb{C})$.
\newblock {\em J. Reine Angew. Math.}, 380:14--58, 1987.

\bibitem[FGIK16]{FiGiIoKu16Int}
Michael Finkelberg, Victor Ginzburg, Andrei Ionov, and Alexander Kuznetsov.
\newblock Intersection cohomology of the {Uhlenbeck} compactification of the
  {Calogero--Moser} space.
\newblock {\em Selecta Math.}, 22(4):2491--2534, 2016.

\bibitem[GK04]{GorKul04Hel}
Alexey~L. Gorodentsev and Sergej~A. Kuleshov.
\newblock Helix theory.
\newblock {\em Mosc. Math. J.}, 4(2):377--440, 2004.

\bibitem[HL10]{HuyLeh10Geo}
Daniel Huybrechts and Manfred Lehn.
\newblock {\em The geometry of moduli spaces of sheaves}.
\newblock Cambridge University Press, 2010.

\bibitem[Hor64]{Horr64Vector}
Geoffrey Horrocks.
\newblock Vector bundles on the punctured spectrum of a local ring.
\newblock {\em Proc. Lond. Math. Soc.}, 3(4):689--713, 1964.

\bibitem[Hul79]{Hulek79Stable}
Klaus Hulek.
\newblock Stable rank-2 vector bundles on $\mathbb{P}^2$ with $c_1$ odd.
\newblock {\em Math. Ann.}, 242(3):241--266, 1979.

\bibitem[Hul80]{Hulek80OnThe}
Klaus Hulek.
\newblock On the classification of stable rank-r vector bundles over the
  projective plane.
\newblock In {\em Vector bundles and differential equations}, pages 113--144.
  Springer, 1980.

\bibitem[Huy06]{Huyb06Four}
Daniel Huybrechts.
\newblock {\em Fourier-Mukai transforms in algebraic geometry}.
\newblock Oxford University Press, 2006.

\bibitem[Kin94]{King94Mod}
Alastair~D. King.
\newblock Moduli of representations of finite dimensional algebras.
\newblock {\em Q. J. Math.}, 45(4):515--530, 1994.

\bibitem[Kul97]{Kule97OnMod}
Sergej~A. Kuleshov.
\newblock On moduli spaces for stable bundles on quadrics.
\newblock {\em Math. Notes}, 62(6):707--725, 1997.

\bibitem[LP94]{LePo94AProp}
Joseph Le~Potier.
\newblock A propos de la construction de l'espace de modules des faisceaux
  semi-stables sur le plan projectif.
\newblock {\em Bull. Soc. Math. France}, 122(3):363--370, 1994.

\bibitem[LP97]{LePo97Lect}
Joseph Le~Potier.
\newblock {\em Lectures on vector bundles}, volume~54.
\newblock Cambridge University Press, 1997.

\bibitem[Neu09]{Neum09Alg}
Frank Neumann.
\newblock {\em Algebraic stacks and moduli of vector bundles}.
\newblock IMPA, 2009.

\bibitem[NS07]{NevSta07Skly}
T.A. Nevins and J.T. Stafford.
\newblock Sklyanin algebras and {Hilbert} schemes of points.
\newblock {\em Adv. Math.}, 210(2):405--478, 2007.

\bibitem[Ohk10]{Ohka10Mod}
Ryo Ohkawa.
\newblock Moduli of {Bridgeland} semistable objects on $\mathbb{P}^2$.
\newblock {\em Kodai Math. J.}, 33(2):329--366, 2010.

\bibitem[OSS80]{OSS80Ve}
Christian Okonek, Michael Schneider, and Heinz Spindler.
\newblock {\em Vector bundles on complex projective spaces}, volume~3.
\newblock Springer, 1980.

\bibitem[Pol07]{Poli07Const}
Alexander Polishchuk.
\newblock Constant families of t-structures on derived categories of coherent
  sheaves.
\newblock {\em Mosc. Math. J.}, 7(1):109--134, 2007.

\bibitem[Rei08]{Rein08Mod}
Markus Reineke.
\newblock Moduli of representations of quivers.
\newblock {\em arXiv preprint arXiv:0802.2147}, 2008.

\bibitem[Rud90]{Rudak90Heli}
Alexei~N. Rudakov.
\newblock {\em Helices and vector bundles: Seminaire Rudakov}, volume 148.
\newblock Cambridge University Press, 1990.

\bibitem[Rud97]{Ruda97Stab}
Alexei~N. Rudakov.
\newblock Stability for an {Abelian} category.
\newblock {\em J. Algebra}, 197(1):231--245, 1997.

\end{thebibliography}
